\pdfoutput=1
\documentclass{amsart}
\usepackage{graphicx, color}
\usepackage{tabularx}
\usepackage{amsmath}
\usepackage{amssymb}

\theoremstyle{plain}
\newtheorem{theorem}{Theorem}

\newtheorem{fact}{Fact}
\newtheorem{lemma}{Lemma}
\newtheorem{corollary}{Corollary}

\theoremstyle{definition}
\newtheorem{definition}{Definition}

\newtheorem{remark}{Remark}

\newcommand{\bigcircle}{\begin{picture}(8,5)
\put(3.5,2.5){\circle{6.5}}
\end{picture}}

\newcommand{\h}{\begin{picture}(8,5)
\put(3.5,2.5){\circle{6.5}}
\qbezier(2,-0.4)(2,2.9)(2,5.4)
\qbezier(0,2.5)(3.4,2.5)(6.8,2.5)
\qbezier(5,-0.4)(5,2.9)(5,5.4)
\end{picture}}

\newcommand{\tr}{\begin{picture}(8,5)
\put(3,2.5){\circle{6.5}}
\qbezier(-0.6,2.5)(2.9,2.5)(6.4,2.5)
\qbezier(1.8,-0.6)(3.3,2.9)(4.8,5.4)
\qbezier(4.5,-0.6)(3,2.9)(1.5,5.4)
\end{picture}
}

\begin{document}
\title[Triple chords and strong (1, 2) homotopy]{Triple chords and strong (1, 2) homotopy}
\author{Noboru Ito}
\author{Yusuke Takimura}
\address{Waseda Institute for Advanced Study, 1-6-1 Nishi-Waseda Shinjuku-ku, Tokyo, 169-8050, Japan}
\email{noboru@moegi.waseda.jp}
\address{Gakushuin Boy's Junior High School, 1-5-1 Mejiro Toshima-ku Tokyo 171-0031 Japan}
\email{Yusuke.Takimura@gakushuin.ac.jp}
\thanks{The work of N. Ito was partly supported by a Program of WIAS in the Program for the Enhancement of Research Universities, a Waseda University Grant for Special Research Projects (Project number: 2014K-6292) and the JSPS Japanese-German Graduate Externship.}
\maketitle
\begin{abstract}
A triple chord $\tr$ is a sub-diagram of a chord diagram that consists of a circle and finitely many chords connecting the preimages for every double point on a spherical curve.  This paper describes some relationships between the number of triple chords and an equivalence relation called strong (1, 2) homotopy, which consists of the first and one kind of the second Reidemeister moves involving inverse self-tangency if the curve is given any orientation.  We show that a prime knot projection is trivialized by strong (1, 2) homotopy, if it is a simple closed curve or a prime knot projection without $1$- and $2$-gons whose chord diagram does not contain any triple chords.  
We also discuss the relation between Shimizu's reductivity and triple chords.  
\end{abstract}


\section{Introduction}
Sakamoto and Taniyama \cite{ST} characterized the sub-chord diagrams $\otimes$ (cross chord) and $\h$ ($H$ chord), embedded in a chord diagram associated with a generic plane curve, where a {\it{chord diagram}} is a circle with the preimages of each double point of the curve connected by a chord.  For example, a chord diagram of a plane curve contains $\h$, if and only if the plane curve is not equivalent to any connected sum of plane curves, each of which is either the simple closed curve $\bigcircle$, the curve that appears similar to $\infty$, or a standard torus knot projection \cite[Theorem 3.2]{ST}.  

This paper aims to obtain a similar characterization of the {\it{triple chord}} $\tr$, stated in Theorem \ref{main1}.  A {\it{knot projection}} is a generic spherical curve that is a regular projection image on $S^{2}$ of a knot.  For a knot projection $P$, a chord diagram $CD_P$ is defined as a circle with the preimages of each double point of the knot projection connected by a chord.  A knot projection is called {\it{prime}}, if it is not the connected sum of two knot projections.  
Let $P^{r}$ be a unique knot projection with no $1$- or $2$-gons obtained by a finite sequence of the first and second Reidemeister moves always decreasing double points in an arbitrary manner for an arbitrary knot projection $P$ (for the uniqueness of $P^r$, see \cite{khovanov, IT}).  
\begin{theorem}\label{main1}
If the chord diagram $CD_P$ of a knot projection $P$ has no triple chord $\tr$, and $P^r$ is a prime knot projection or a simple closed curve, then there exists a finite sequence consisting of local replacements $1a$ and $s2a$ shown in Fig.~\ref{ch11} from a simple closed curve $\bigcircle$ to $P$.  
\begin{figure}[h!]
\includegraphics[width=4cm]{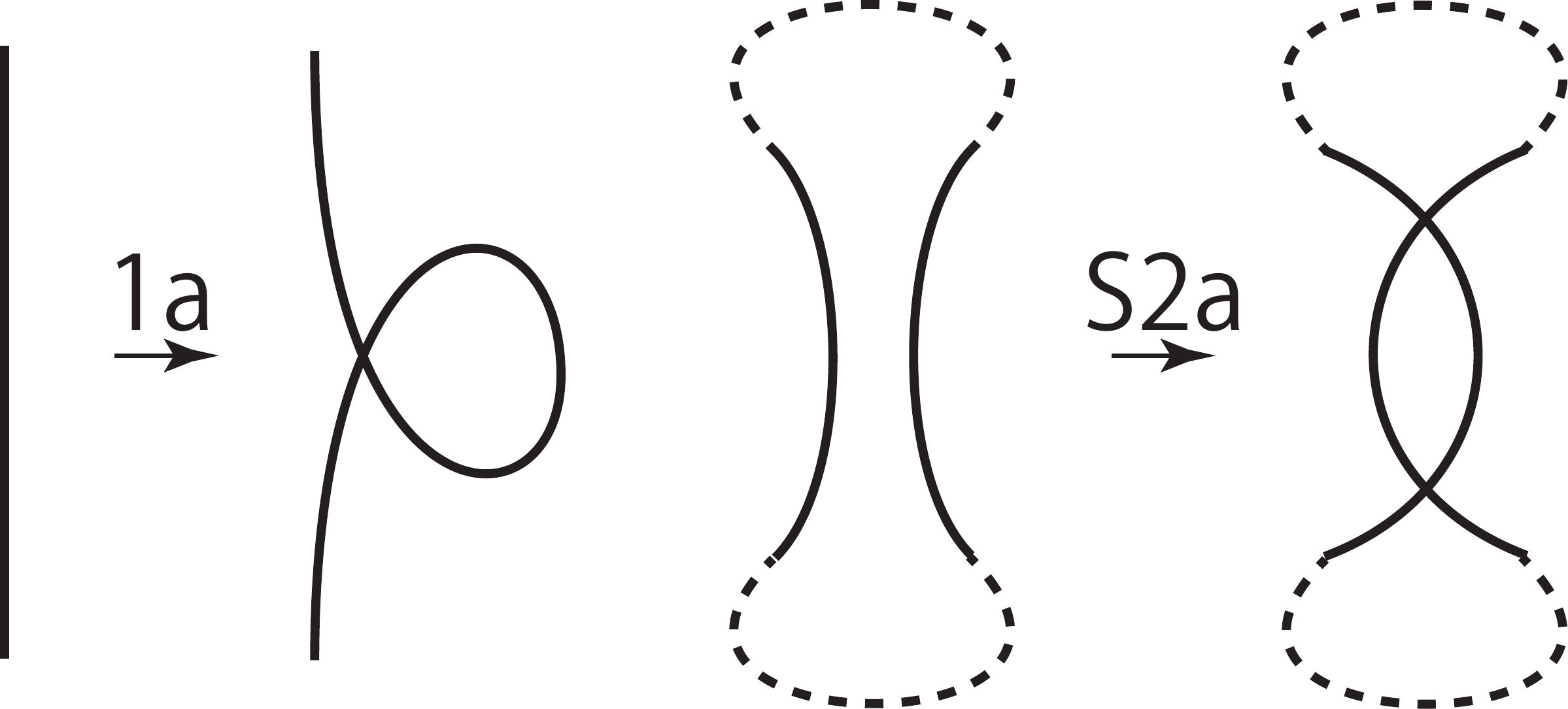}
\caption{Local replacements $1a$ (left) and $s2a$ (right).  Dotted arcs show the connections of non-dotted arcs.}\label{ch11}
\end{figure}
\end{theorem}
We define {\it{strong (1, 2) homotopy equivalence}} as follows: two knot projections $P$ and $P'$ are strong (1, 2) homotopy equivalent, if and only if $P$ is related to $P'$ by a finite sequence consisting of local replacements $1$ and $s2$, as shown in Fig.~\ref{ch12}.  Corollary \ref{cor_main1} from Theorem \ref{main1} helps in understanding the relation between the triple chords and strong (1, 2) homotopy.  
\begin{corollary}\label{cor_main1}
If the chord diagram of an arbitrary prime knot projection $P$ with no $1$- or $2$-gons has no triple chord $\tr$ or $P$ is a simple closed curve, then $P$ is equivalent to a simple closed curve $\bigcircle$ under strong (1, 2) homotopy.  
\begin{figure}[htbp]
\includegraphics[width=4cm]{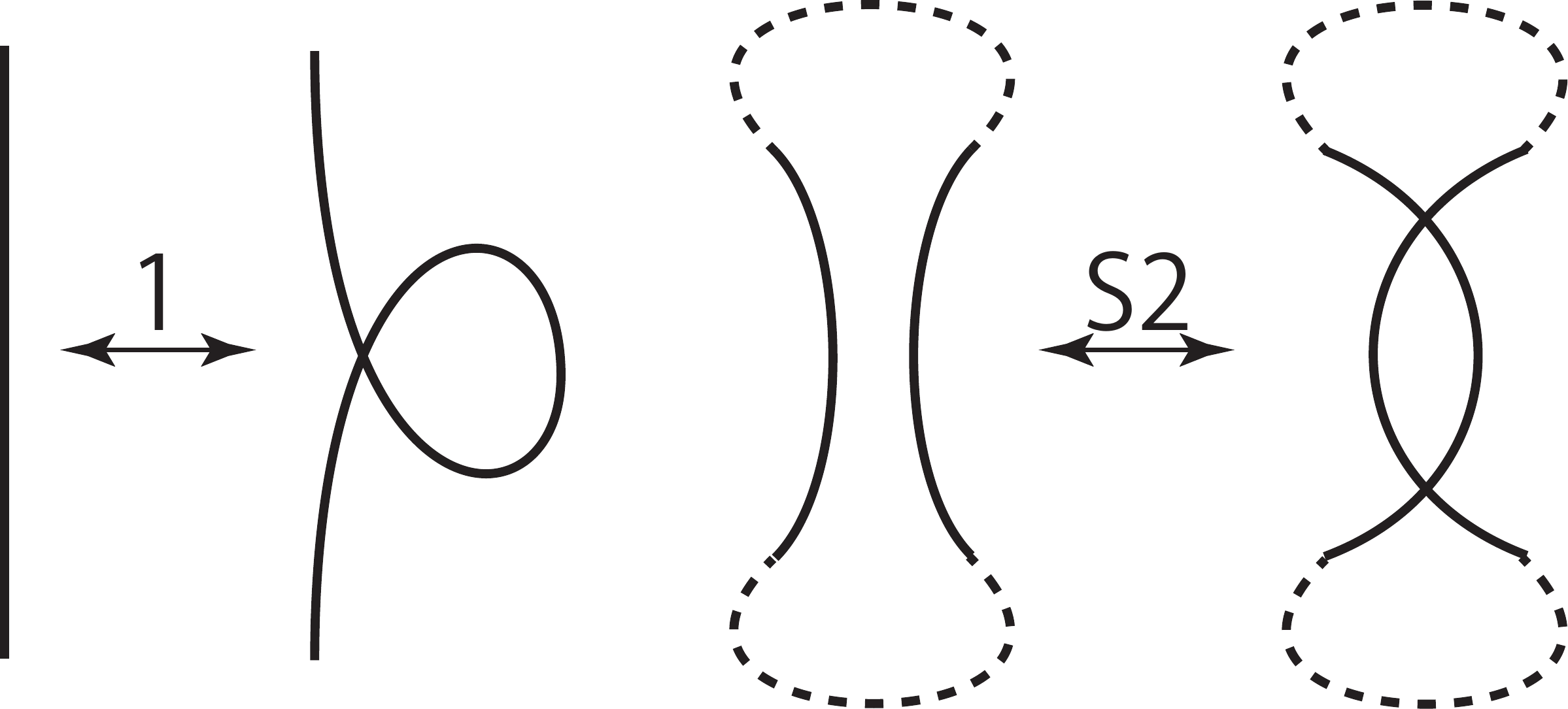}
\caption{Local replacements $1$ (left) and $s2$ (right).}\label{ch12}
\end{figure}
\end{corollary}

The reminder of this paper contains the following sections.  Sec.~\ref{sec1.4} states our conventions.  Sec.~\ref{sec2} and Sec.~\ref{sec_main3} provide proofs of Theorems \ref{main1} and \ref{main3}, respectively.  
Sec.~\ref{sec4.1} mentions a relation between Shimizu's reductivity of knot projections and the triple chord.

\section{Preliminary}\label{sec1.4}
{\it{Reidemeister moves}}, which are three local replacements on an arbitrary knot projection, are defined by Fig.~\ref{tr1}.  It is known that there exists a finite sequence of Reidemeister moves between any two knot projections.  
\begin{figure}[h!]
\includegraphics[width=10cm]{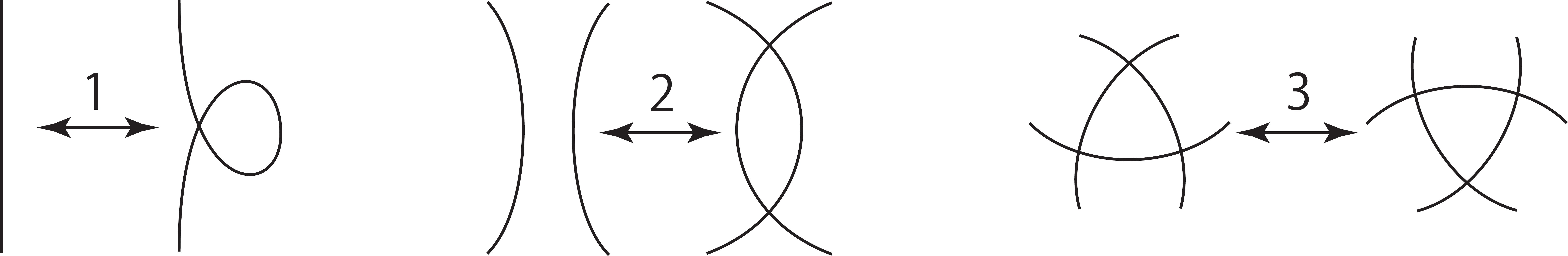}
\caption{First (left), second (center), and third (right) Reidemeister moves.}\label{tr1}
\end{figure}
Shown left to right in Fig.~\ref{tr1} are the first, second, and third Reidemeister moves.  
There are two types of the second Reidemeister moves, local replacement, $s2$, shown in Fig.~\ref{ch12}, and $w2$, shown in Fig.~\ref{tr2}.  
\begin{figure}[h!]
\includegraphics[width=4cm]{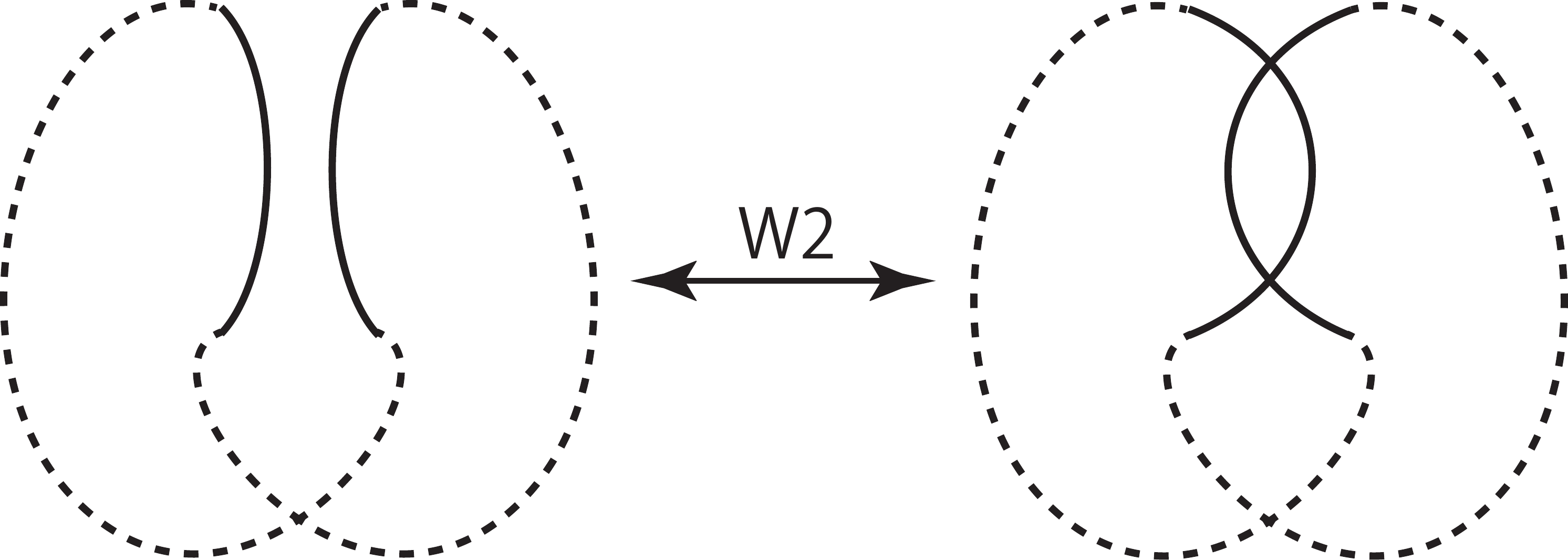}
\caption{Local replacement $w2$.  Dotted arcs show the connections of non-dotted arcs.}\label{tr2}
\end{figure}
Now, we define the notion of {\it{reducible}} and {\it{reduced}} knot projection.  
\begin{definition}[Reducible and reduced knot projection]
A knot projection $P$ is {\it{reducible}}, if there is a double point $d$, called a {\it{reducible crossing}}, in $P$, as shown in Fig.~\ref{ych4}.  If a knot projection is not reducible, it is called a {\it{reduced knot projection}}.  
\begin{figure}[h!]
\includegraphics[width=3cm]{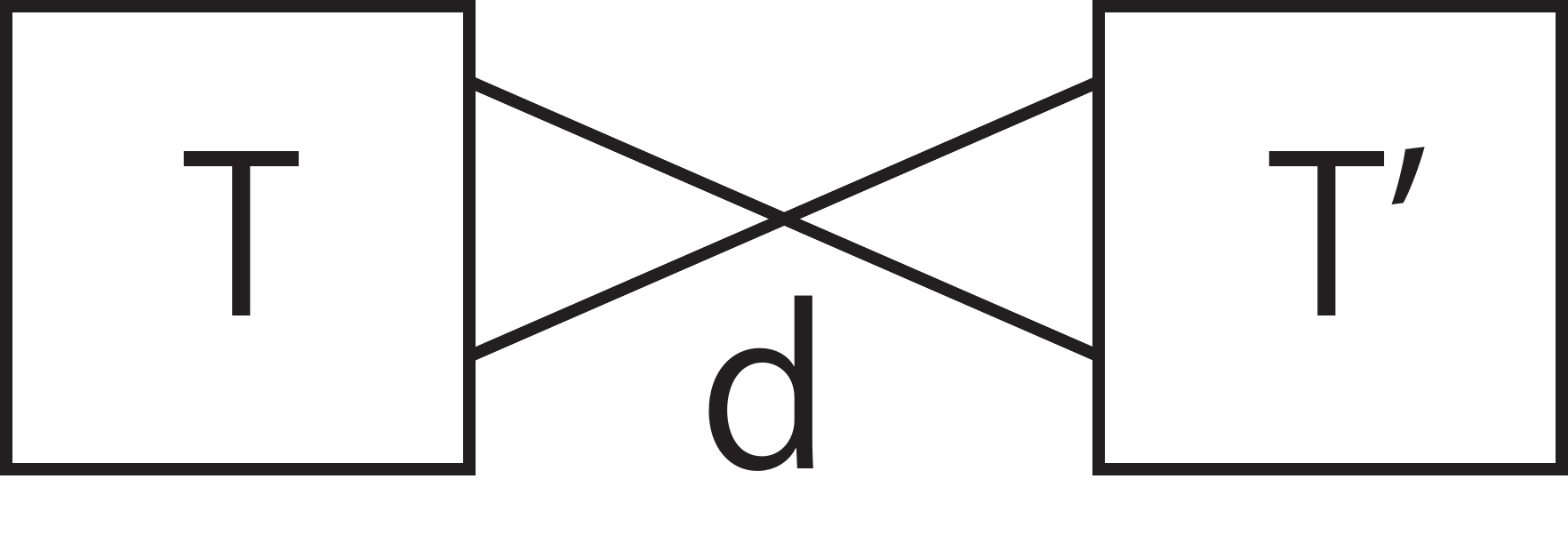}
\caption{Reducible crossing $d$.}\label{ych4}
\end{figure}
\end{definition}
From this definition, we obtain Lemma \ref{lem_reduced}, which is easy to prove and is used often throughout this paper.  
\begin{lemma}\label{lem_reduced}
An arbitrary prime knot projection with no $1$-gons is a reduced knot projection.  
\end{lemma}
\begin{proof}
To establish the claim, it is sufficient to show that ($\star$) if an arbitrary knot projection with no $1$-gons is reducible, then knot projection is non-prime.  We will now show ($\star$).  
Let $P$ be an arbitrary knot projection with no $1$-gons.  Assume that $P$ is reducible.  Then, $P$ can be presented in Fig.~\ref{ych4}.  If the two faces having the point $d$ of $P$, as in Fig.~\ref{ych4}, are not $1$-gons, then $T$ and $T'$ are not simple arcs.  Thus, $P$ is non-prime.  This completes the proof.  
\end{proof}
\section{Proof of Theorem \ref{main1}.}\label{sec2}
To establish Theorem \ref{main1}, we prove Theorem \ref{main3}.  If a knot projection is not a simple closed curve $\bigcircle$, we call it a {\it{non-trivial}} knot projection.  
\begin{theorem}\label{main3}
A prime non-trivial knot projection with no $1$- or $2$-gons contains at least one triple chord.  
\end{theorem}
Now, we deduce Theorem \ref{main1} from Theorem \ref{main3}.  
\begin{proof}
Based on our assumption in Theorem \ref{main1}, a knot projection $P$ has no triple chords in $CD_P$.  
For $P$, we can consider $P^{r}$, the unique knot projection with no $1$- or $2$-gons by a finite sequence consisting of the first and second Reidemeister moves decreasing the number of double points \cite[Theorem 2.2]{khovanov} or \cite[Theorem 2.2]{IT}.  By the assumption of Theorem \ref{main1}, $P^{r}$ is a prime knot projection with no $1$- and $2$-gons or a simple closed curve $\bigcircle$.  Thus, by Theorem \ref{main3}, $P^{r} = \bigcircle$.  

Recover $P$ from $P^r$ using the sequence consisting of $1a$ and $2a$, where $1a$ (resp.~$2a$) is the first (resp.~second) Reidemeister move always increasing a double point (resp.~double points).  If at least one $2a$ in the sequence is $w2$, then there exists at least one triple chord $\tr$ in $CD_P$.  This is because a $2$-gon raised by $w2$ can be represented as shown in Fig.~\ref{tr3}, and the corresponding chord diagram is shown at the left of Fig.~\ref{tr3}.  
\begin{figure}[h!]
\includegraphics[width=5cm]{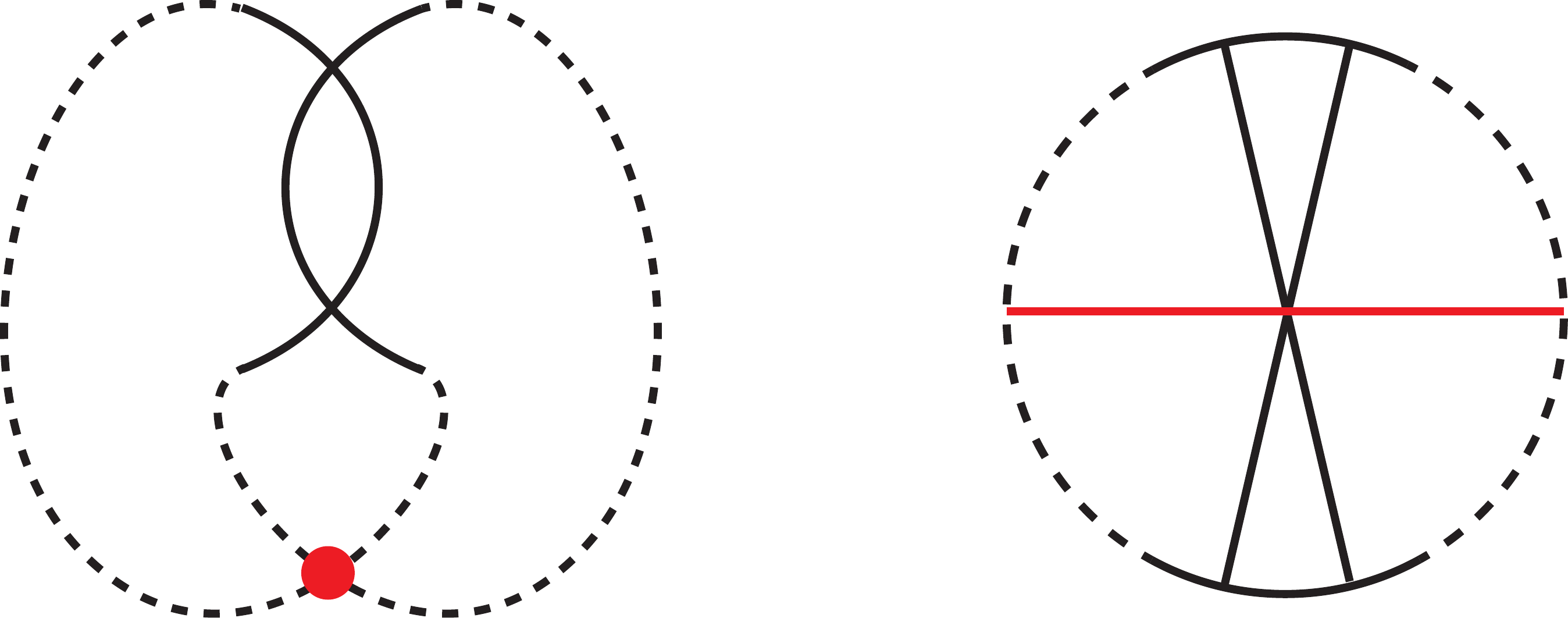}
\caption{$2$-gon appearing in $w2$ (left) and its chord diagram (right).}\label{tr3}
\end{figure}
We can see that the point contained in both dotted arcs exists.  Thus, we can find $\tr$ in $CD_P$.  

However, the existence of $\tr$ contradicts the assumption that $P$ has no triple chords.  Thus, the sequence consisting of $1a$ and $2a$ must consist of $1a$ and $s2a$.  We conclude that Theorem \ref{main3} implies Theorem \ref{main1}.  
\end{proof}
In the next section, we present the proof of Theorem \ref{main3}.

\section{Proof of Theorem \ref{main3}.}\label{sec_main3}
To prove Theorem \ref{main3}, we first recall Fact~ \ref{unavoidable}.  Fact~\ref{unavoidable} and its proof were obtained by A. Shimizu \cite[Proof of Prop.~3.1]{shimizu2014}.  
\begin{fact}[\cite{shimizu2014}]\label{unavoidable}
A reduced knot projection $P$ contains at least one element of the following set: 
\begin{center}
\includegraphics[width=8cm]{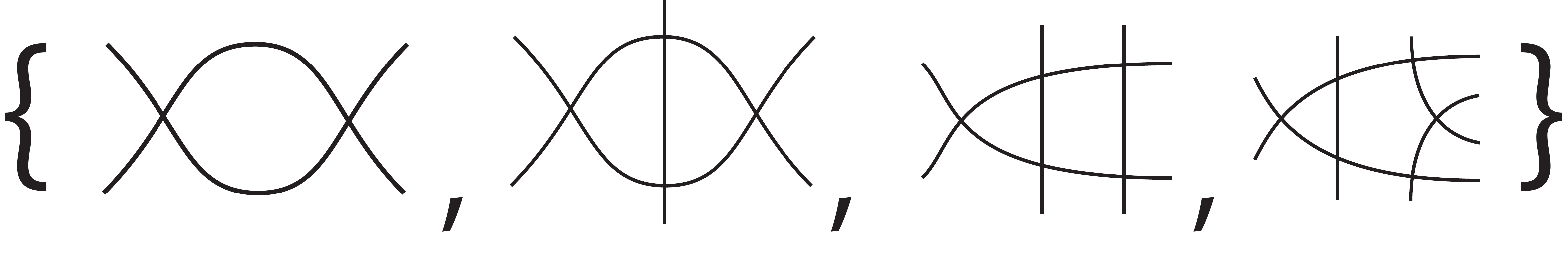}~.
\end{center}  
\end{fact}
We must also check the following Lemma \ref{basic_lem}.  
Recall that a knot projection that is not a simple closed curve $\bigcircle$ is called a {\it{non-trivial}} knot projection.  
\begin{lemma}\label{basic_lem}

\begin{description}
\item[(a)] A non-trivial knot projection with no $1$- or $2$-gons has at least eight $3$-gons.  \label{L1} $($Cf.~\cite[Theorem 2.2]{EHK}$.)$
\item[(b)] If a non-trivial prime knot projection $P$ with no $1$- or $2$-gons has at least one $3$-gon in $\{A, B, C\}$ in the following, then $P$ has a triple chord in $CD_P$. \label{L2}  
\end{description}
\begin{figure}[h!]
\includegraphics[width=8cm]{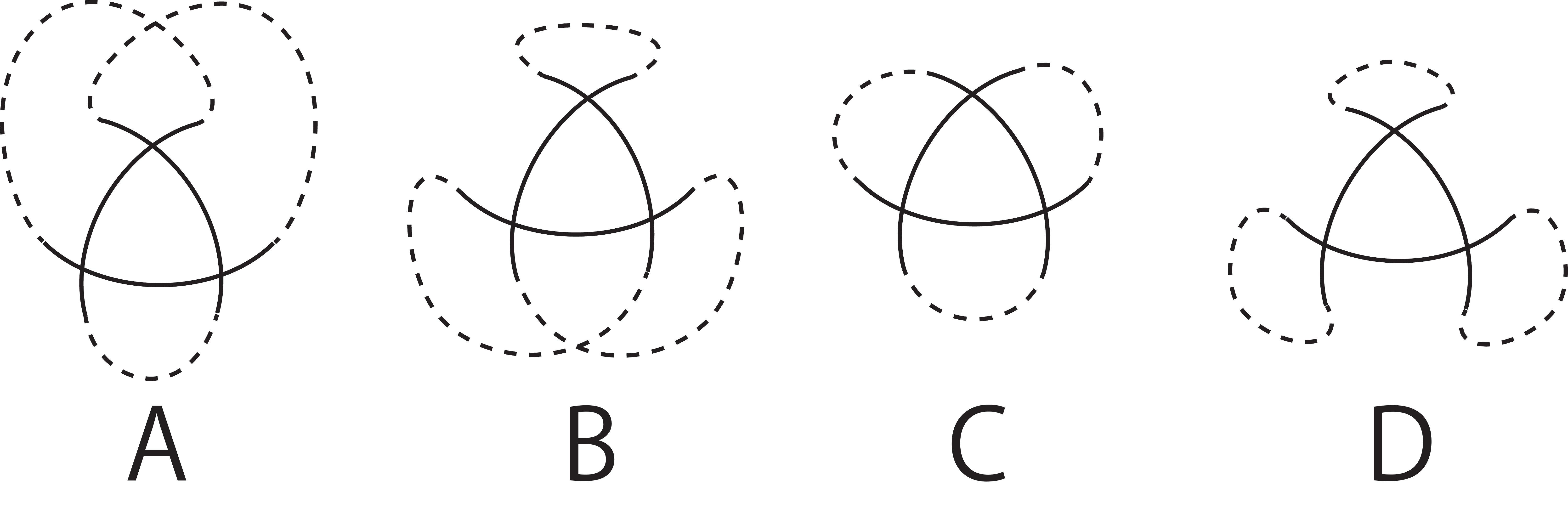}
\caption{All types of $3$-gons.  Dotted arcs show the connections of arcs.}\label{3hen}
\end{figure}
\end{lemma}
\begin{proof}
\begin{description}
\item[(a)] Let $V$ be the number of double points (i.e., vertices), $E$ the number of edges, and $F$ the number of faces.  Let $p_k$ be the number of $k$-gons.  For a non-trivial knot projection $P$ with no $1$- or $2$-gons, 
\begin{equation}
\begin{split}
&\sum_{k \ge 3} k p_k = 2E, \\
&\sum_{k \ge 3} p_k = F.  \label{e1}
\end{split}
\end{equation}
Now, we consider knot projections that are graphs on $S^2$ such that every vertex has four edges.  Thus, 
\begin{equation}
\begin{split}
&4V=2E, \\
&V - E + F =2.  \label{e2}
\end{split}
\end{equation}
Formula (\ref{e2}) implies $4F-2E=8$.  Substituting formula (\ref{e1}) into $2E$ and $F$ of $4F-2E=8$, we have 
\begin{equation*}
p_3 + \sum_{k \ge 4} (4-k) p_k = 8.  
\end{equation*}
Then, we have $p_3 \ge 8$.  This completes the proof.  
\item[(b)] 
\begin{itemize}
\item A-type $3$-gon.  Observe the figure of the spherical curve that contains dotted arcs and an A-type $3$-gon shown in Fig.~\ref{3hen}.  
\begin{figure}[h!]
\includegraphics[width=6cm]{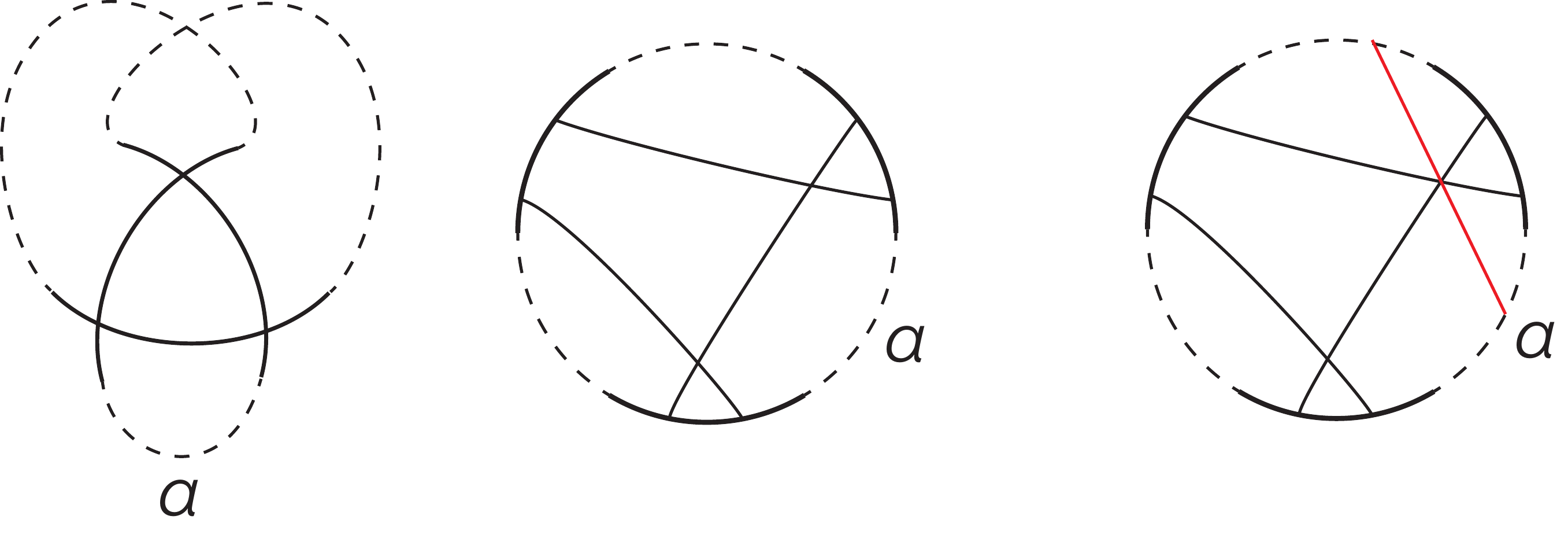}
\caption{A-type $3$-gon having the dotted arc labeled $\alpha$ (left), the corresponding chord diagram (center), and chord diagram with a triple chord (right).}\label{tr4}
\end{figure}
From the assumption, a knot projection $P$ containing an $A$-type $3$-gon is prime.  Then, the $\alpha$-part of Fig.~\ref{tr4} must intersect at least one of the other dotted arcs.  Similar to Fig.~\ref{tr4}, there exists $\tr$ in $CD_P$.   
\item B-type $3$-gon.  Note the spherical curve $P$ that contains dotted arcs and a B-type $3$-gon shown in Fig.~\ref{3hen}.  The corresponding chord diagram $CD_P$ is shown at the right of Fig.~\ref{tr5}.  
\begin{figure}[h!]
\includegraphics[width=5cm]{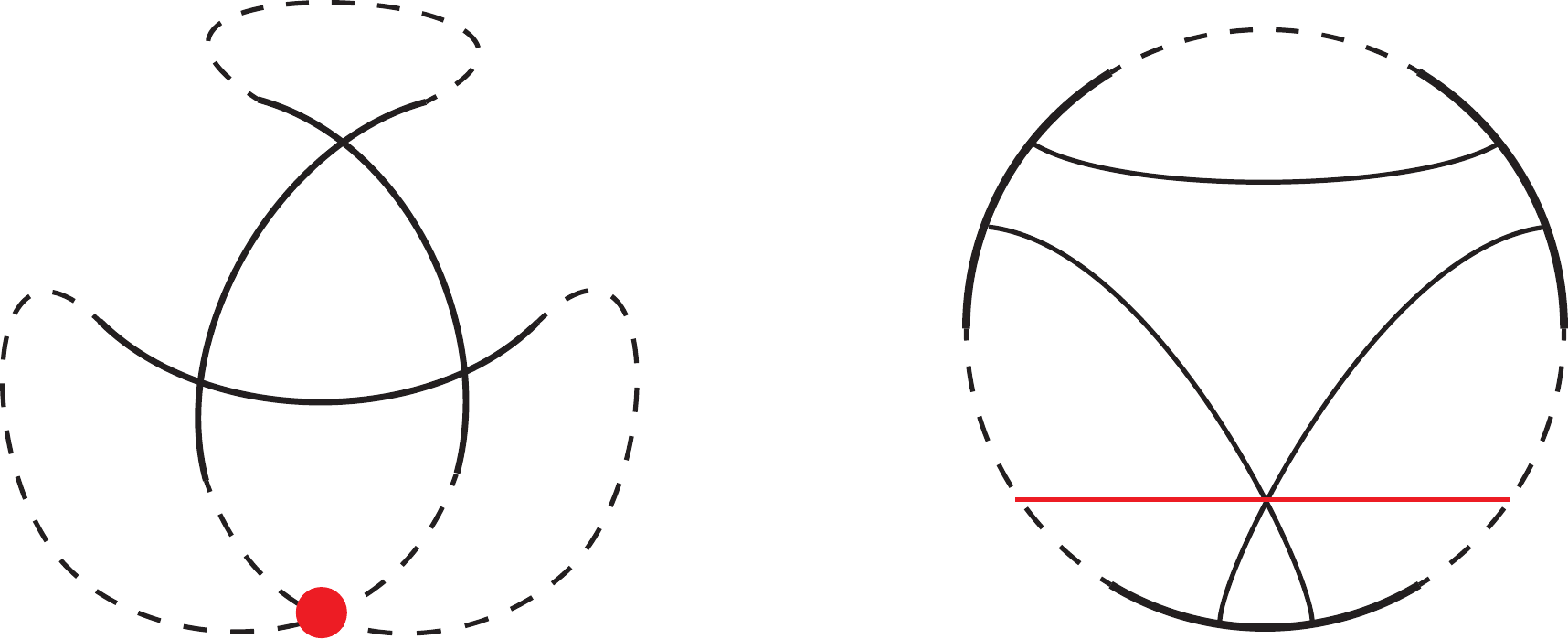}
\caption{B-type $3$-gon, which must have a red double point, and its chord diagram.}\label{tr5}
\end{figure}
In $CD_P$, we can find $\tr$, since there are two dotted arcs in a $B$-type $3$-gon that must intersect (Fig.~\ref{tr5}, left).  
\item C-type $3$-gon.  If a knot projection $P$ contains a C-type $3$-gon, then $CD_P$ immediately has a triple chord.  
\end{itemize}
\end{description}
The consideration of the three cases completes the proof.  
\end{proof}

Now, we prove Theorem \ref{main3}.  
\begin{figure}[htbp]
\includegraphics[width=6cm]{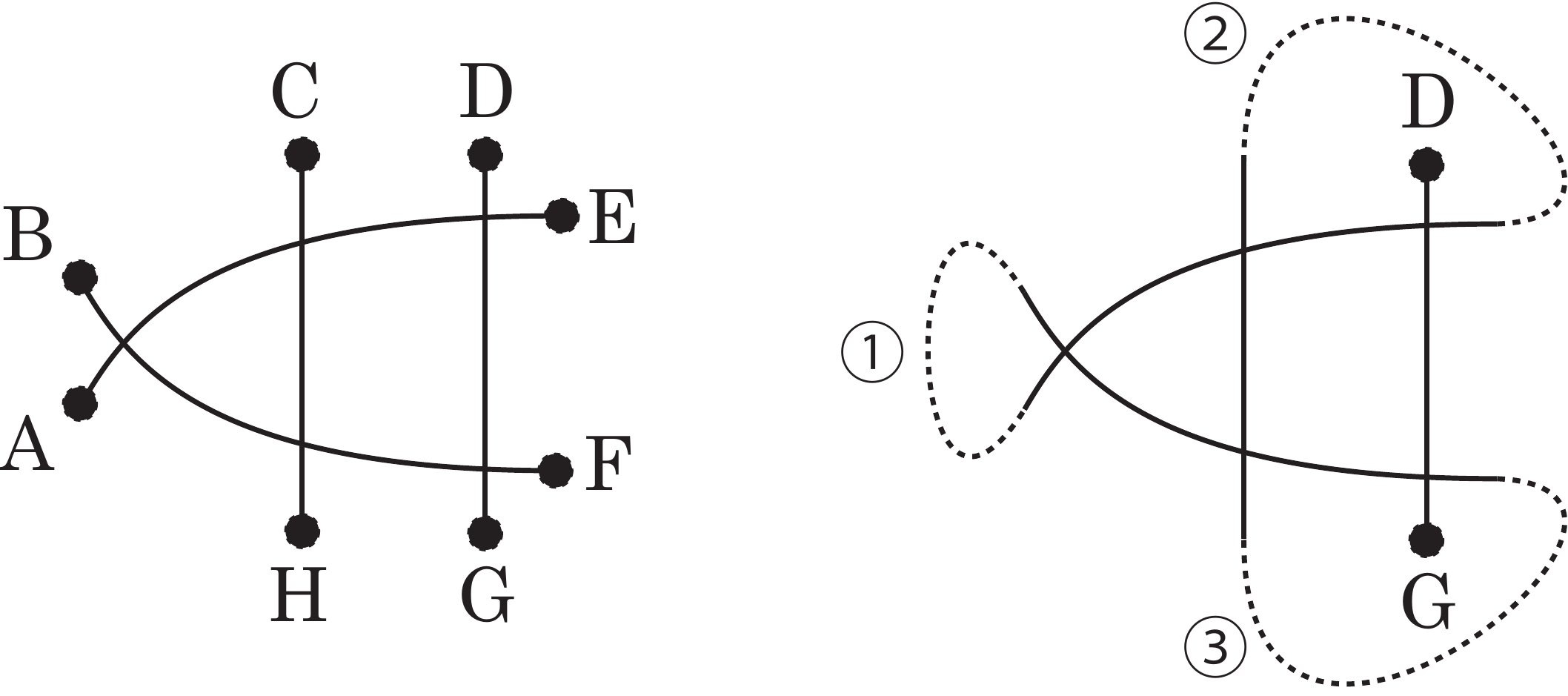}
\caption{Third element (left) of the set of Fact~\ref{unavoidable} and the case having a D-type $3$-gon (right).  Dotted arcs show the connections of non-dotted arcs.}\label{3}
\end{figure}
\begin{proof}
By Fact \ref{unavoidable}, a knot projection $P$, which we have considered, contains at least one of the elements mentioned in Fact \ref{unavoidable}.  Thus, we consider the possibilities that $P$ contains the first, second, third, or fourth of those elements.  In what follows, checking the possibility of the first (resp.~second, third, or fourth) element is called {\it{the first (resp.~second, third, or fourth) element case}}.  

\noindent{\bf{The first element case.}}

By assumption, $P$ has no $2$-gon.  Thus, there is no possibility of the existence of the first element of the set of Fact \ref{unavoidable}.  

\noindent{\bf{The second element case.}}

If $P$ has the second element (i.e., two neighboring $3$-gons) from the left-hand side of the set shown in Fact \ref{unavoidable}, assume that one of the two neighboring $3$-gons is $D$ type.  
In this case, another $3$-gon in the two neighborhood $3$-gons is type $B$.  This implies that $P$ has at least a type $A$, $B$, or $C$ $3$-gon, from which we conclude that $P$ has triple chords in $CD_P$ by Lemma \ref{basic_lem}.  Thus, it is sufficient to consider the two cases of the third or the fourth figure from the left-hand side in the set of Fact \ref{unavoidable}.  Below, we consider these figures.  

\noindent {\bf{The third element case.}} 

By Lemma \ref{basic_lem}, if a knot projection contains the part shown in Fig.~\ref{3}, we can assume that the $3$-gon is type $D$ from that figure.  Since the $3$-gon is type $D$, dotted arcs arise as shown at the right-hand side.  Thus, we distinguish the following cases in which a dotted arc contains the arc DG shown in the figure:
\begin{itemize}
\item Arc number 1 contains DG (Case A, B), 
\item Arc number 2 (or 3) contains DG (Case C, D).  
\end{itemize}
In the remainder of the proof, the symbol (X, Y) (resp.~($x \sim y$)) means we connect a point X with a point Y (resp.~ a vertex $x$ with a vertex $y$) via a route outside the fixed part of a knot projection, e.g., as seen below, Case A and Fig.~\ref{casea}.  
\begin{itemize}
\item Case A is defined by (A, G), (B, D), (C, E), and (F, H).  See Fig.~\ref{casea}.  
\begin{figure}[h!]
\includegraphics[width=5cm]{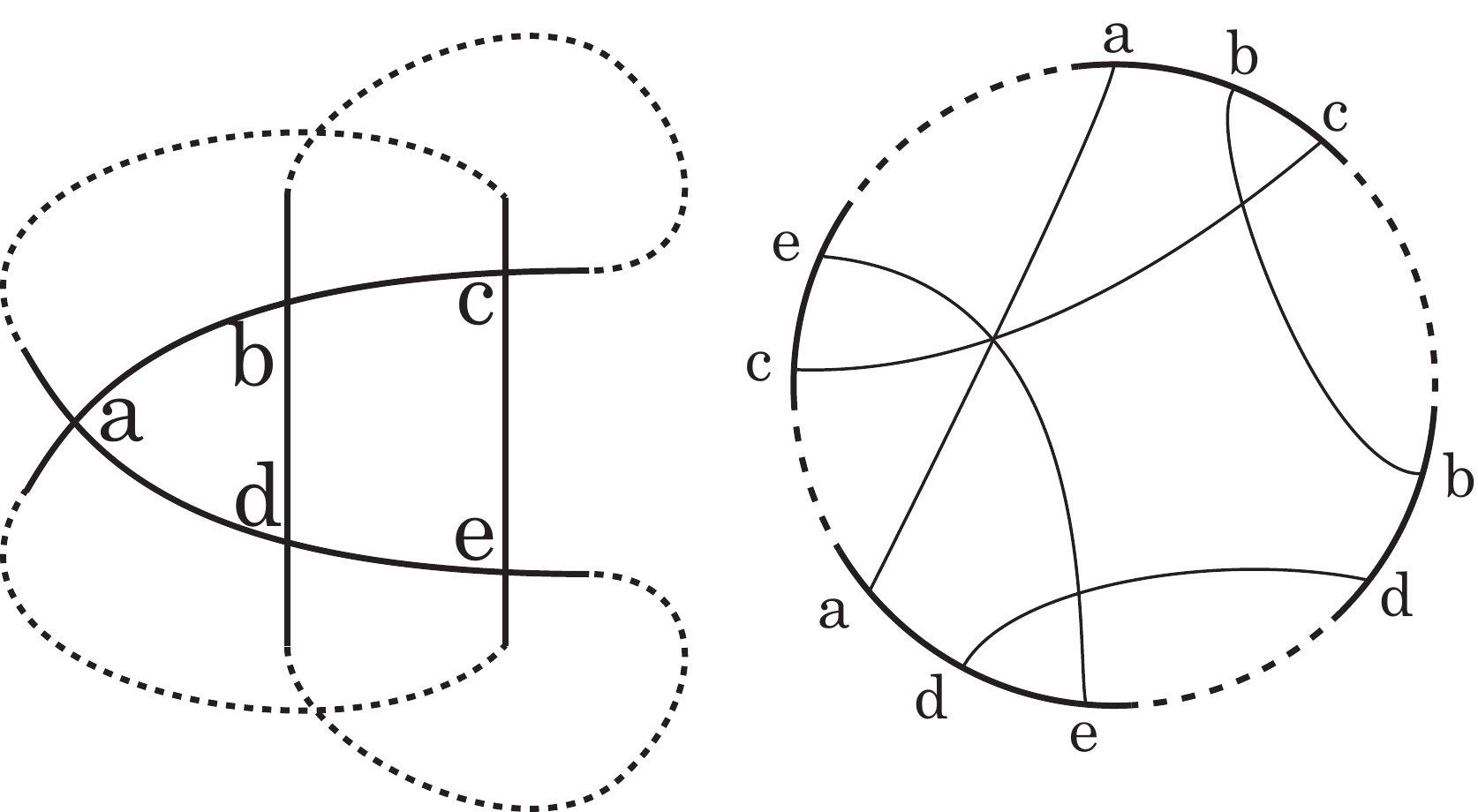}
\caption{Case A.}\label{casea}
\end{figure}
\item Case B is defined by (A, D), (B, G), (C, E), and (F, H).  See Fig.~\ref{caseb}.  
\begin{figure}[h!]
\includegraphics[width=5cm]{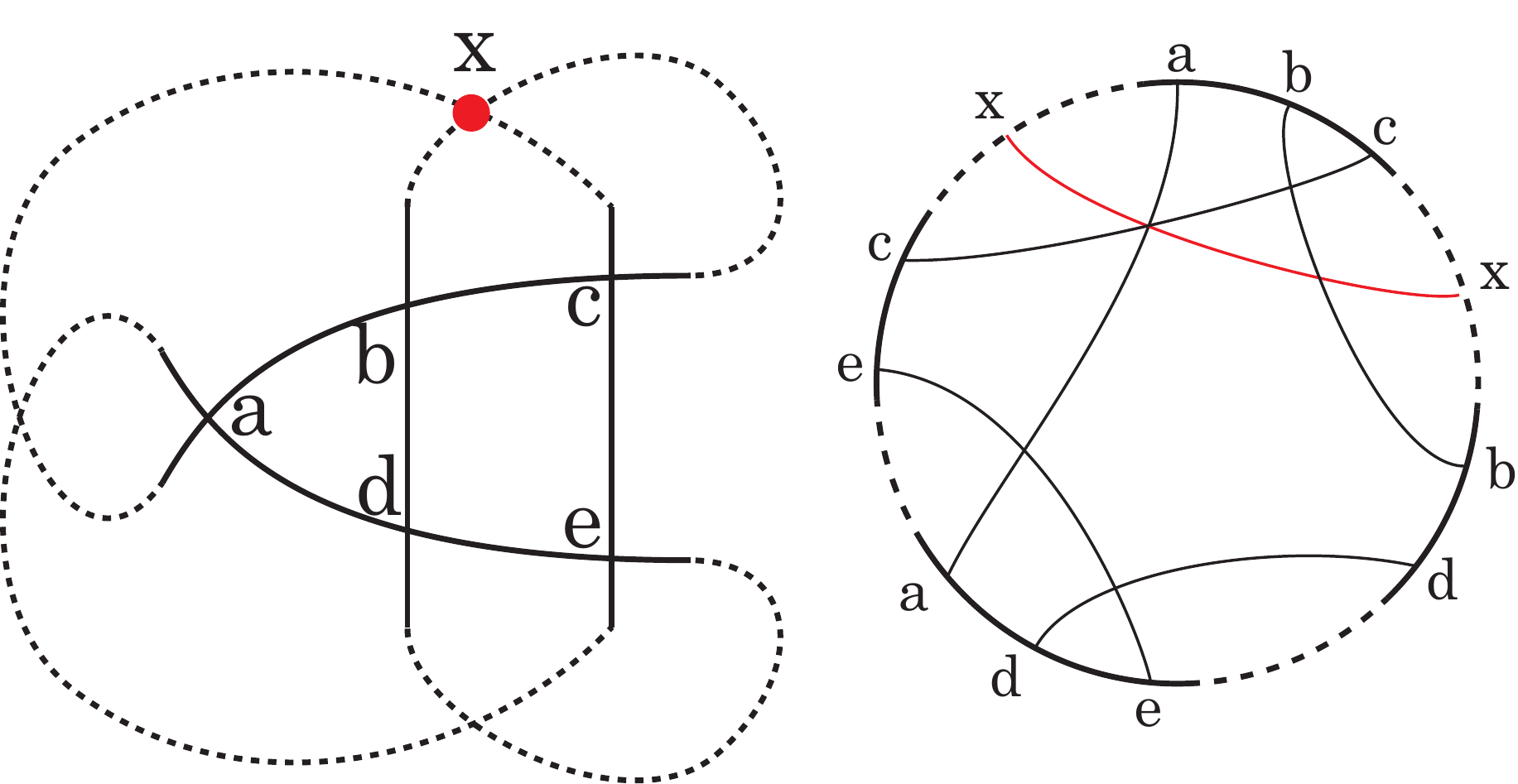}
\caption{Case B.}\label{caseb}
\end{figure}
\item Case C is defined by (A, B), (C, D), (E, G), and (F, H).  See Fig.~\ref{casec}.  
\begin{figure}[h!]
\includegraphics[width=5cm]{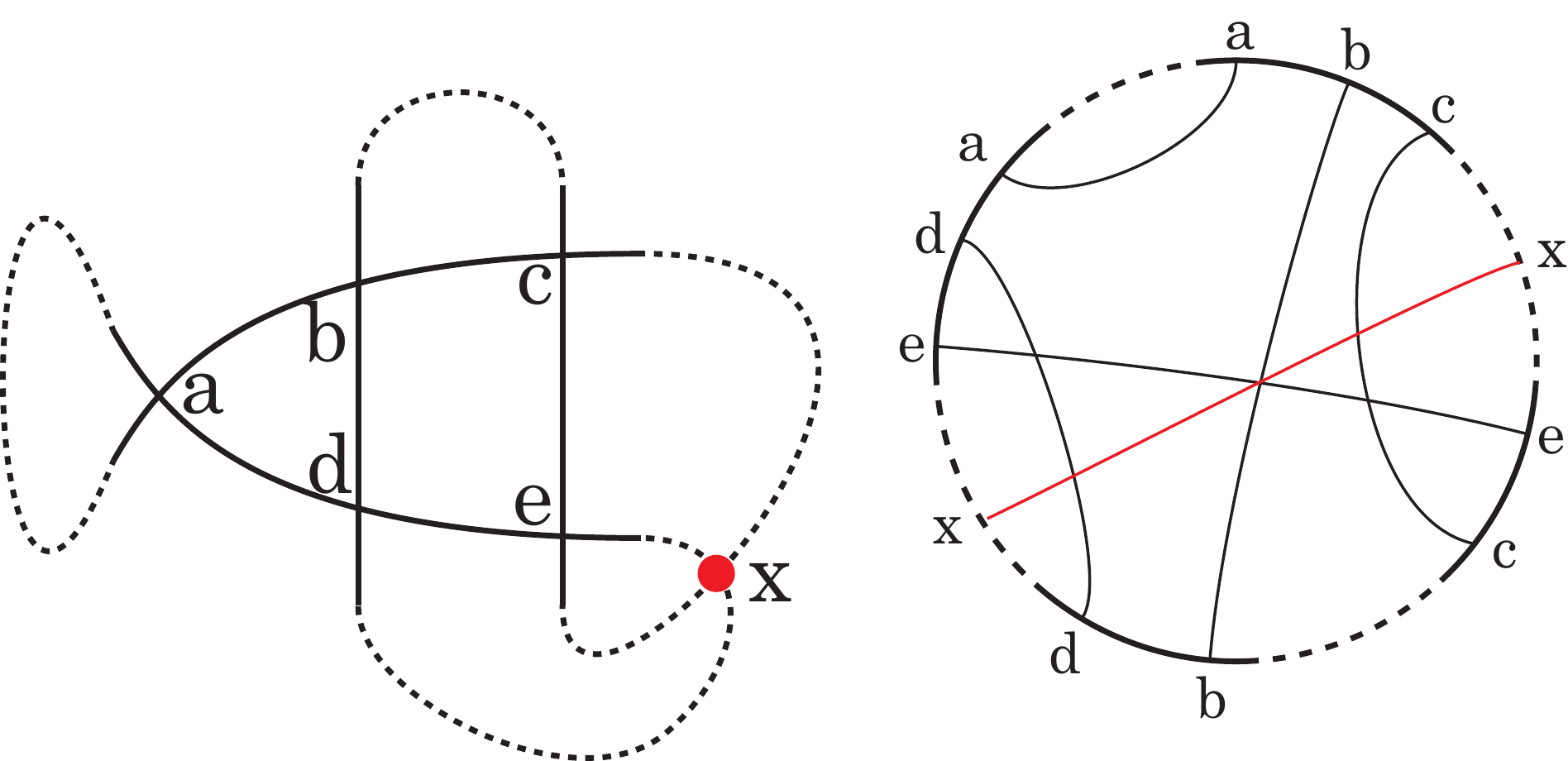}
\caption{Case C.}\label{casec}
\end{figure}
\item Case D is defined by (A, B), (C, G), (D, E), and (F, H) as shown in Fig.~\ref{cased}.  
\begin{figure}[h!]
\includegraphics[width=5cm]{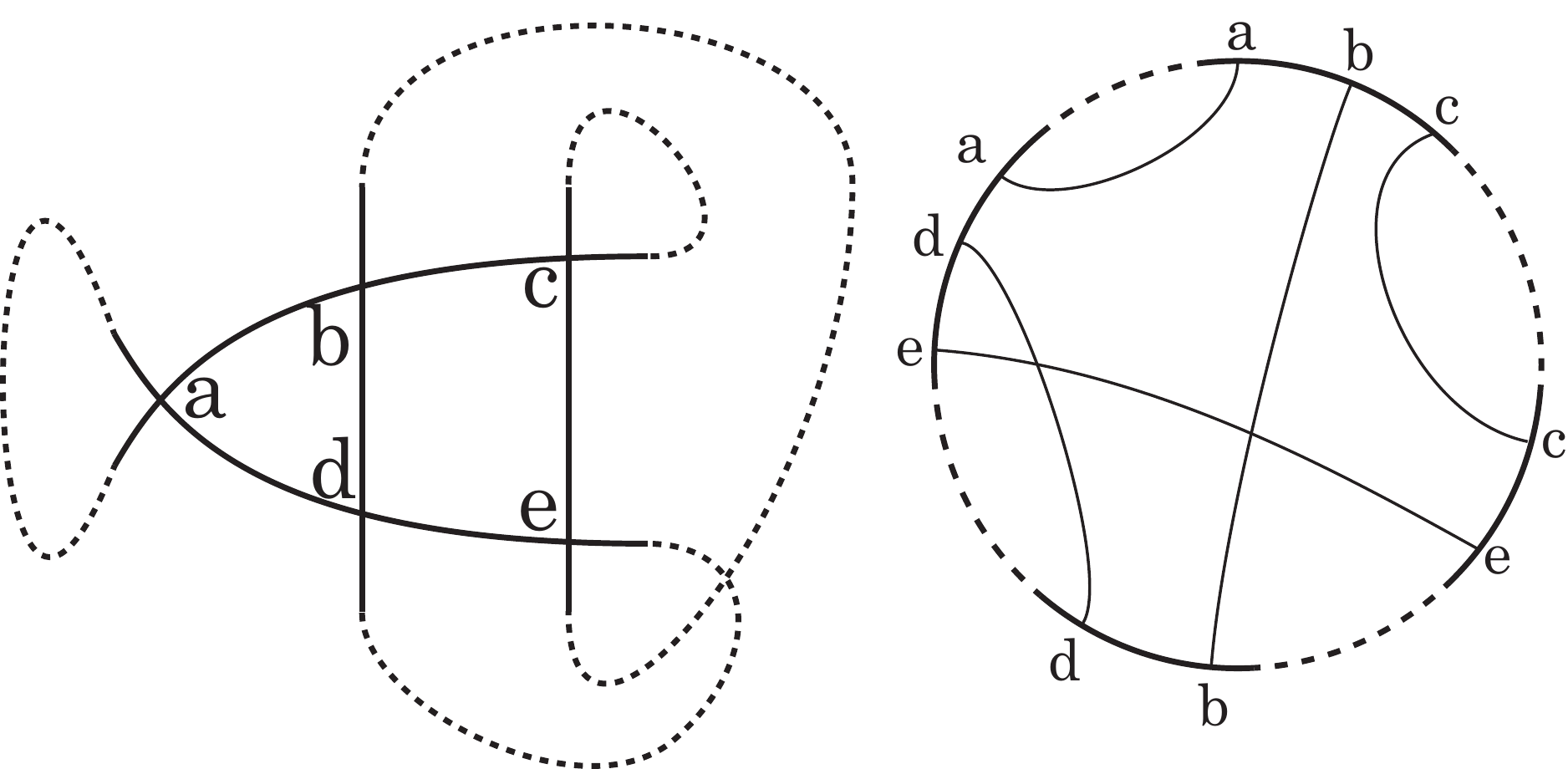}
\caption{Case D.}\label{cased}
\end{figure}
This knot projection $P$ is a prime knot projection with no $1$- or $2$-gons; hence, $P$ is reduced (Lemma \ref{lem_reduced}).  Thus, (a$\sim$a) intersects another dotted arc ($\ast$).  If ($\ast$) is (b$\sim$e) or (d$\sim$e), $P$ has a triple chord in $CD_P$.  If ($\ast$) is neither (b$\sim$e) nor (d$\sim$e), but is (c$\sim$c), the knot projection $P$ and its $CD_P$ appears as Fig.~\ref{cased1}, and thus, there exists a triple chord in $CD_P$. 
\begin{figure}[h!]
\includegraphics[width=5cm]{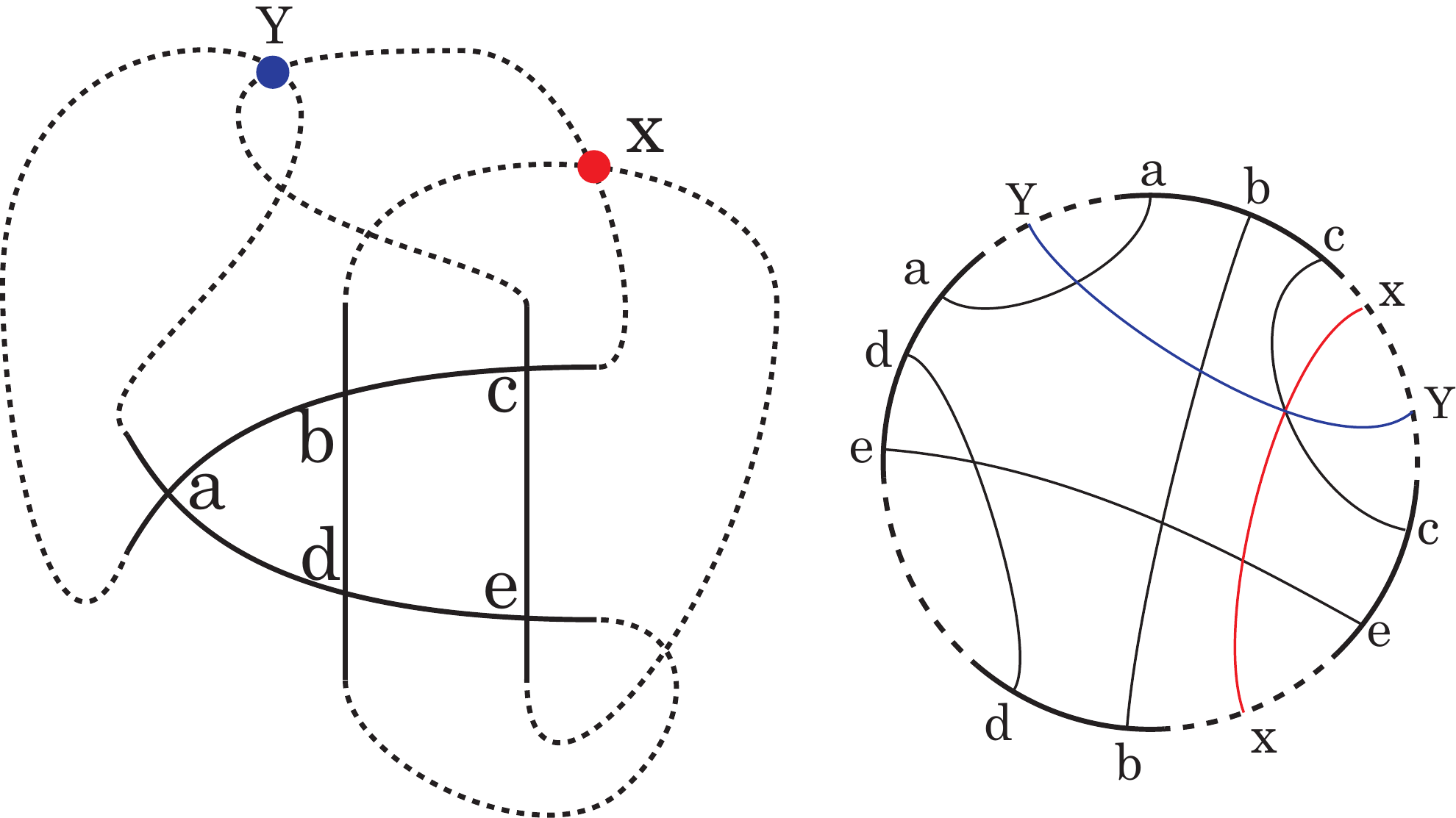}
\caption{Instance of Case D.}\label{cased1}
\end{figure} 
\end{itemize}
In summary, if a knot projection $P$ has the third element of the set of Fig.~\ref{unavoidable}, then $P$ has a triple chord in $CD_P$.  Then, we are left with only the case of the fourth element of the set of Fact \ref{unavoidable}.  

\noindent {\bf{The fourth element case.}}
Start by setting the symbols for points to be connected and vertices as in Fig.~\ref{4}.  
\begin{figure}[htbp]
 \begin{center}
 \includegraphics[width=8cm]{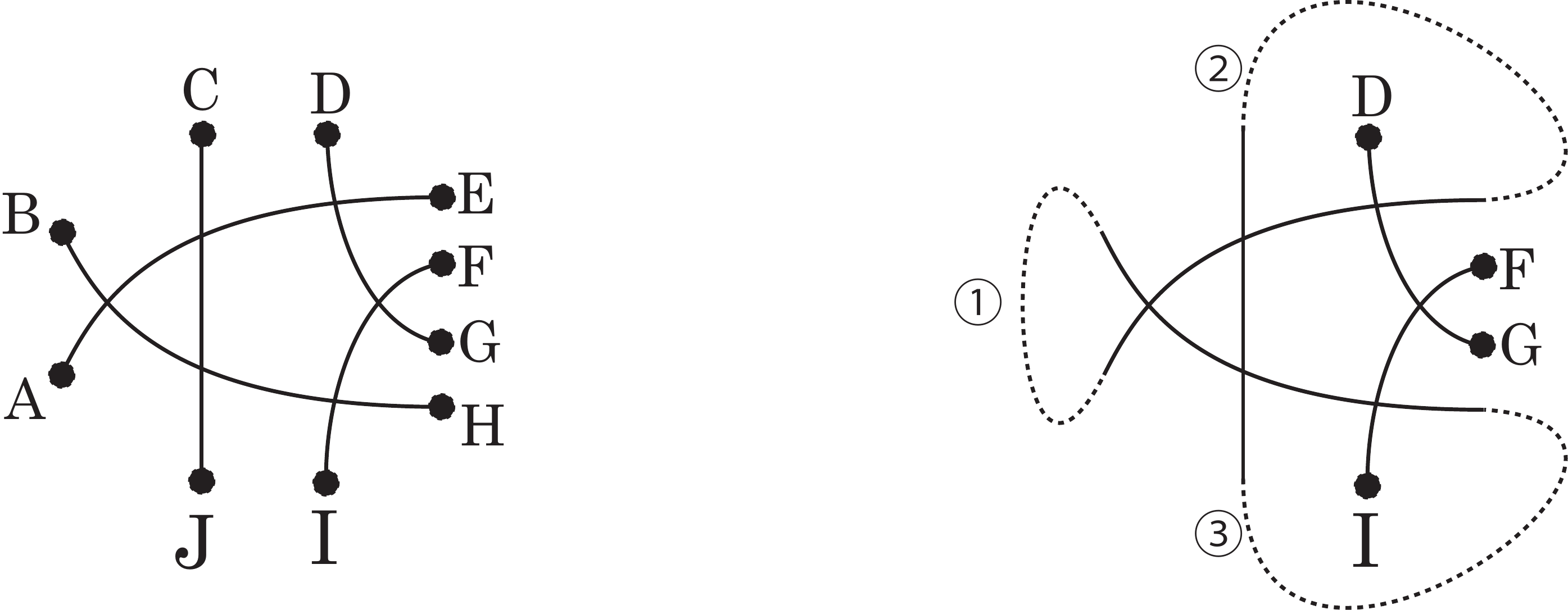}
 \end{center}
 \caption{Fourth element (left) of the set of Fact~\ref{unavoidable} and the case having a D-type $3$-gon (right).  Dotted arcs show the connections of non-dotted arcs.}\label{4}
\end{figure}
By Lemma \ref{unavoidable}, we can fix the $3$-gon in Fig.~\ref{4} as type $D$.  Then, we can draw dotted arcs as in the figure.  Next, we consider the dotted arcs that contain the non-dotted arcs DG and FI.  Based on this consideration, we prove the claim case by case.  Since there is the symmetry between arc numbers $2$ and $3$, it is sufficient to consider the following four groups, each of which contains eight cases (in total, 32 cases).  
\begin{table}[htbp]
\caption{Each of four groups having eight cases.  Dotted arcs show the connections of arcs.}\label{4a}
\begin{tabular}{|c|c|c|c|}\hline
\includegraphics[width=2.5cm]{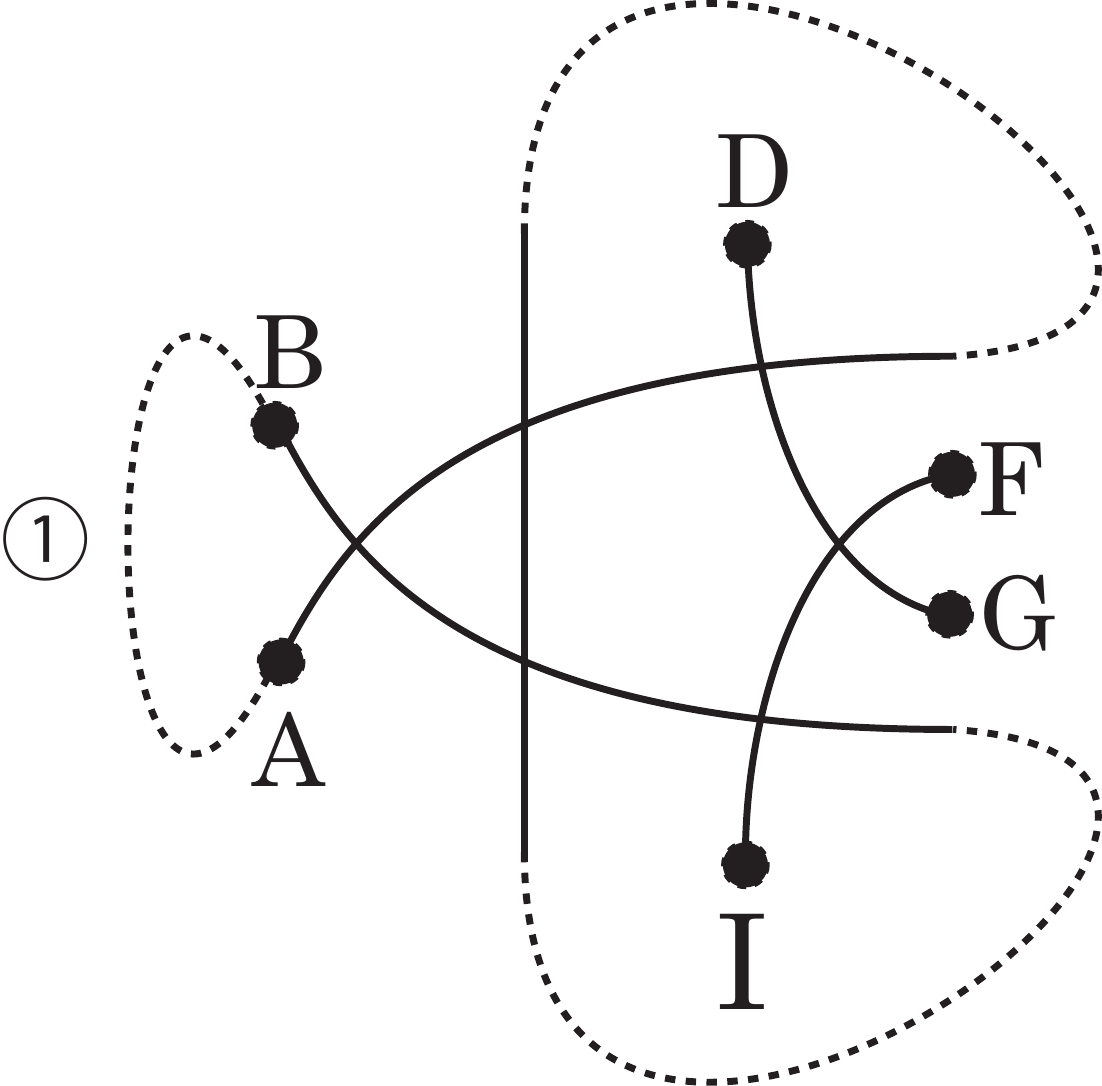} &
\includegraphics[width=2.5cm]{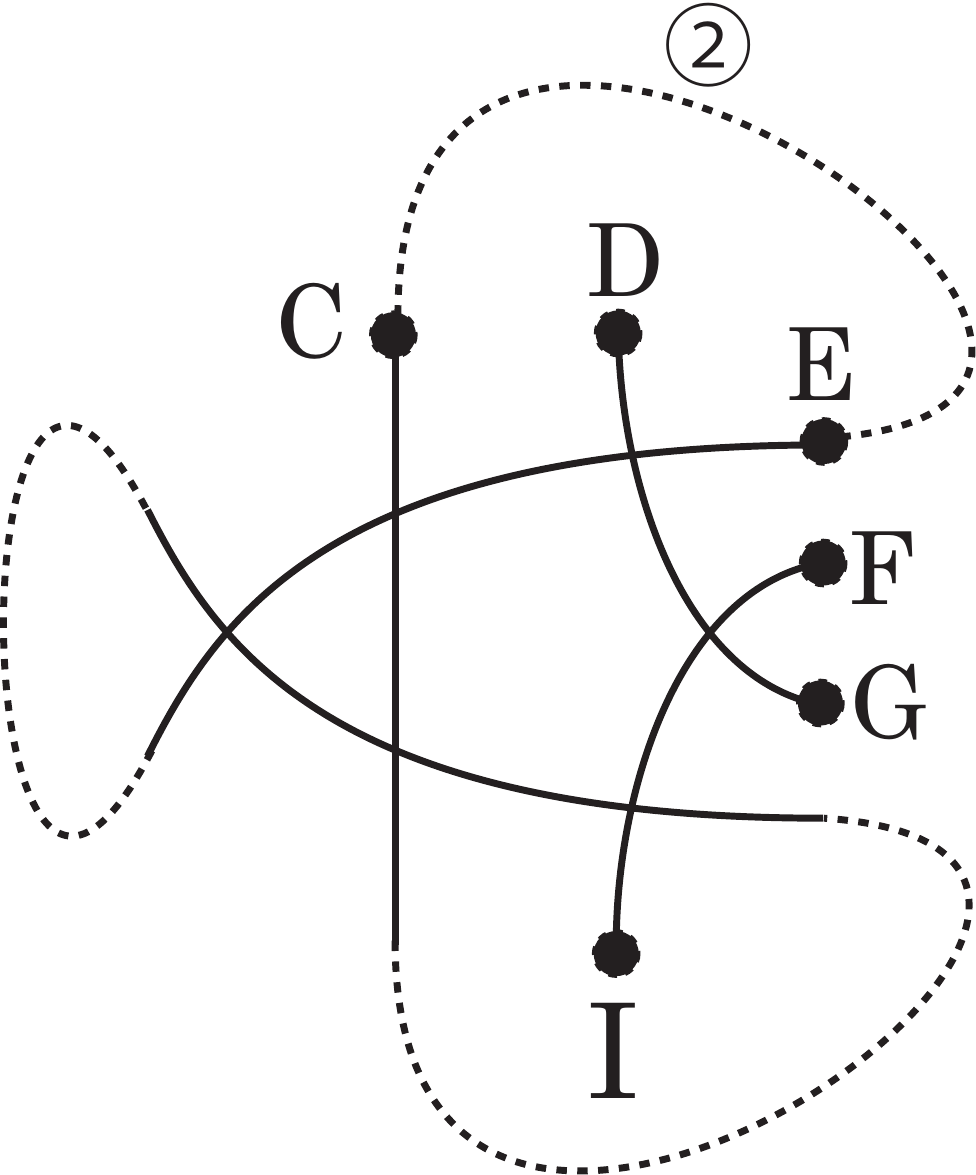} &
\includegraphics[width=2.5cm]{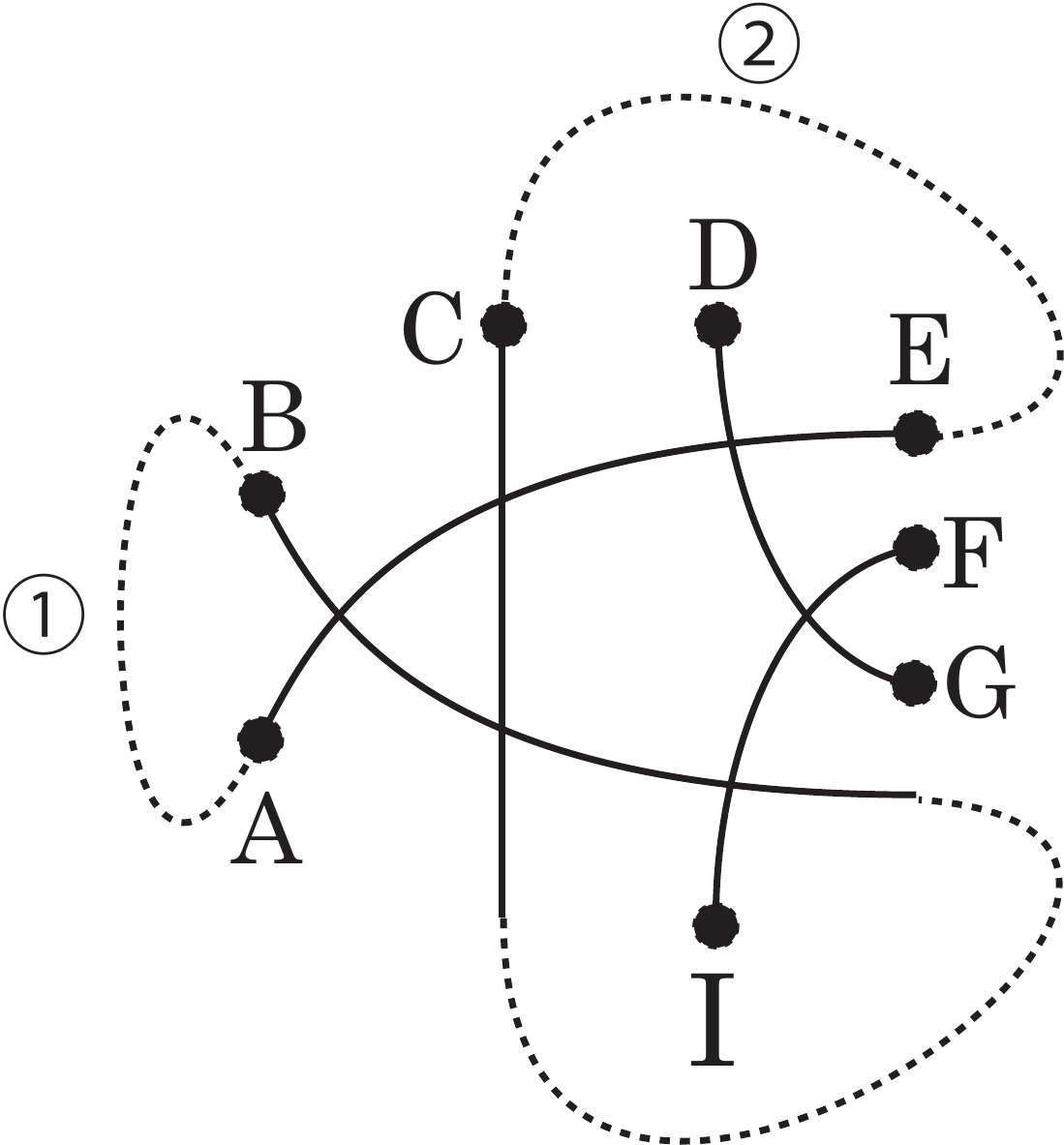} &
\includegraphics[width=2.5cm]{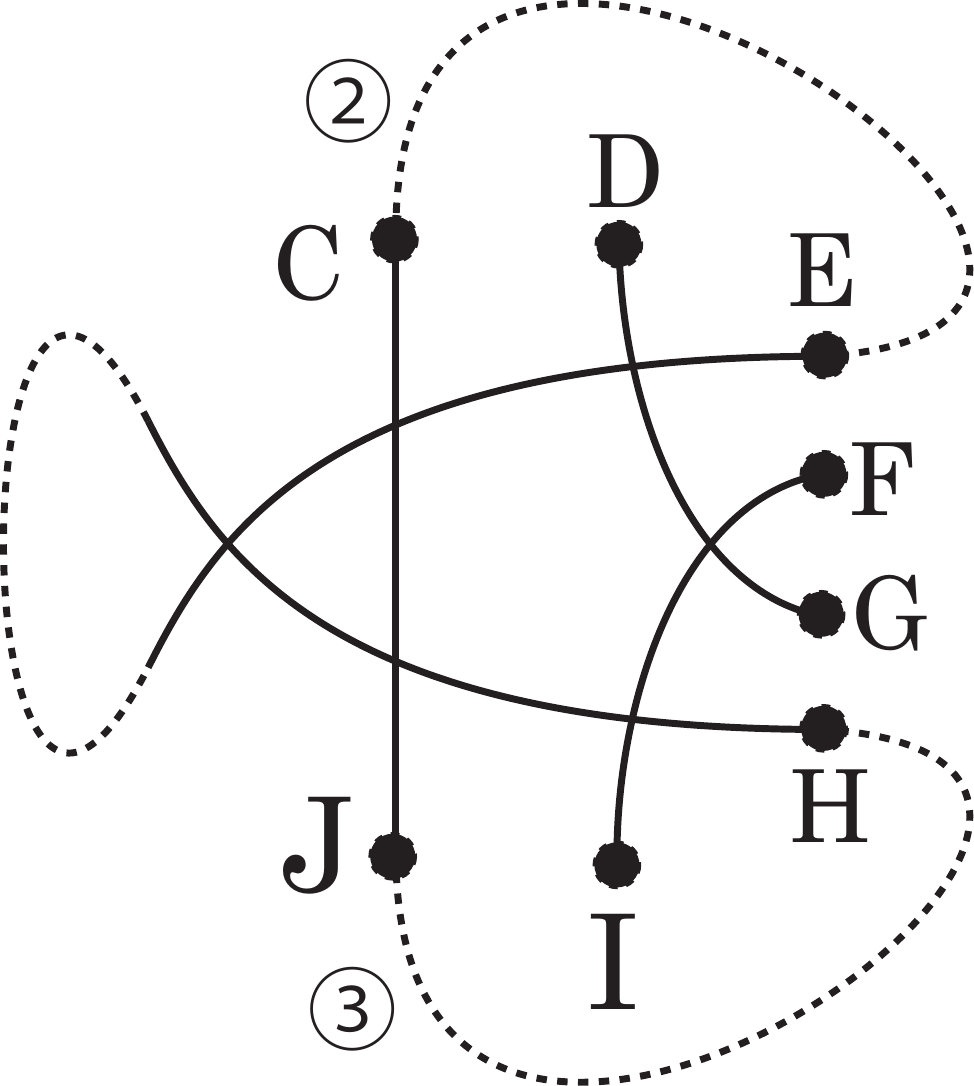} \\ \hline
\end{tabular}
\end{table}
The points of grouping are as follows.  
\begin{itemize}
\item Dotted arc number $1$ contains both DG and FI (Cases 1--8).  This condition fixes (C, E) and (H, J).  
\item Dotted arc number $2$ contains both DG and FI (Cases 9--16). This condition fixes (A, B) and (H, J).  
 
(Replacing $2$ with $3$ replicates the discussion as a result of their symmetry; hence, we omit the respective cases.)

\item Arc number $1$ contains exactly one non-dotted arc (i.e.,~DG or FI), and arc number $2$ contains exactly one non-dotted arc (Cases 17--24).  This condition fixes (H, J).  

(Replaying $2$ with $3$ replicates the discussion as a result of their symmetry, hence we omit the respective cases.)

\item Arc number $2$ contains exactly one non-dotted arc (i.e.,~DG or FI) and arc number $3$ contains exactly one non-dotted arc (Cases 25--32).  This condition fixes (A, B).  
\end{itemize}

\noindent $\bullet$ {\bf{Cases 1--8.}}  
If arc number $1$ contains both two arcs DG and FI, then we can automatically fix (C, E) and (J, H) (Table~\ref{4a}, Cases 1--8).   
Table \ref{splittingCases1--8} shows how arcs connect, considering all possibilities.  Recall that the symbol ``(X, Y)'' means that we connect X and Y.  
\begin{table}
\caption{Method to split into Cases 1--8.}\label{splittingCases1--8}
\begin{tabular}{|c|c|}\hline
 $ ($B$,~ $D$) \begin{cases}
    ($G$,~ $F$) ($I$,~ $A$) & ($Case$~1) \\
    ($G$,~ $I$) ($F$,~ $A$) & ($Case$~2)
  \end{cases}
$
& $($B$,~ $F$) \begin{cases}
    ($I$,~ $D$) ($G$,~ $A$) & ($Case$~3) \\
    ($I$,~ $G$) ($D$,~ $A$) & ($Case$~4)
  \end{cases}
$
\\ \hline
$($B$,~ $G$) \begin{cases}
    ($D$,~ $F$) ($I$,~ $A$) & ($Case$~5) \\
    ($D$,~ $I$) ($F$,~ $A$) & ($Case$~6)
  \end{cases}$
  &
  $($B$,~ $I$) \begin{cases}
    ($F$,~ $D$) ($G$,~ $A$) & ($Case$~7) \\
    ($F$,~ $G$) ($D$,~ $A$) & ($Case$~8)
  \end{cases}$ \\ \hline
\end{tabular}
\end{table}
For every case 1--8, a knot projection $P$ has at least one triple chord in $CD_P$.  See Table~\ref{cases1--8}.  
\begin{table}
\caption{Easy cases to prove.  Cases 1--8.}\label{cases1--8}
\begin{tabular}{|c|c|c|} \hline
Case 1 & Case 2 & Case 3 \\ \hline
\includegraphics[width=3.5cm]{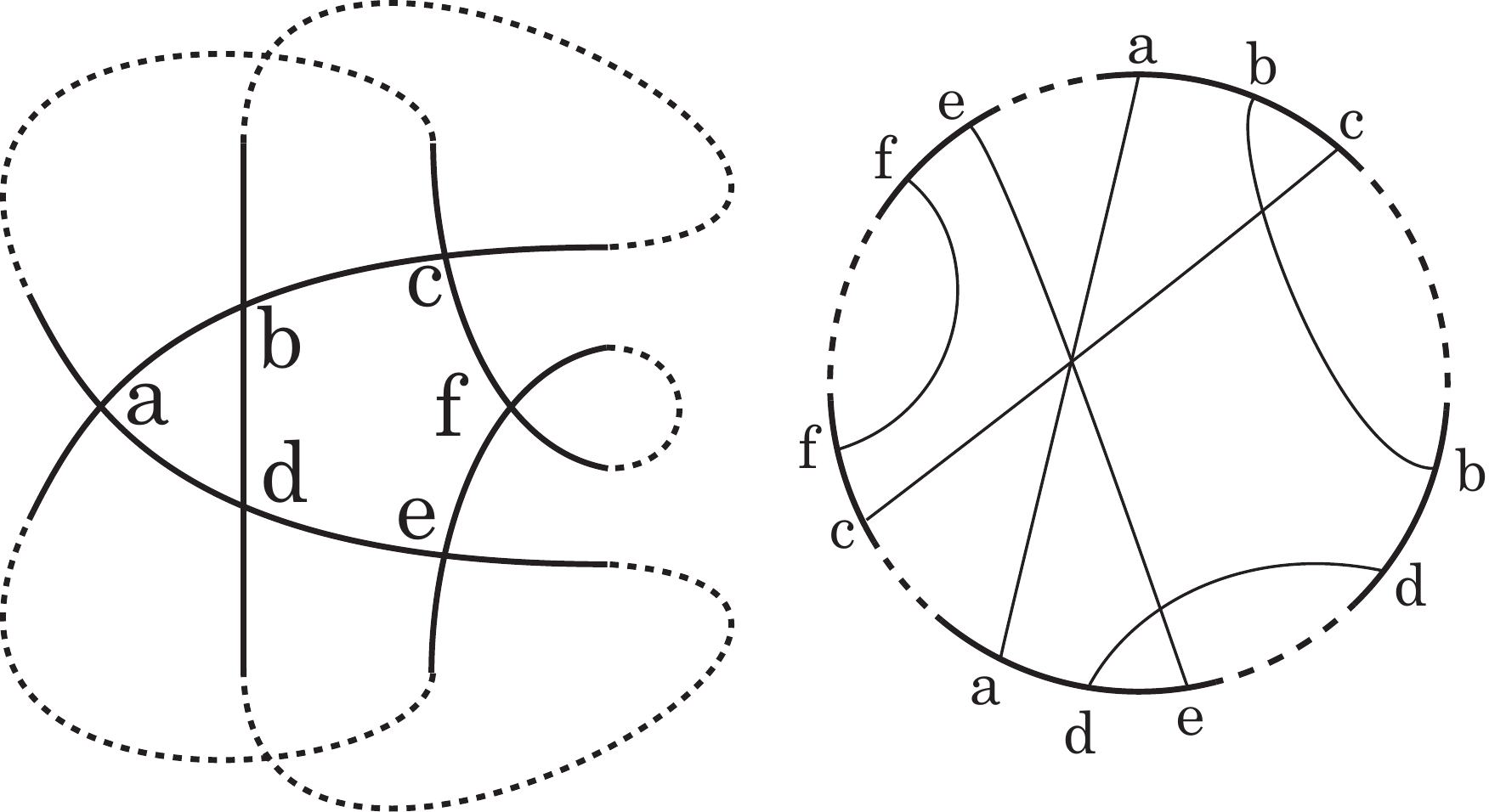} & \includegraphics[width=3.5cm]{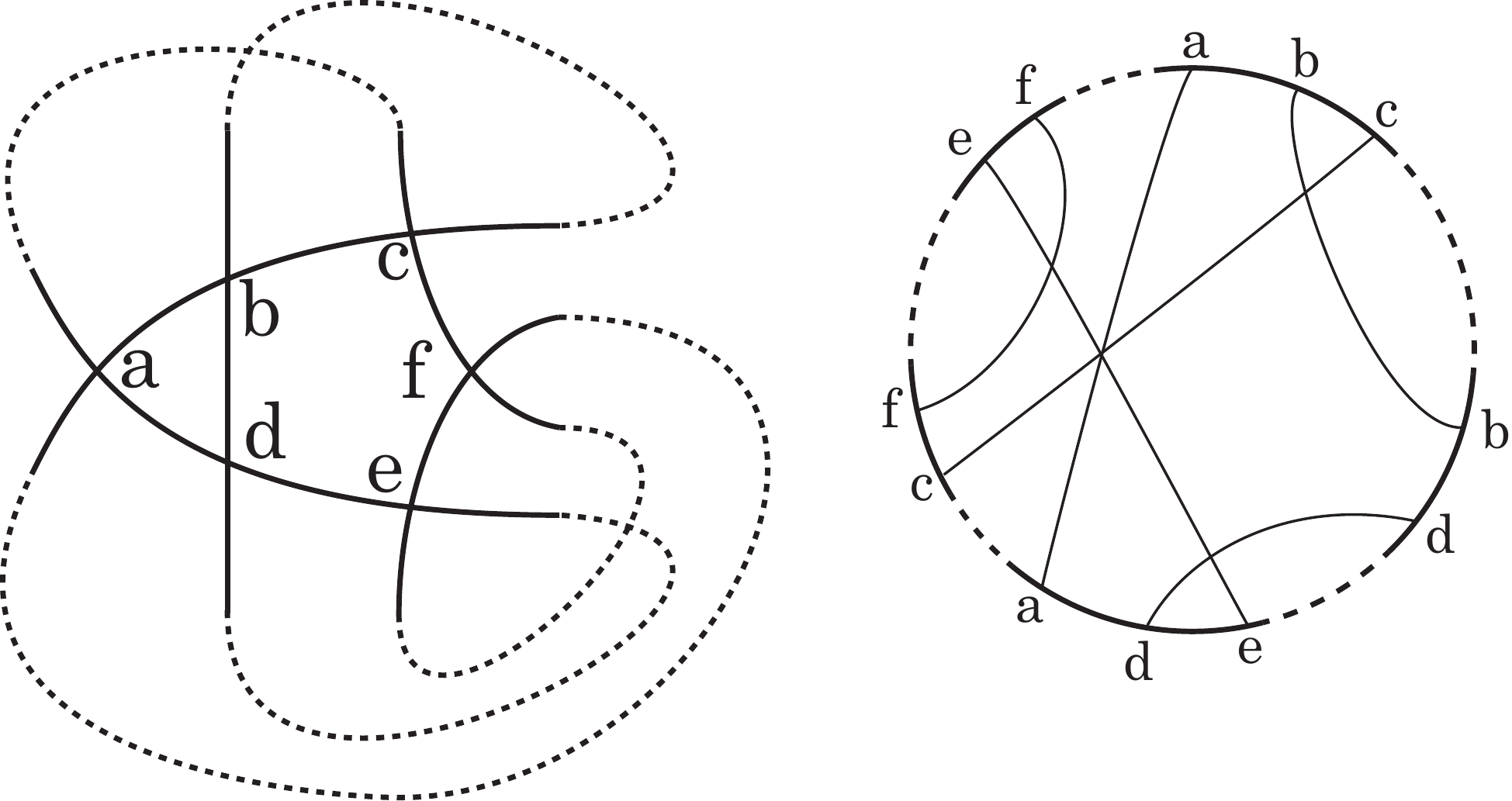} & \includegraphics[width=3.5cm]{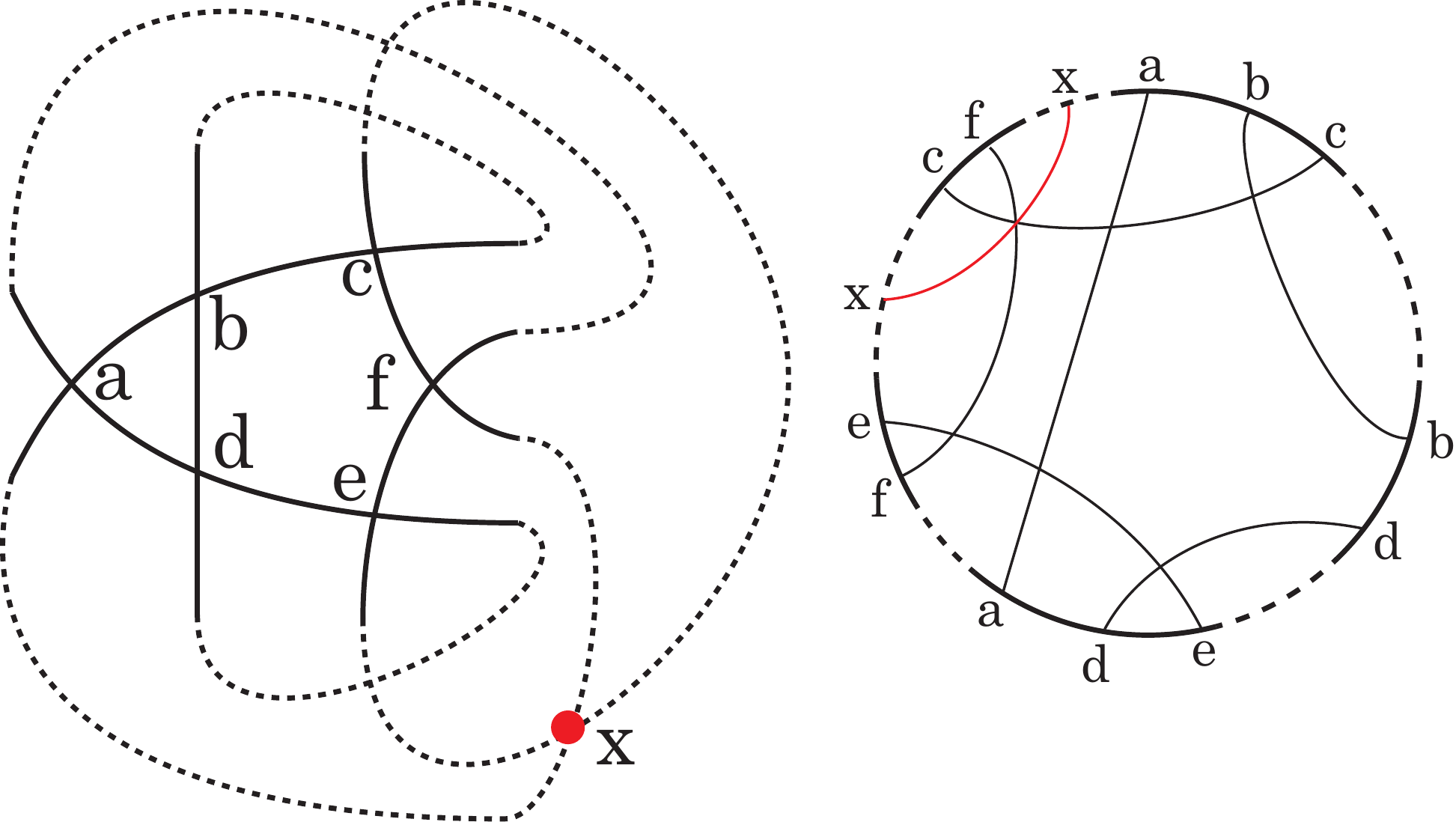} \\ \hline
Case 4 & Case 5 & Case 6 \\ \hline
\includegraphics[width=3.5cm]{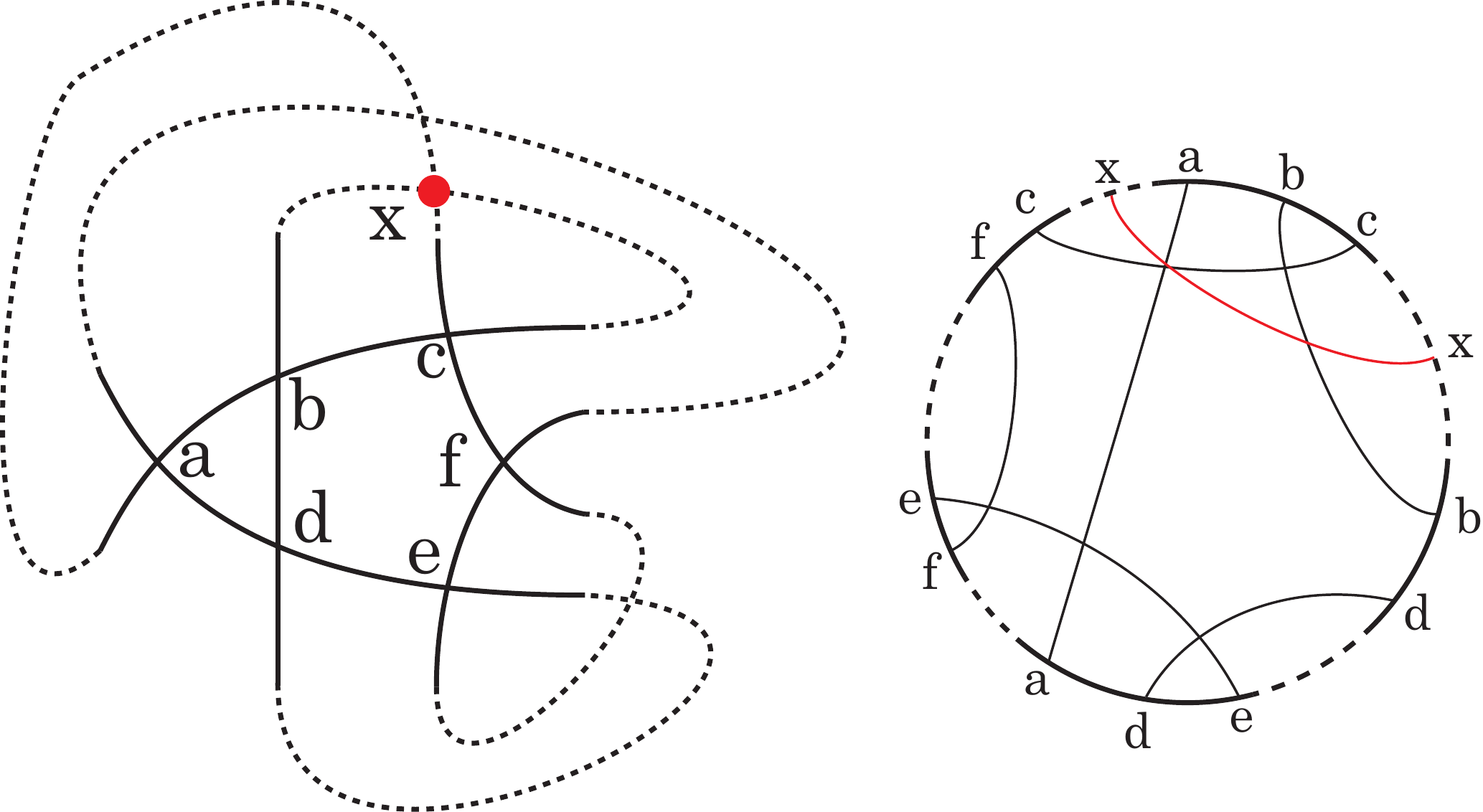} & \includegraphics[width=3.5cm]{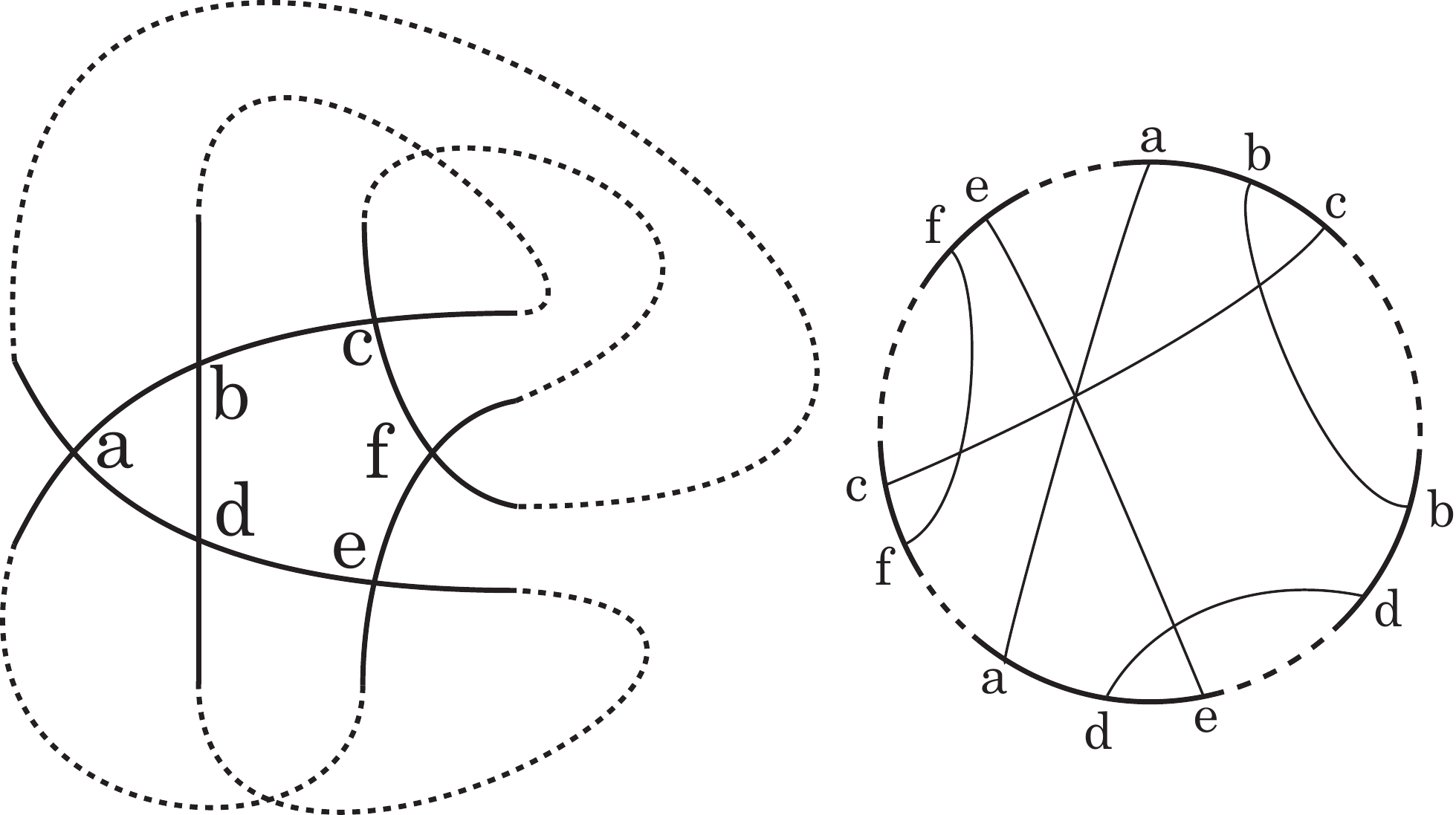} &
\includegraphics[width=3.5cm]{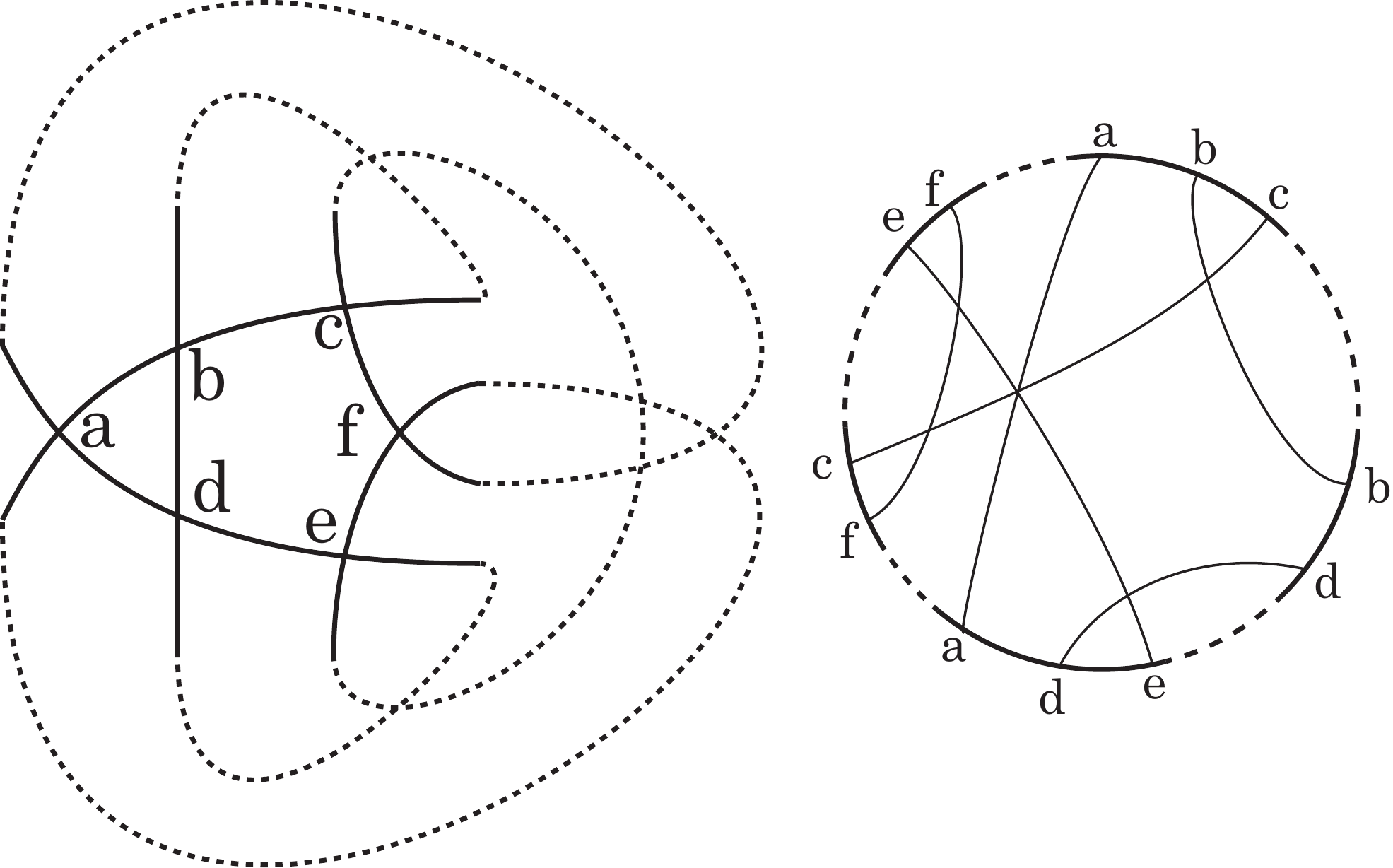} \\ \hline
Case 7 & Case 8 & \\ \hline
\includegraphics[width=3.5cm]{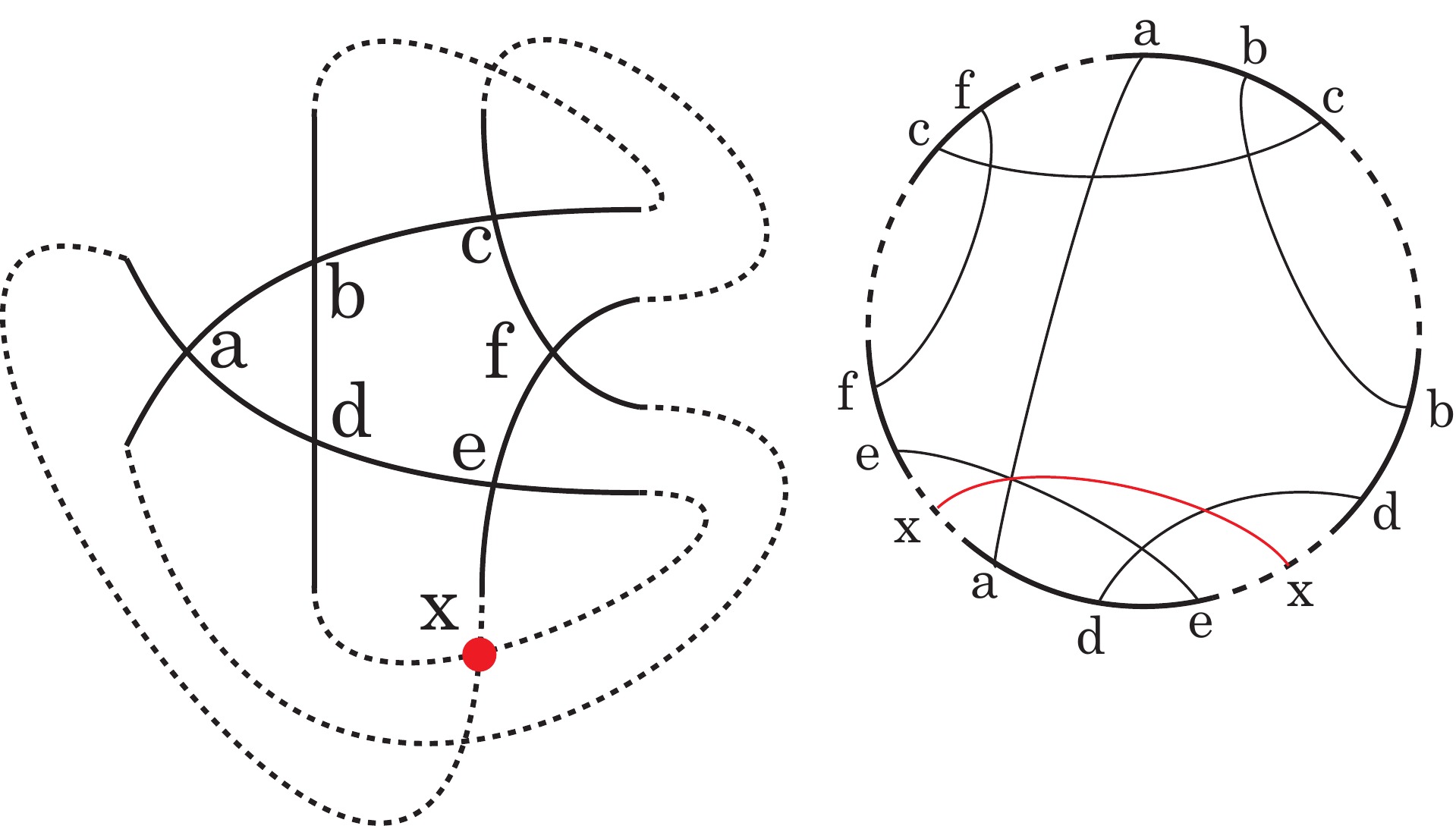} & \includegraphics[width=3.5cm]{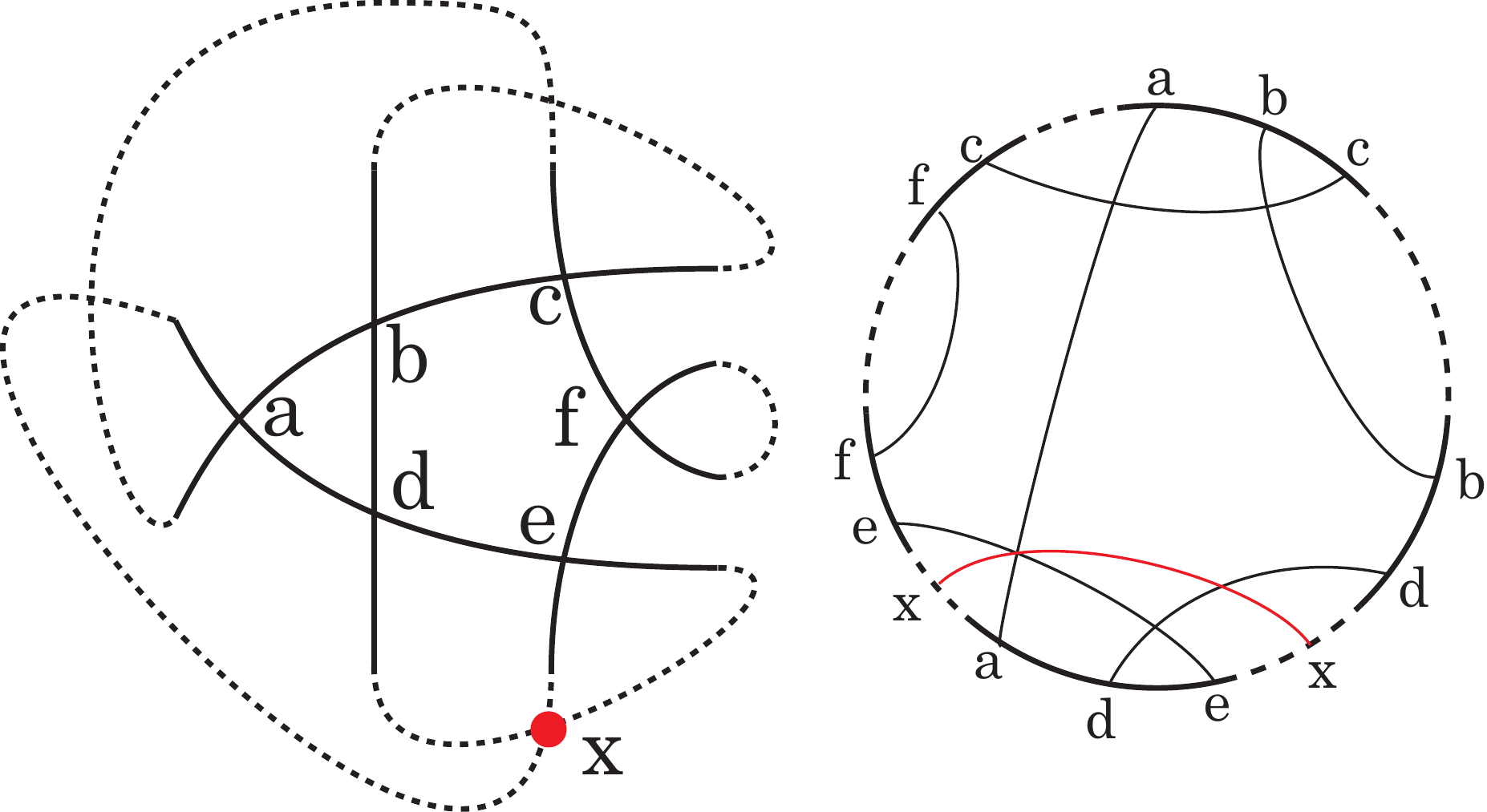} & \\ \hline
\end{tabular}
\end{table}

\noindent $\bullet$ {\bf{Cases 9--18.}}  If arc number $2$ contains both arcs DG and FI, we can automatically fix connections (A, B) and (H, J) (Table~\ref{4a}, Cases 9--16).  Table \ref{splittingCase9--16} shows how arcs connect, considering all possibilities.  
\begin{table}
\caption{Method to split into Cases 9--16.}\label{splittingCase9--16}
\begin{tabular}{|c|c|} \hline
  $($C$,~ $D$) \begin{cases}
    ($G$,~ $F$) ($I$,~ $E$) & ($Case$~9) \\
    ($G$,~ $I$) ($F$,~ $E$) & ($Case$~10) 
  \end{cases}$
&
$ ($C$,~ $F$) \begin{cases}
    ($I$,~ $D$) ($G$,~ $E$) & ($Case$~11) \\
    ($I$,~ $G$) ($D$,~ $E$) & ($Case$~12)
  \end{cases}
$ \\ \hline
$ ($C$,~ $G$) \begin{cases}
    ($D$,~ $F$) ($I$,~ $E$) & ($Case$~13) \\
    ($D$,~ $I$) ($F$,~ $E$) & ($Case$~14)
  \end{cases}$
&
$($C$,~ $I$) \begin{cases}
    ($F$,~ $D$) ($G$,~ $E$) & ($Case$~15) \\
    ($F$,~ $G$) ($D$,~ $E$) & ($Case$~16) 
  \end{cases}$ \\ \hline
  \end{tabular}
  \end{table}
Except for Cases 10 and 16, the existence of a triple chord is directly proved by Table \ref{cases9--16}.   
\begin{table}
\caption{Cases easily proved: Case 9, Cases 11--15.  Non-easy cases: Case 10 and its additional figure Case 10a, Case 16.}\label{cases9--16}
\begin{tabular}{|c|c|c|} \hline
Case 9 & Case 11 & Case 12  \\ \hline
\includegraphics[width=3.5cm]{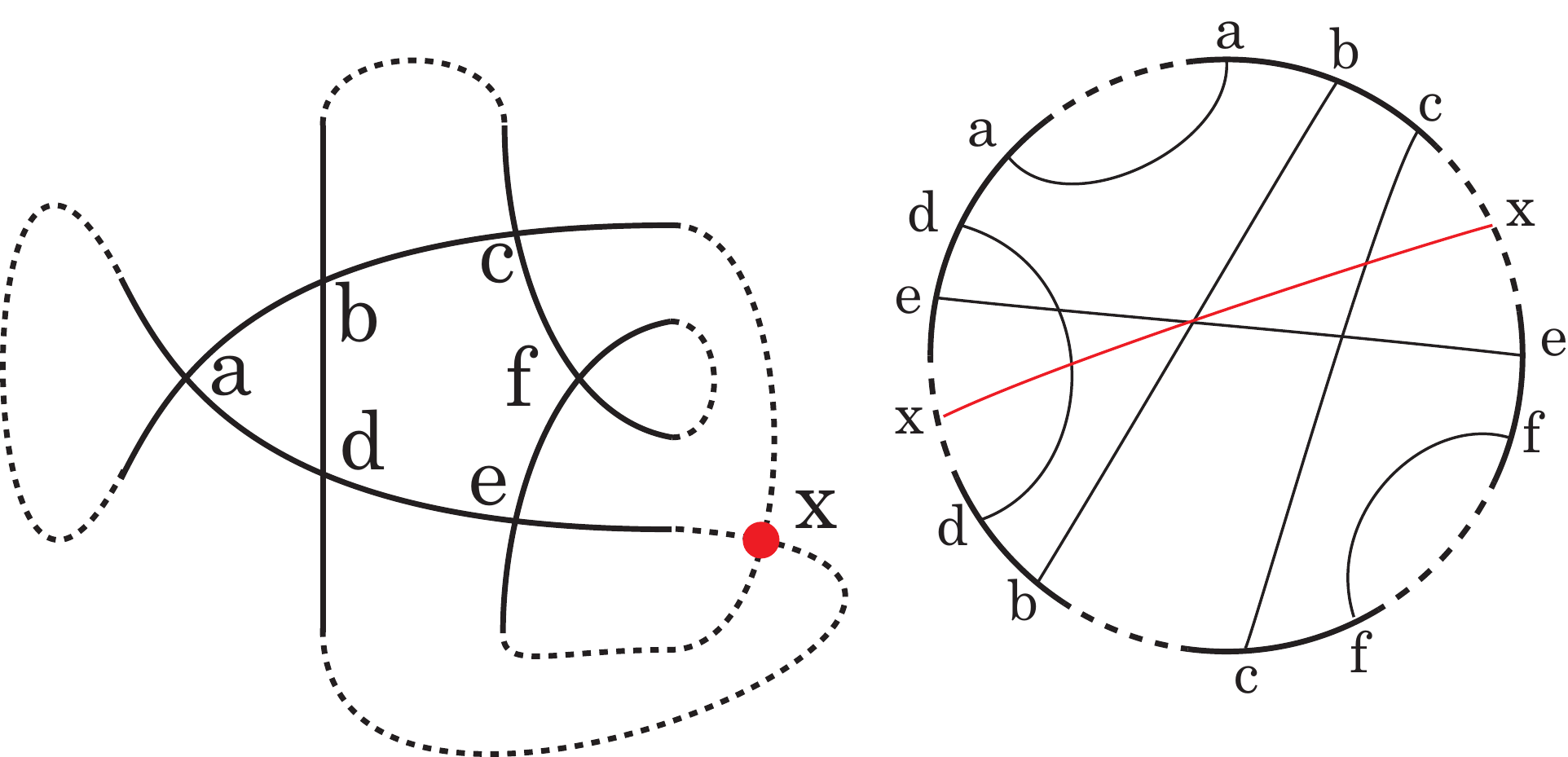} & \includegraphics[width=3.5cm]{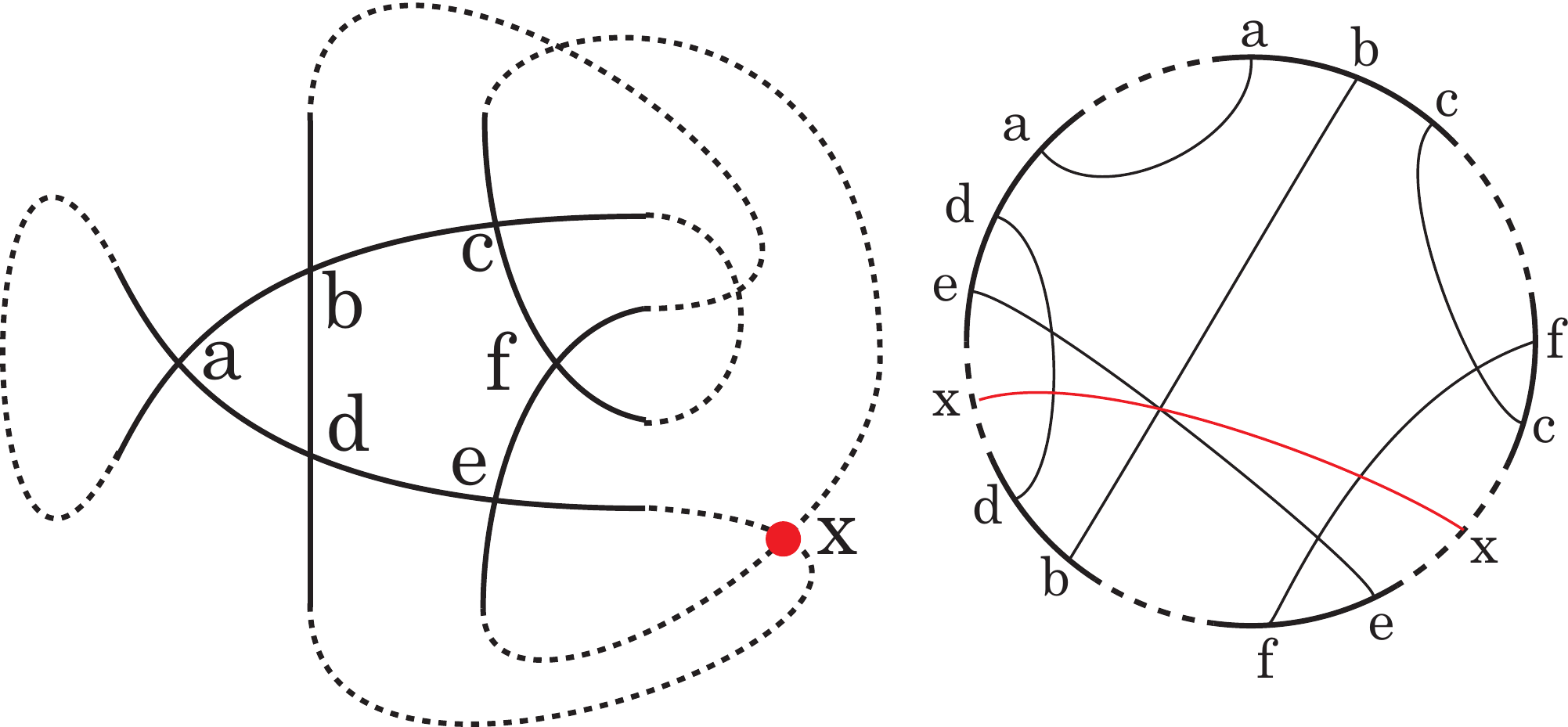} & \includegraphics[width=3.5cm]{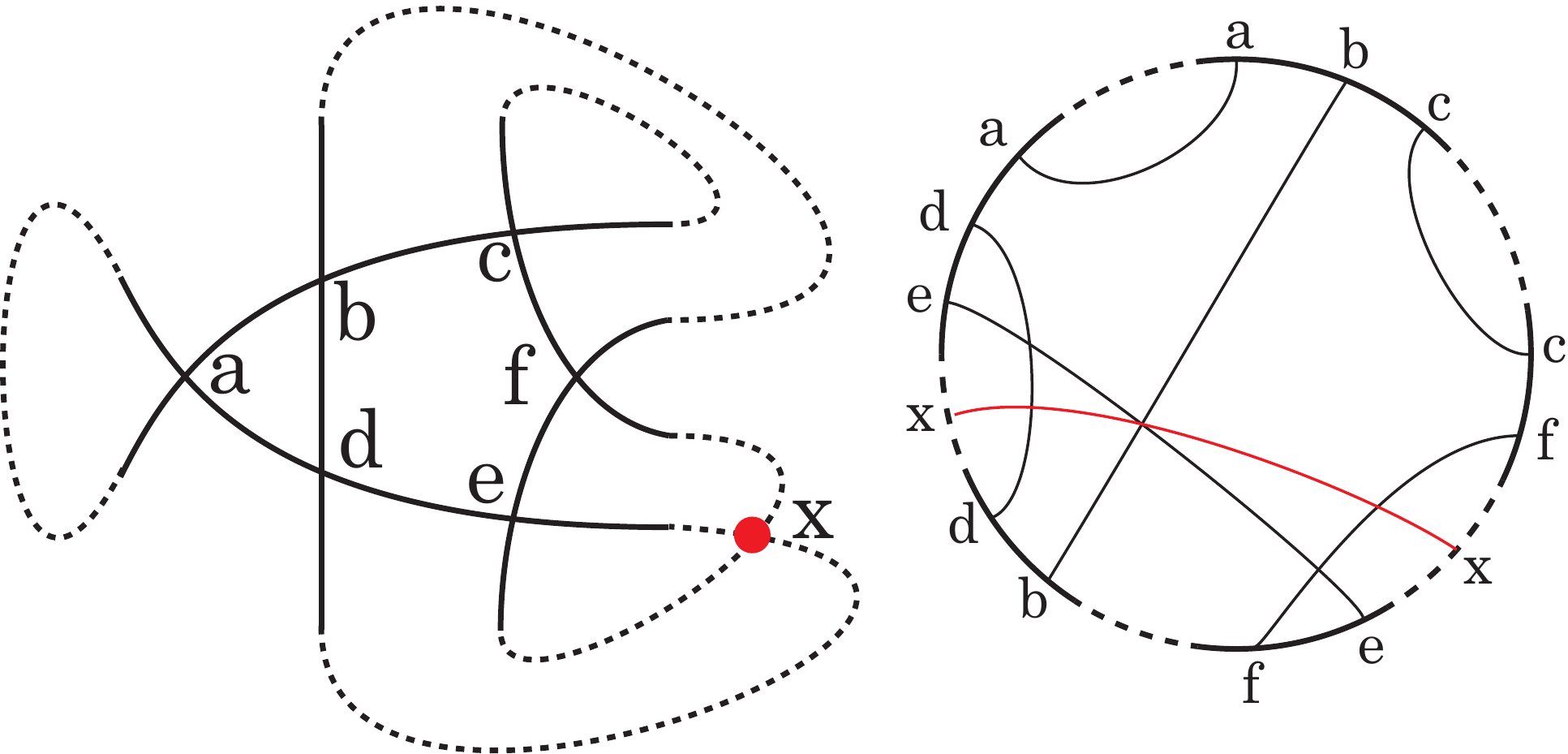} \\ \hline
Case 13 & Case 14 & Case 15 \\ \hline
\includegraphics[width=3.5cm]{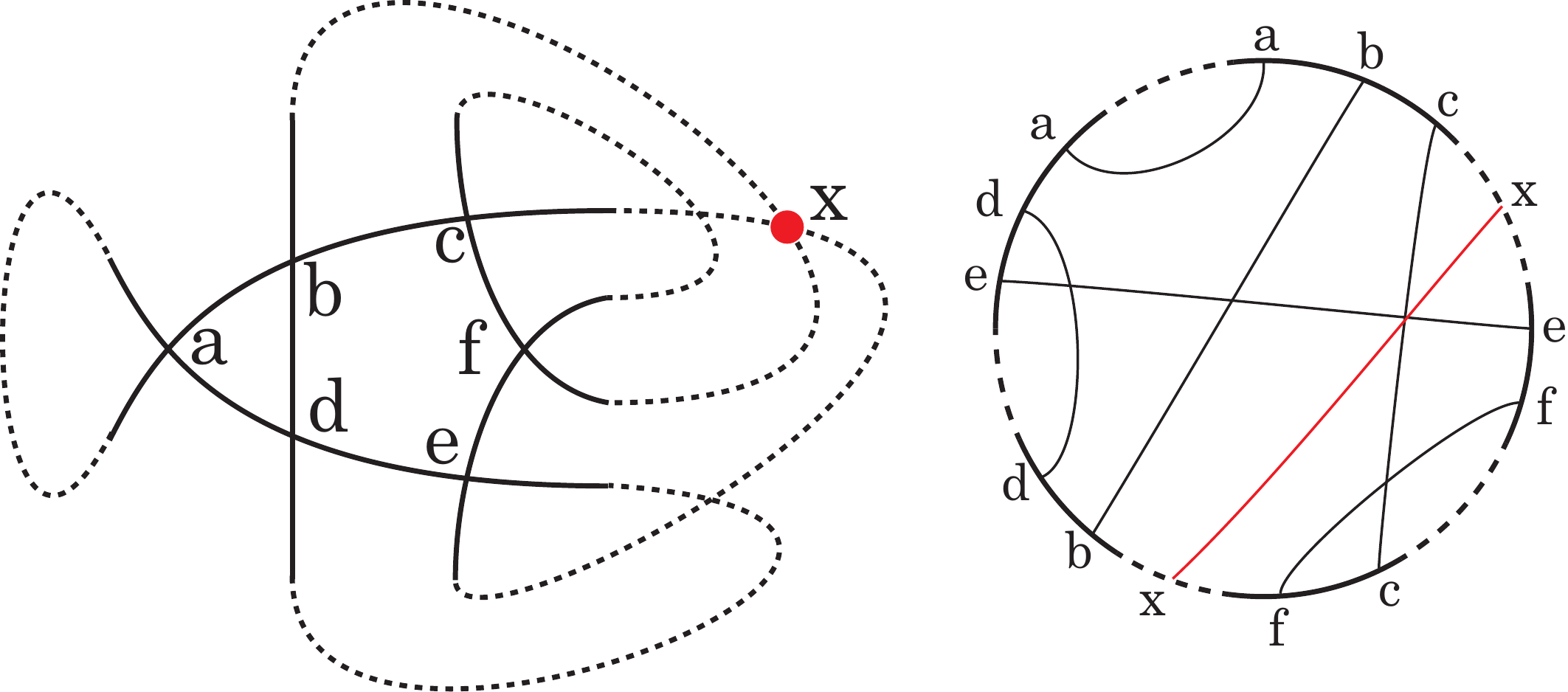} & \includegraphics[width=3.5cm]{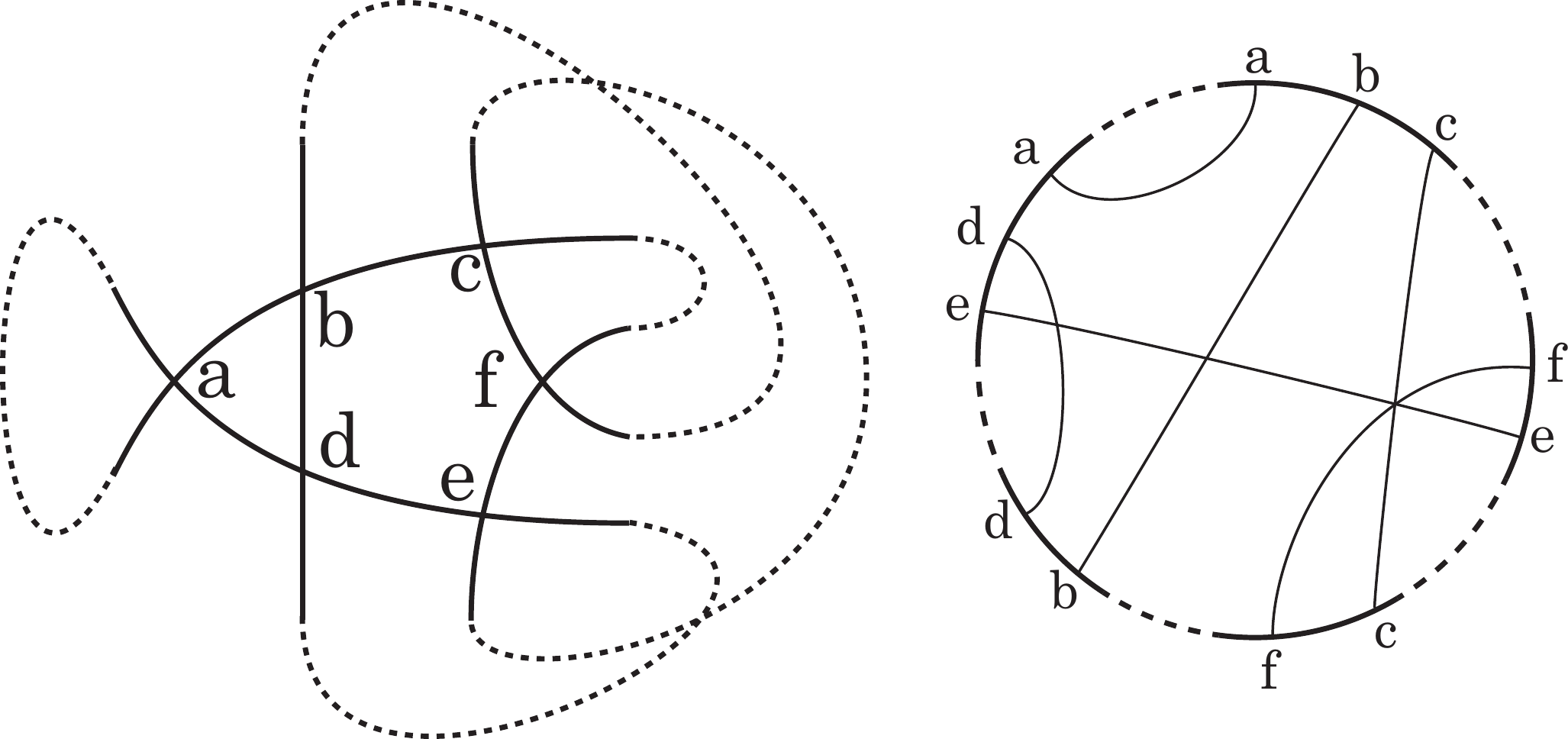} & \includegraphics[width=3.5cm]{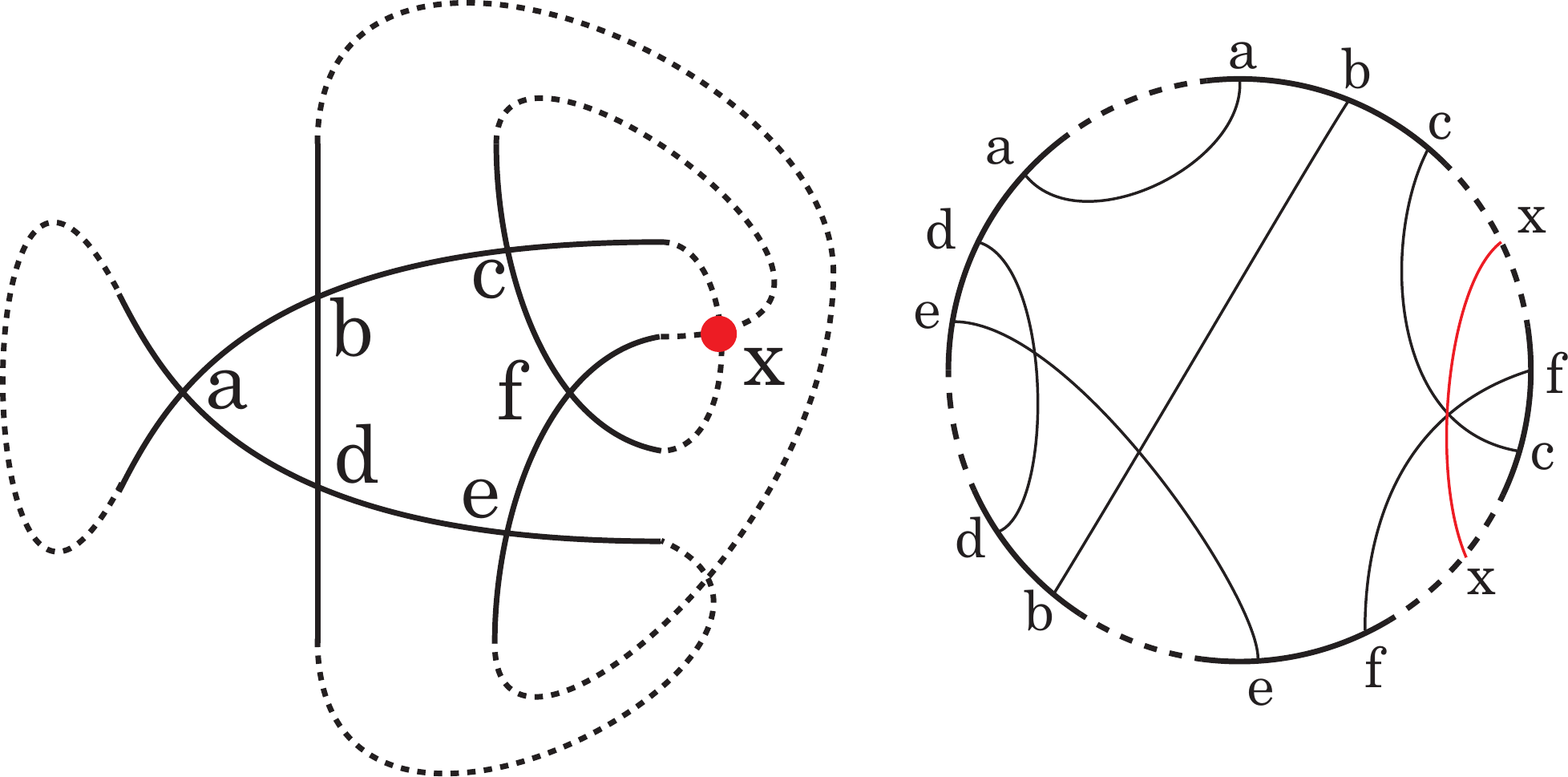} \\ \hline \hline
Case 10 & Case 10a & Case 16 \\ \hline
\includegraphics[width=3.5cm]{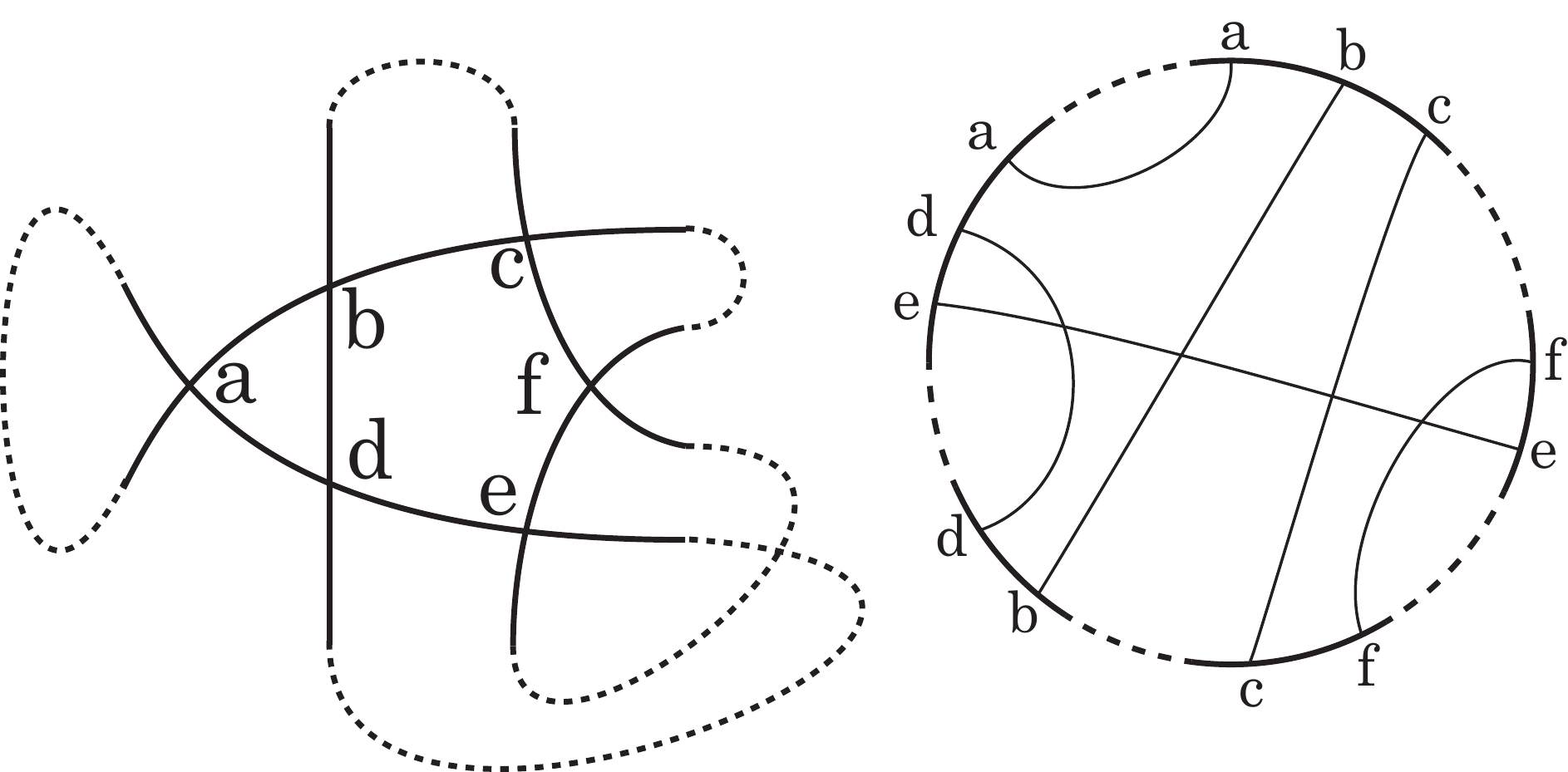} & \includegraphics[width=3.5cm]{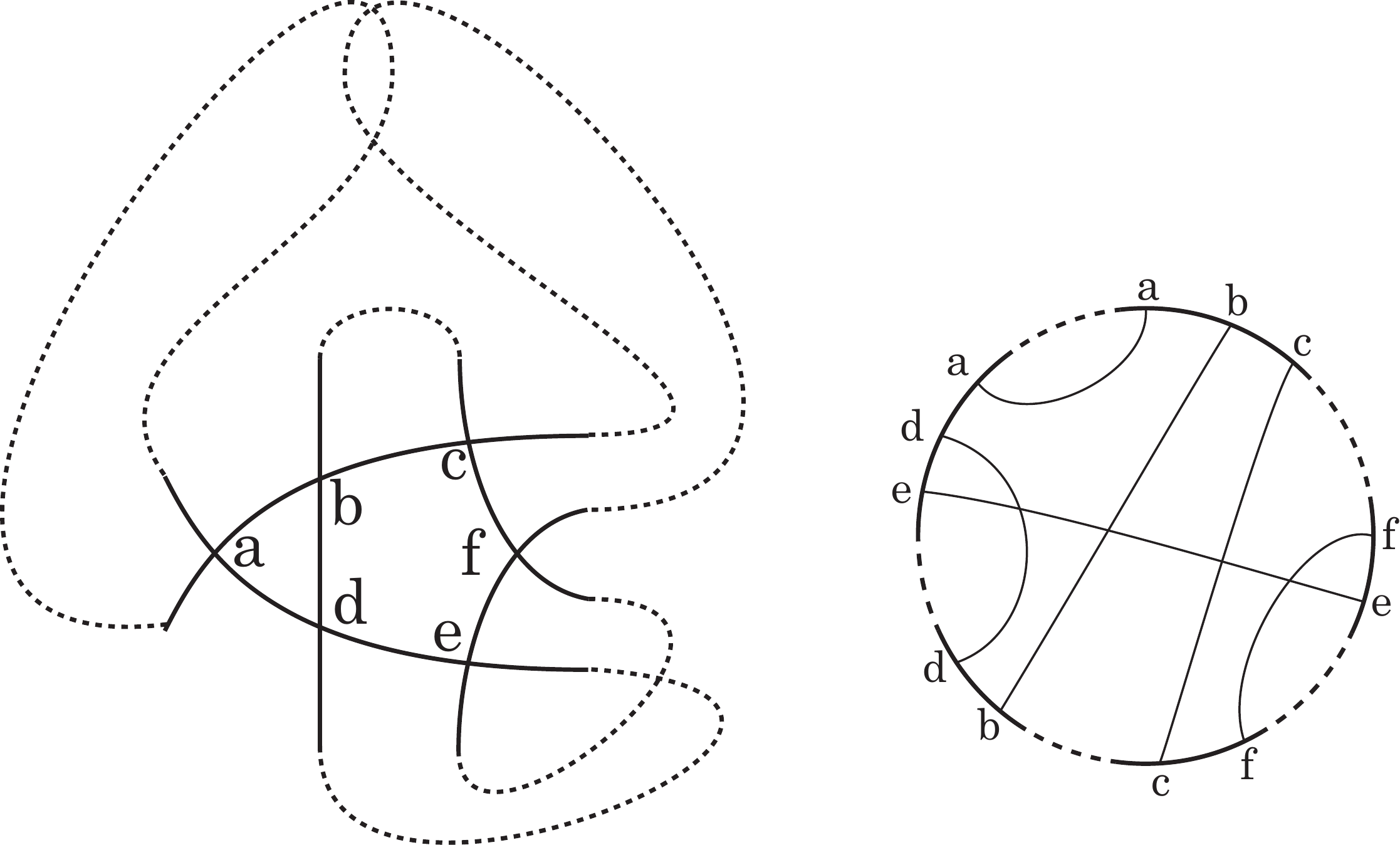} & \includegraphics[width=3.5cm]{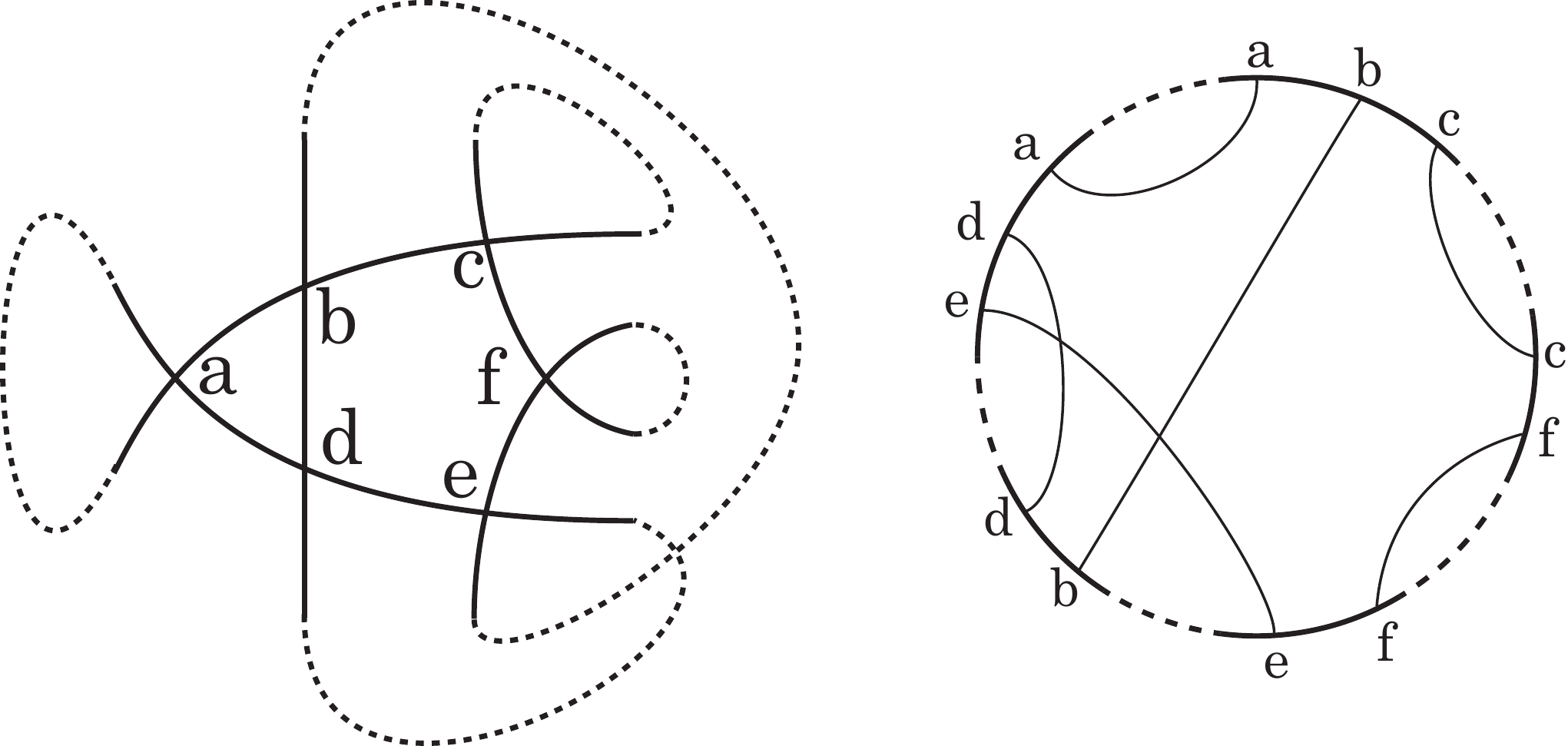} \\ \hline
\end{tabular}
\end{table}

\noindent $\bullet$ {\bf{Case 10}} (not easily proved). Observe the figure in Case 10 on the bottom line of Table \ref{cases9--16}.   
First, this knot projection $P$ is a prime knot projection with no $1$- or $2$-gons.  Thus, $P$ is reduced (Lemma \ref{lem_reduced}).  Therefore, the dotted arc (a$\sim$a) must intersect at least one of the other dotted arcs.  If (a$\sim$a) intersects (b$\sim$c), (d$\sim$e), or (e$\sim$f), then $P$ has a triple chord in $CD_P$.  Therefore, we can assume that (a$\sim$a) intersects (c$\sim$f).  In this case, observe the figure in Case 10a on the bottom line of Table~\ref{cases9--16}.  

Next, since the knot projection we considered is a prime knot projection with no $1$- or $2$-gons, the dotted arc (b$\sim$c) intersects at least one of the other dotted arcs.  
\begin{itemize}
\item If (b$\sim$c) intersects (a$\sim$a) or (c$\sim$f), then $P$ has triple chords in $CD_P$.   
\item If (b$\sim$c) intersects (e$\sim$f), but not (a$\sim$a) or (c$\sim$f), then (e$\sim$f) intersects (a$\sim$a) or (c$\sim$f).  However, in each of these two cases, $P$ has a triple chord in $CD_P$.  
\item If (b$\sim$c) intersects (d$\sim$e), but not (a$\sim$a) or (c$\sim$f), then (d$\sim$e) intersects (a$\sim$a) or (c$\sim$f).  However, in each of these two cases, $P$ has a triple chord in $CD_P$.  
\end{itemize}
Therefore, when (b$\sim$c) intersects another dotted arc, a knot projection $P$ that we considered has a triple chord in $CD_P$.  

\noindent $\bullet$ {\bf{Case 16}} (not easily proved). Observe the right-bottom figure of Table~\ref{cases9--16}. 
The existence of a triple chord this case is proved in the same way as Case D, by omitting dotted $1$-gons (f$\sim$f).  Compare Fig.~\ref{cased} with Case 16 in Table \ref{cases9--16}.  

\noindent $\bullet$ {\bf{Cases 17--24.}}  Dotted arc numbers $1$ and $2$ each contains an instance of DG and FI.  This case implies fixing (H, J), i.e.,~H must connect with J (Table~\ref{4a}, Cases 17--24).  In this case, Table \ref{table17--24} shows how the case is split into eight cases, and it is easy to show that each knot projection $P$ of those cases has a triple chord in $CD_P$.  See Table \ref{17-24}.  
\begin{table}
\caption{Method to split into Cases 17--24.}\label{table17--24}
\begin{tabular}{|c|c|c|} \hline
$
  ($B$,~ $D$) ($G$,~ $A$) \begin{cases}
    ($C$,~ $F$) ($I$,~ $E$) & ($Case$~17) \\
    ($C$,~ $I$) ($F$,~ $E$) & ($Case$~18)
  \end{cases}
$
&
$
  ($B$,~ $G$) ($D$,~ $A$) \begin{cases}
    ($C$,~ $F$) ($I$,~ $E$) & ($Case$~19) \\
    ($C$,~ $I$) ($F$,~ $E$) & ($Case$~20)
  \end{cases}
$
\\ \hline
$
  ($B$,~ $F$) ($I$,~ $A$) \begin{cases}
    ($C$,~ $D$) ($G$,~ $E$) & ($Case$~21) \\
    ($C$,~ $G$) ($D$,~ $E$) & ($Case$~22)
  \end{cases}
$
&
$
  ($B$,~ $I$) ($F$,~ $A$) \begin{cases}
    ($C$,~ $D$) ($G$,~ $E$) & ($Case$~23) \\
    ($C$,~ $G$) ($D$,~ $E$) & ($Case$~24)
  \end{cases}$ \\ \hline
  \end{tabular}
\end{table}

\begin{table}
\caption{Cases easily proved.}\label{17-24}
\begin{tabular}{|c|c|c|} \hline
Case 17 & Case 18 & Case 19 \\ \hline
\includegraphics[width=3.5cm]{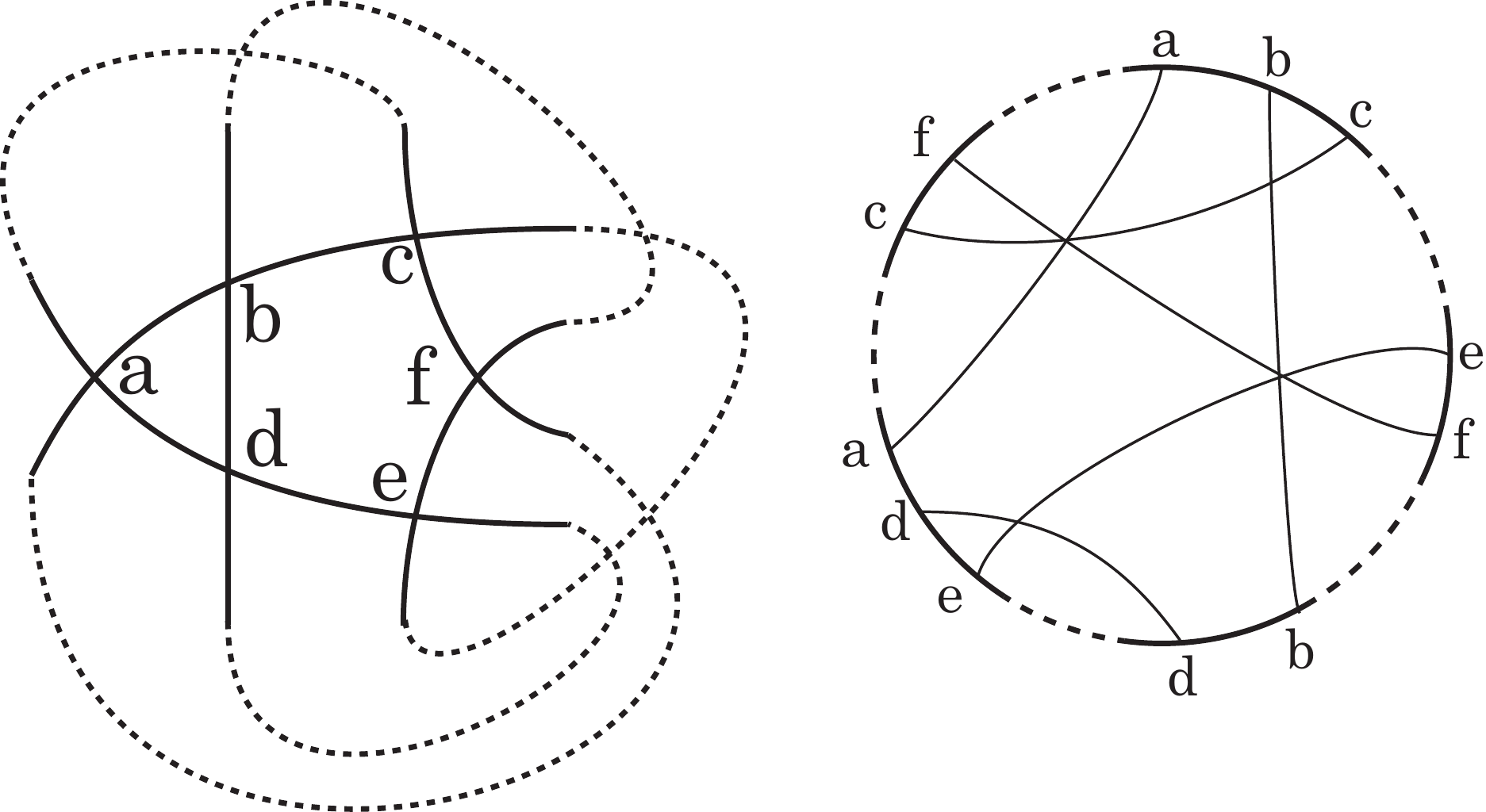}
&\includegraphics[width=3.5cm]{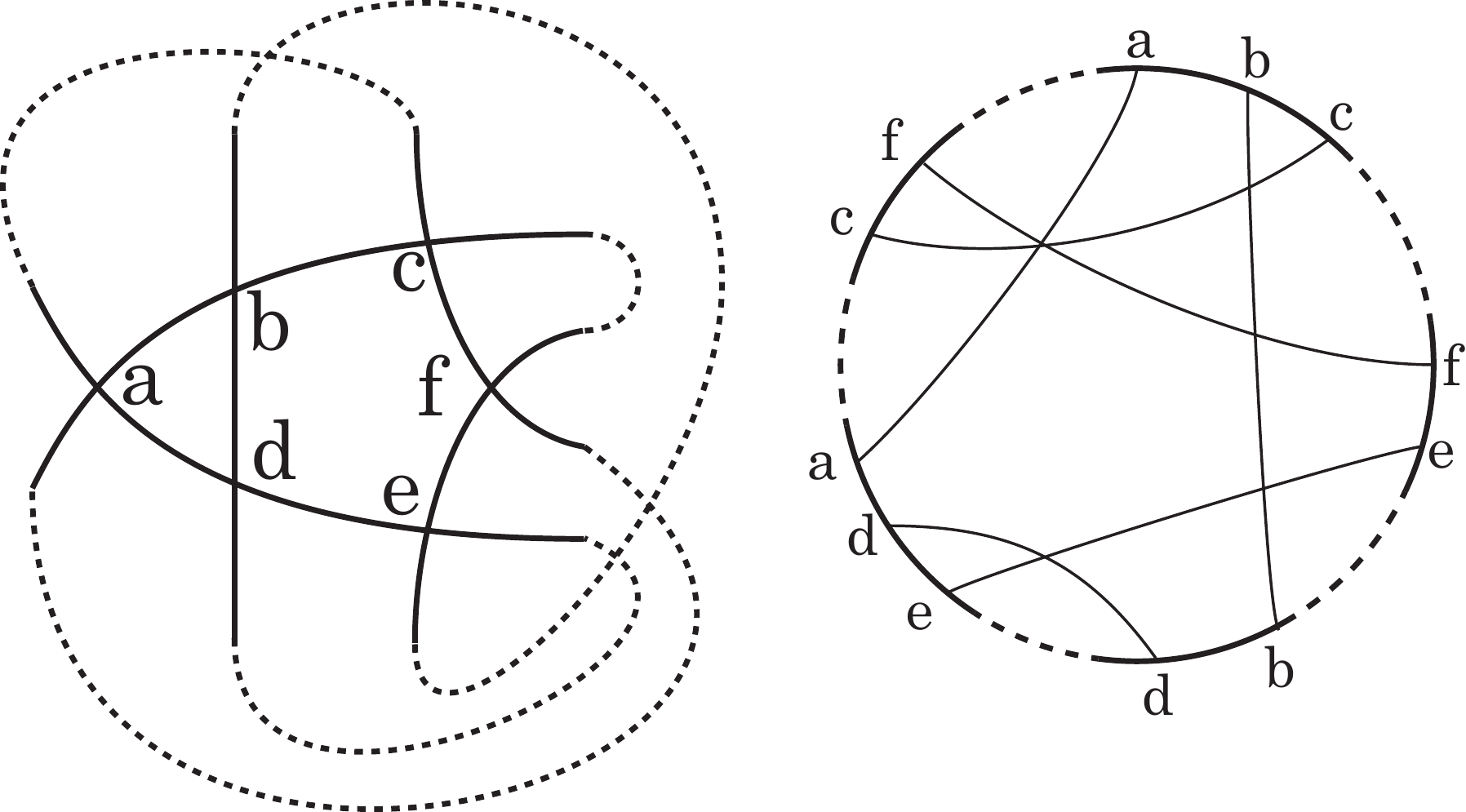}
&\includegraphics[width=3.5cm]{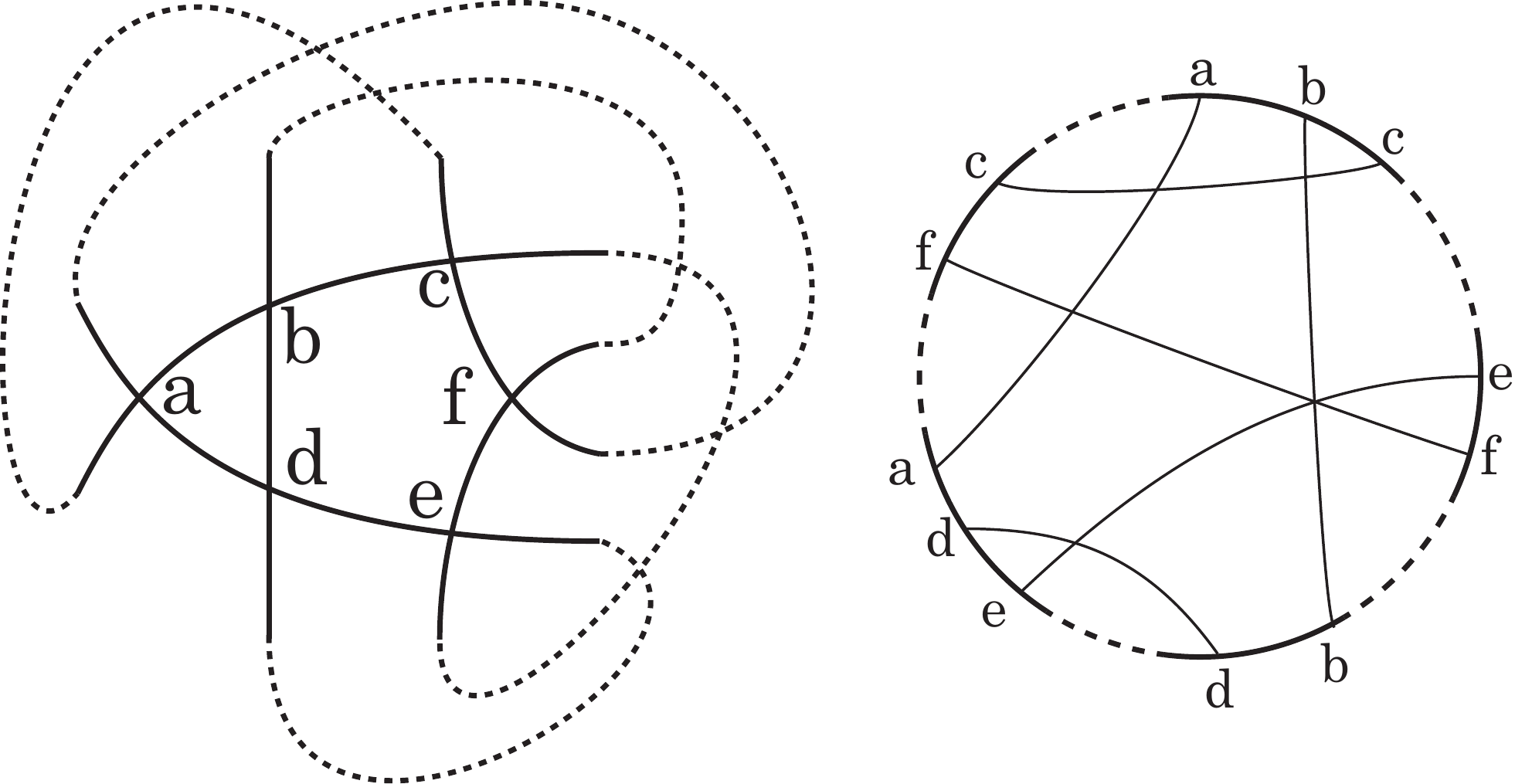}\\ \hline
Case 20& Case 21 & Case 22 \\ \hline
\includegraphics[width=3.5cm]{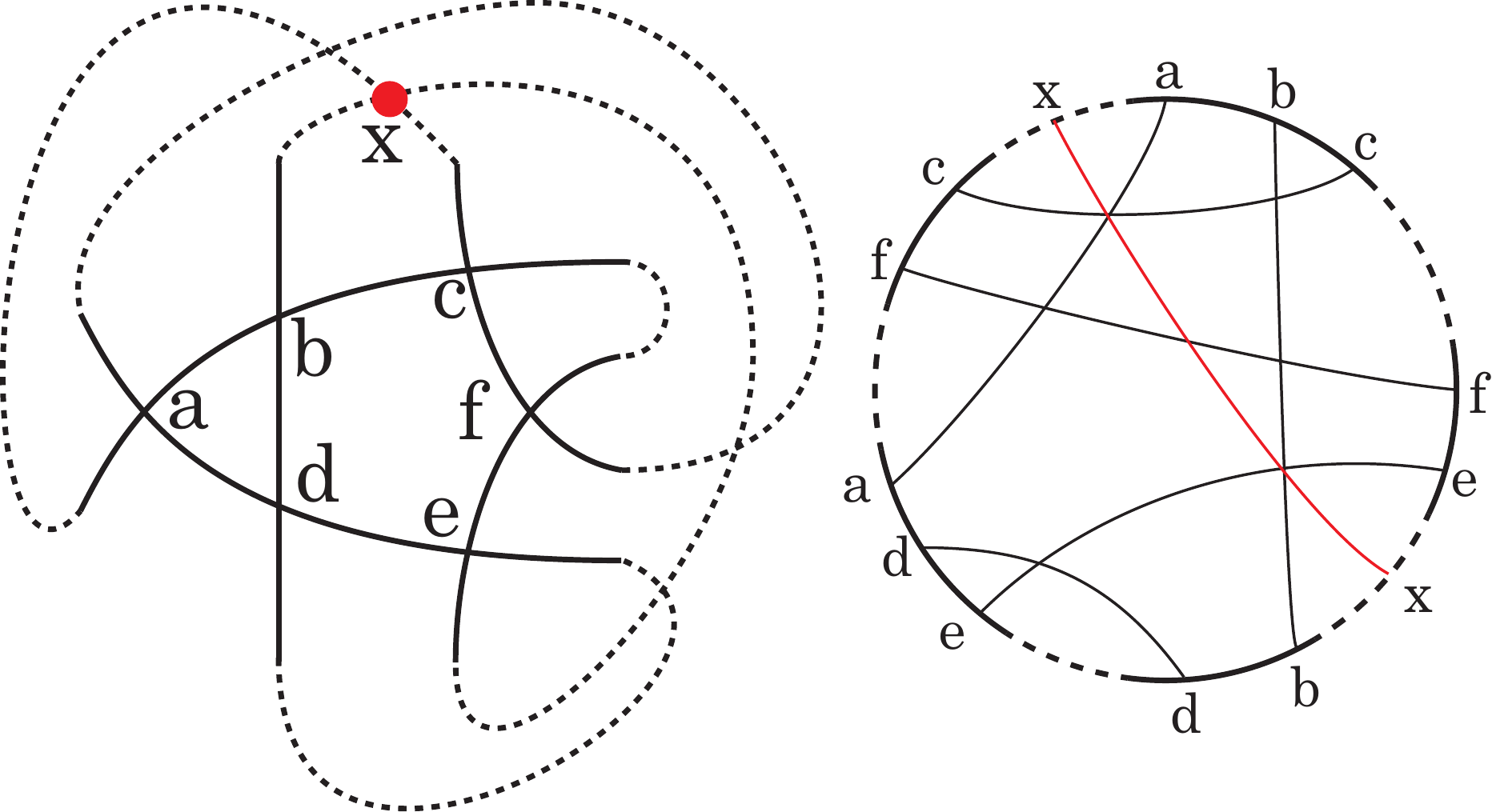}
&\includegraphics[width=3.5cm]{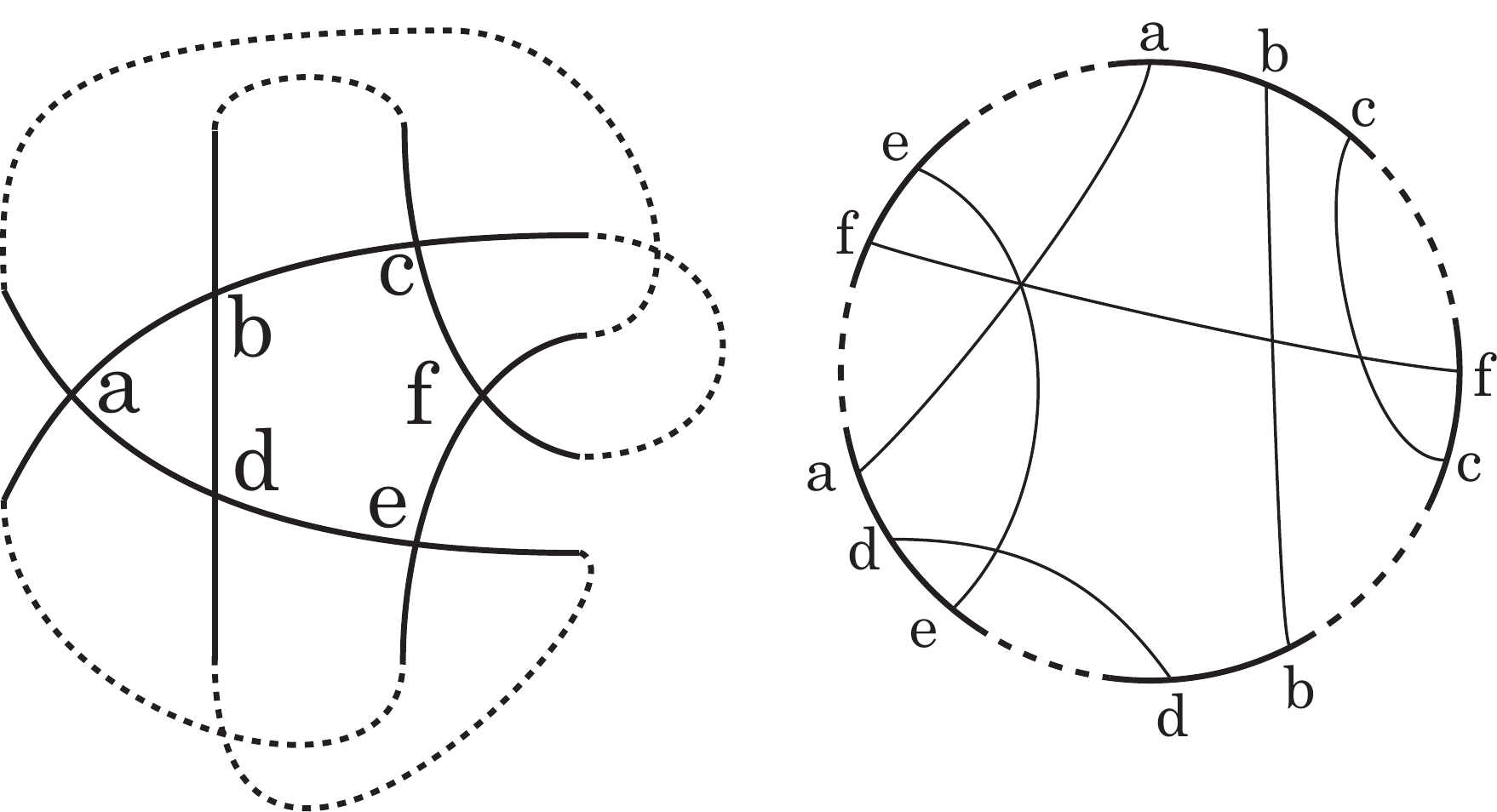}
&\includegraphics[width=3.5cm]{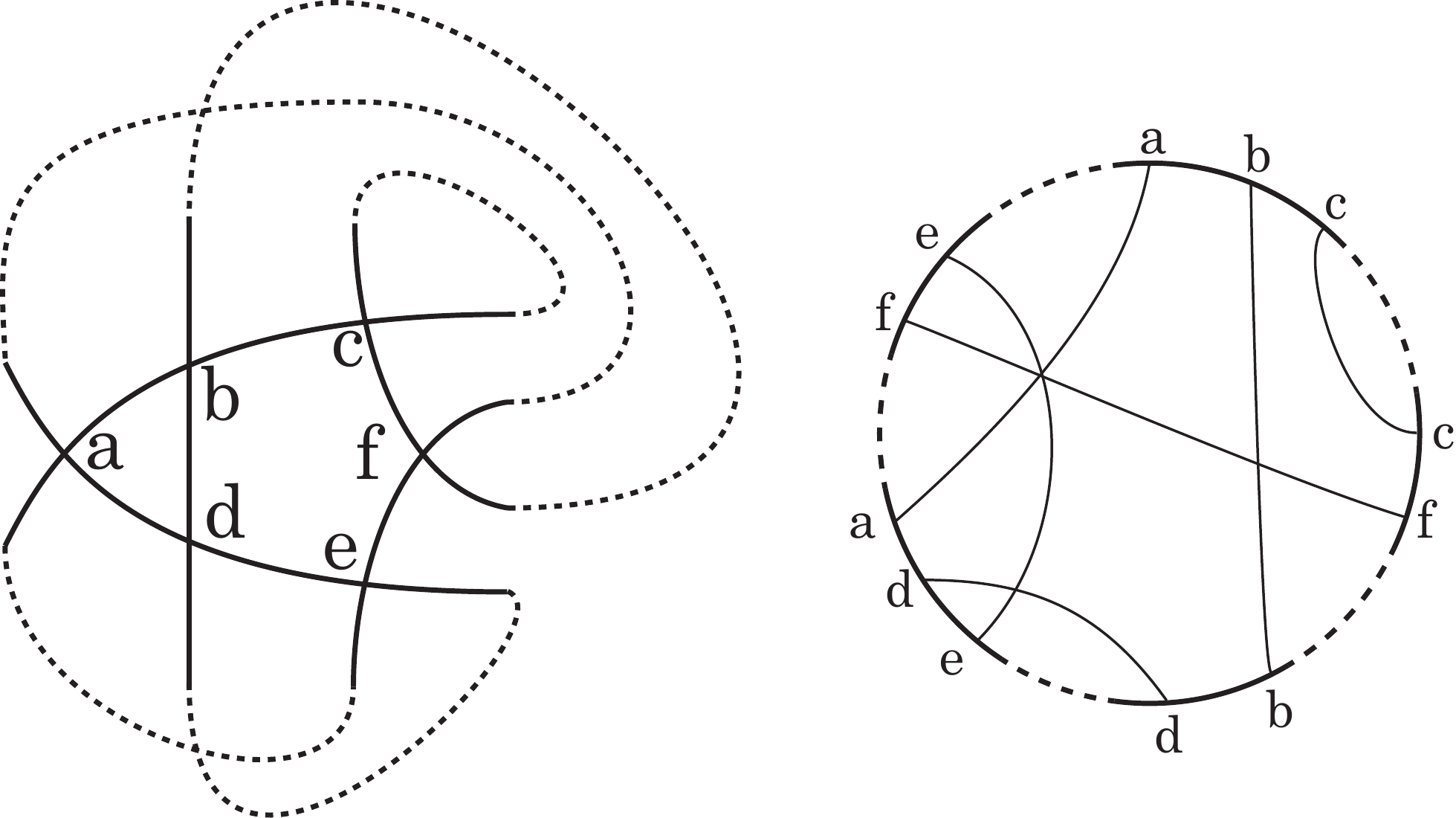}\\ \hline
Case 23 & Case 24 &  \\ \hline
\includegraphics[width=3.5cm]{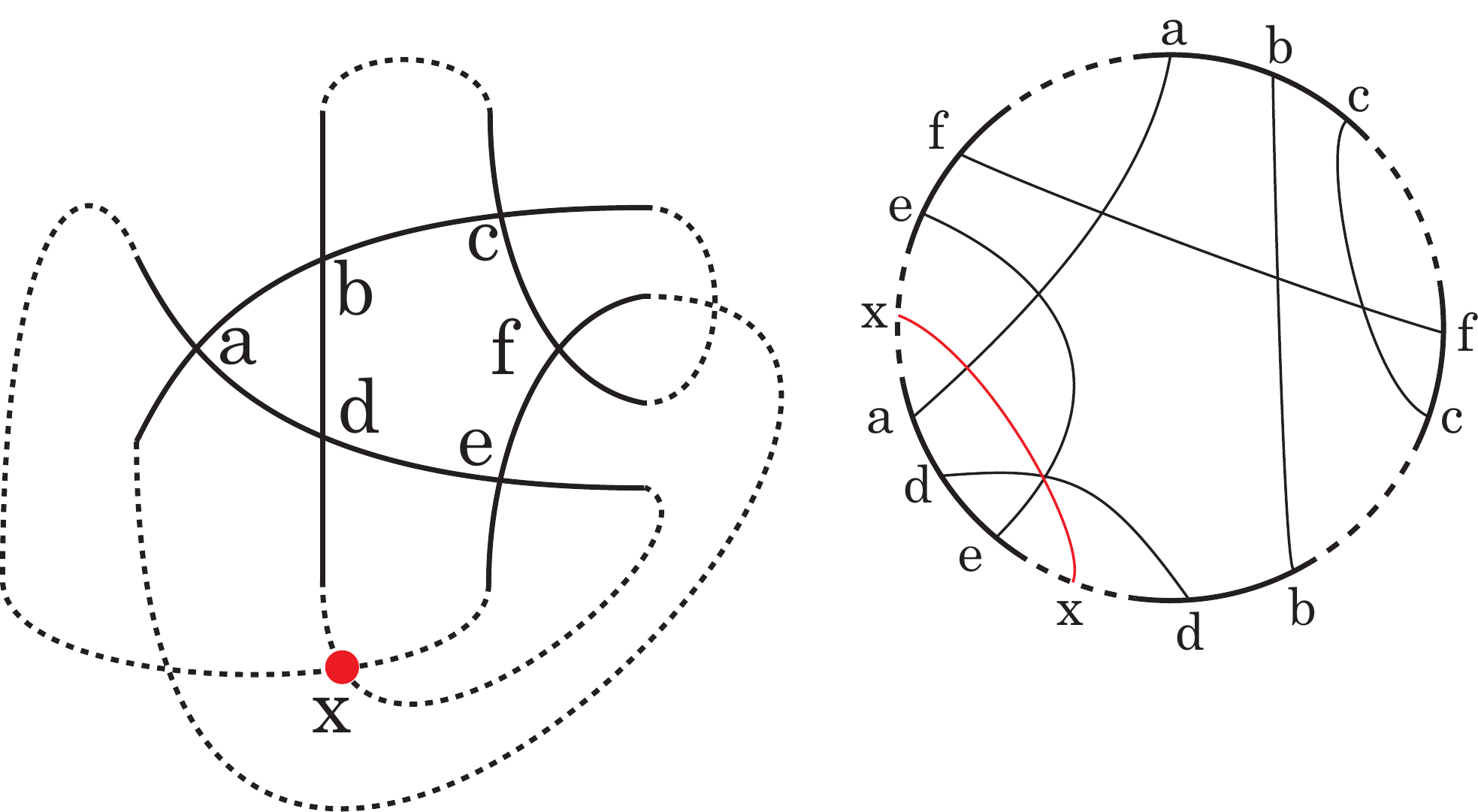}
&\includegraphics[width=3.5cm]{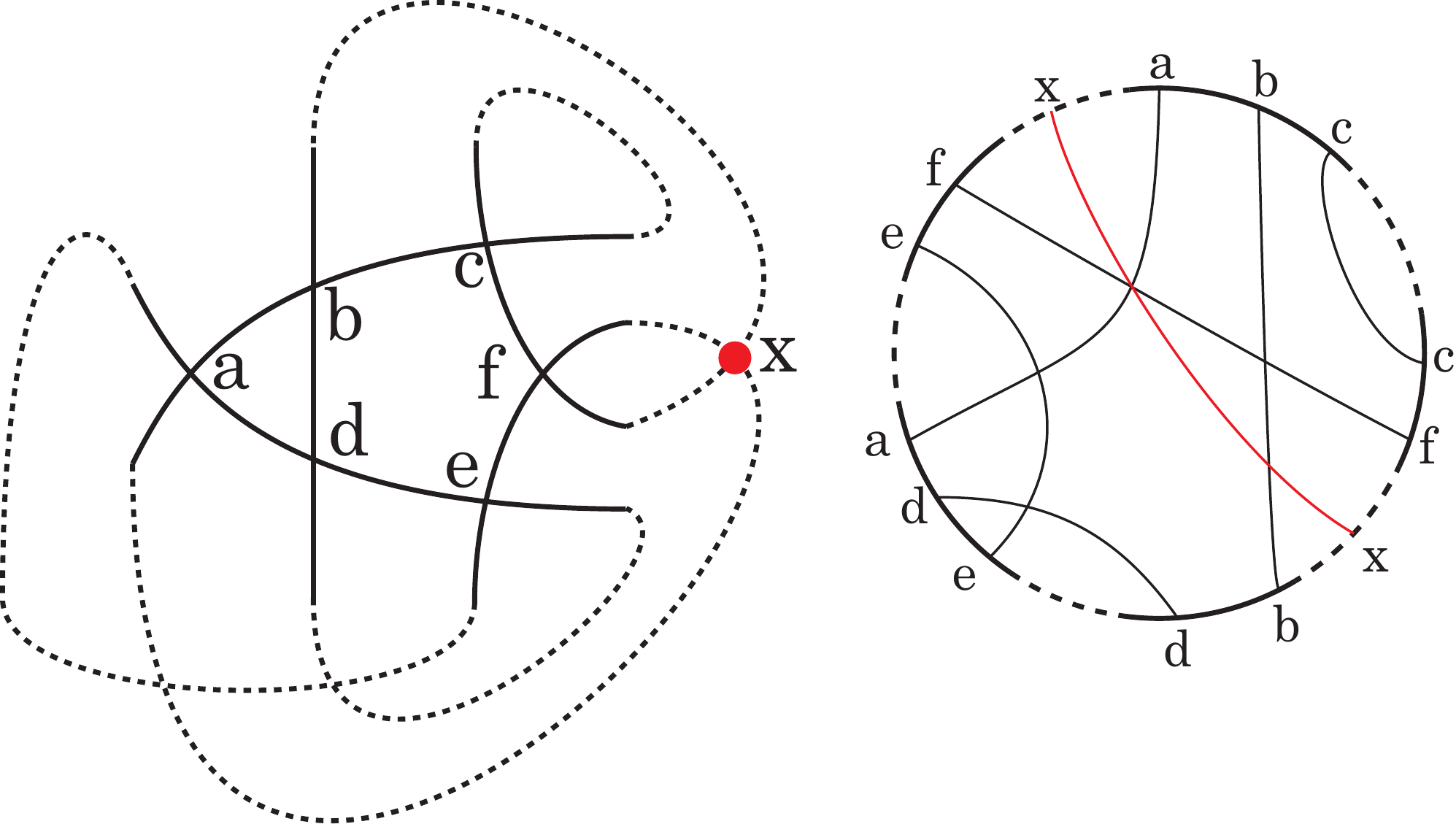}
& \\ \hline
\end{tabular}
\end{table}

\noindent $\bullet$ {\bf{Cases 25--32.}}  Dotted arc numbers 2 and 3 each contains exactly one of the non-dotted arcs.  Hence, we can automatically fix the connection (A, B) (Table~\ref{4a}, Cases 25--32).  Table \ref{splittingCases25--32} shows how arcs connect considering all possibilities.  
\begin{table}
\caption{Method to split into Cases 25--32.}\label{splittingCases25--32}
\begin{tabular}{|c|c|} \hline
  $($C$,~ $D$) ($G$,~ $E$) \begin{cases}
    ($H$,~ $F$) ($I$,~ $J$) & ($Case$~25.) \\
    ($H$,~ $I$) ($F$,~ $J$) & ($Case$~26.)
  \end{cases}
$
&
$
  ($C$,~ $G$) ($D$,~ $E$) \begin{cases}
    ($H$,~ $F$) ($I$,~ $J$) & ($Case$~27.) \\
    ($H$,~ $I$) ($F$,~ $J$) & ($Case$~28.)
  \end{cases}
$
\\ \hline
$
  ($C$,~ $F$) ($I$,~ $E$) \begin{cases}
    ($H$,~ $D$) ($G$,~ $J$) & ($Case$~29.) \\
    ($H$,~ $G$) ($D$,~ $J$) & ($Case$~30.)
  \end{cases}
$
&
$($C$,~ $I$) ($F$,~ $E$) \begin{cases}
    ($H$,~ $D$) ($G$,~ $J$) & ($Case$~31) \\
    ($H$,~ $G$) ($D$,~ $J$) & ($Case$~32)
  \end{cases}$ \\ \hline
\end{tabular}
\end{table}

\noindent $\bullet$ {\bf{Case 25}} (not easily proved). Observe the figure for Case 25 in the lower part of Table~\ref{cases25--32}. 
First, since the considered knot projection $P$ is reduced (cf.~Lemma~\ref{lem_reduced}), the dotted arc (a$\sim$a) intersects one of the other dotted arcs.  If (a$\sim$a) intersects (b$\sim$c) or (d$\sim$e), then $P$ has a triple chord in $CD_P$.  Therefore, we can assume that (a$\sim$a) intersects (c$\sim$f) (the figure Case 25a of Table~\ref{cases25--32}).  Here, note that the assumption that (a$\sim$a) intersects (e$\sim$f) is equivalent to the assumption that (a$\sim$a) intersects (c$\sim$f) by symmetry, hence we omit the case (a$\sim$a) intersects (e$\sim$f).  


Next, consider Case 25a in Table \ref{cases25--32}.  Since $P$ is a prime knot projection with no $1$- or $2$-gons, (b$\sim$c) must intersect one of the other dotted arcs.  
\begin{itemize}
\item If (b$\sim$c) intersects (a$\sim$a) or (c$\sim$f), 
then $P$ has a triple chord in $CD_P$.   
\item If (b$\sim$c) intersects (e$\sim$f), then $P$ has a triple chord in $CD_P$.  
\item If (b$\sim$c) intersects (d$\sim$e), but not (a$\sim$a) or (c$\sim$f), then (d$\sim$e) must intersect (a$\sim$a) or (c$\sim$f).  However, whether (d$\sim$e) intersects (a$\sim$a) or (c$\sim$f), $P$ has a triple chord in $CD_P$.  
\end{itemize}
Thus, if (b$\sim$c) intersects one of the other dotted arcs, a considered knot projection $P$ has a triple chord in $CD_P$.  

\begin{table}
\caption{Cases easily proved: Cases 26, 27 and Cases 29--32.  Non-easy cases: Case 25 and its additional figure Case 25a, Case 28.}\label{cases25--32}
\begin{tabular}{|c|c|c|} \hline
Case 26 &Case 27& Case 29 \\ \hline
\includegraphics[width=3.5cm]{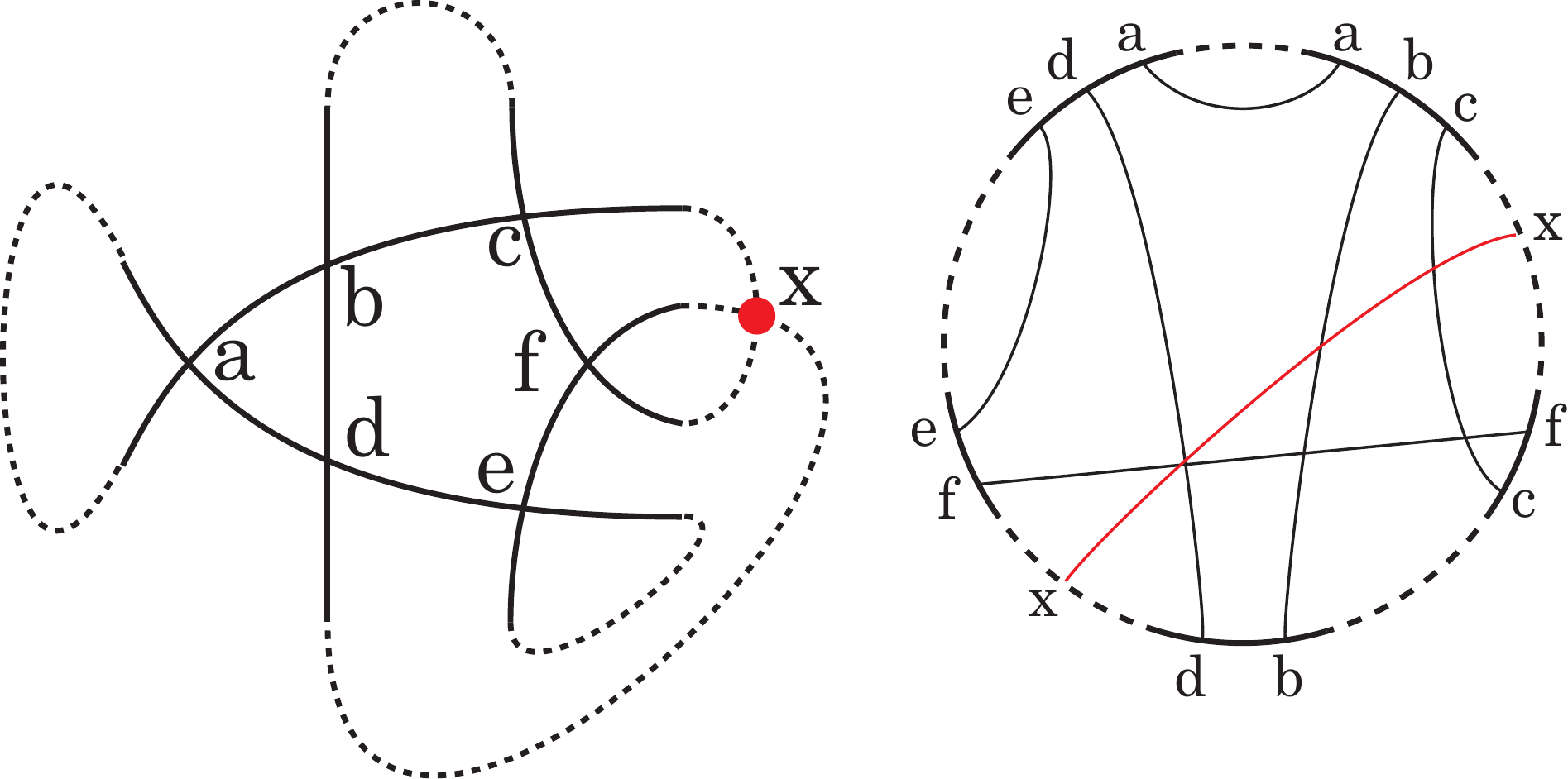} &\includegraphics[width=3.5cm]{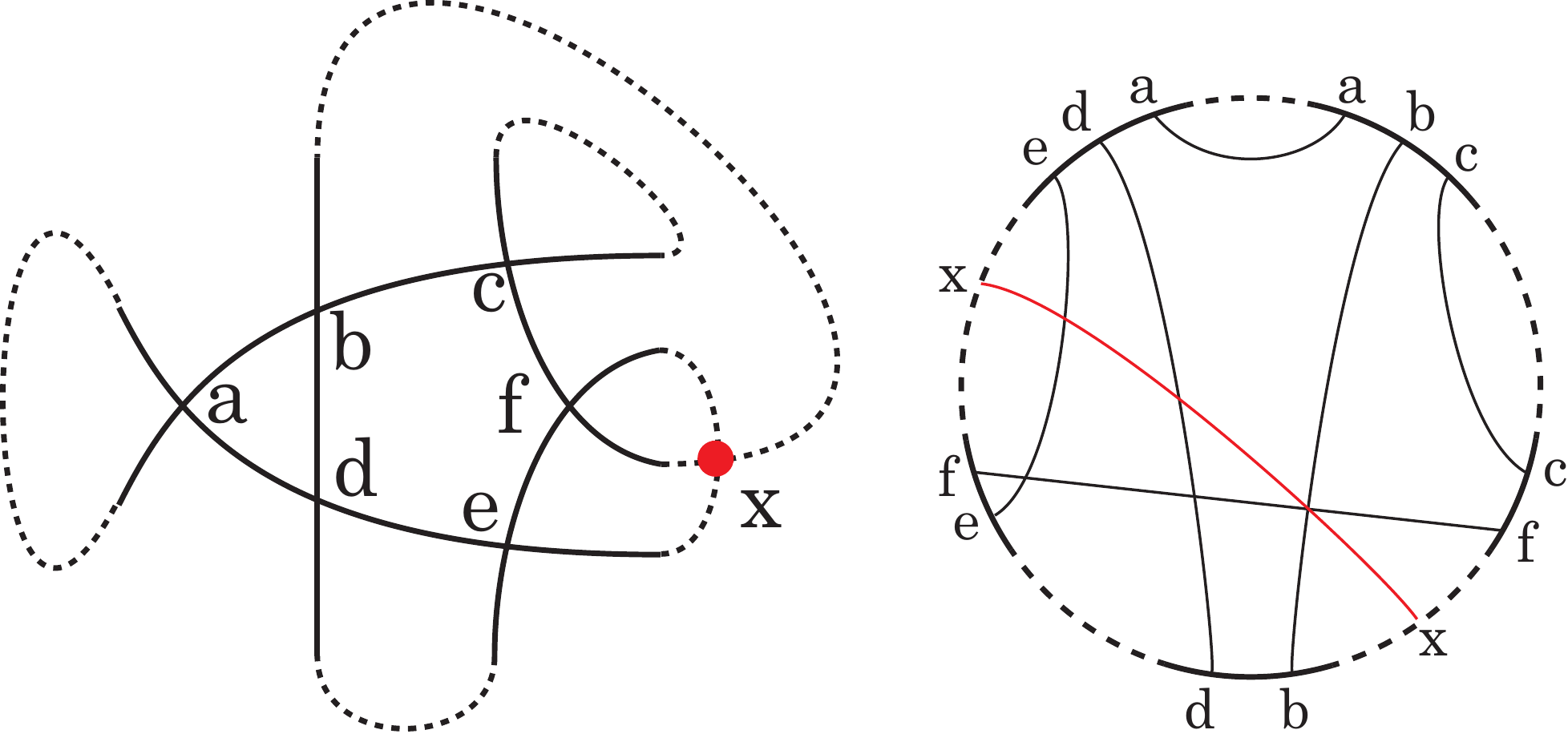}& \includegraphics[width=3.5cm]{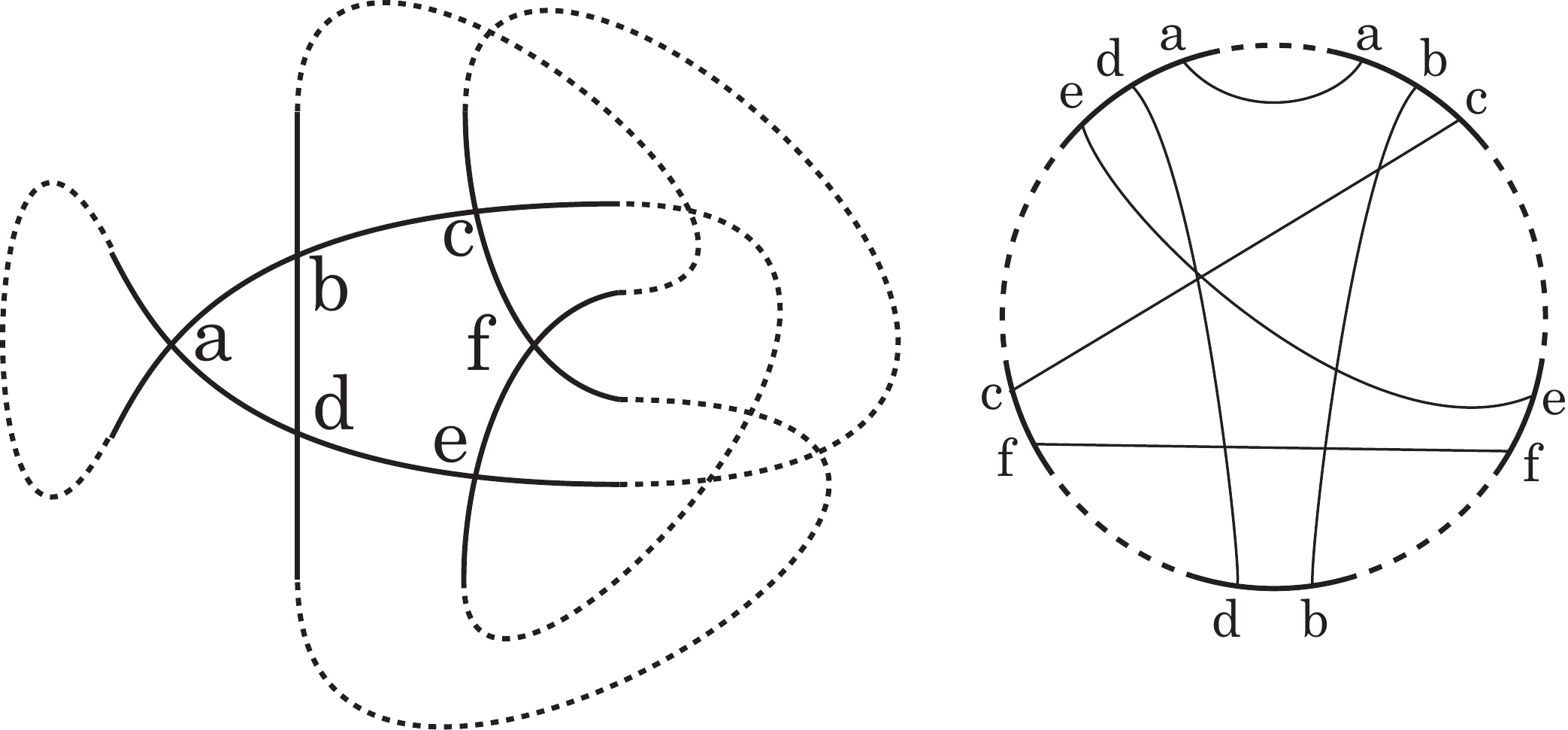} \\ \hline
Case 30 &Case 31 & Case 32 \\ \hline
\includegraphics[width=3.5cm]{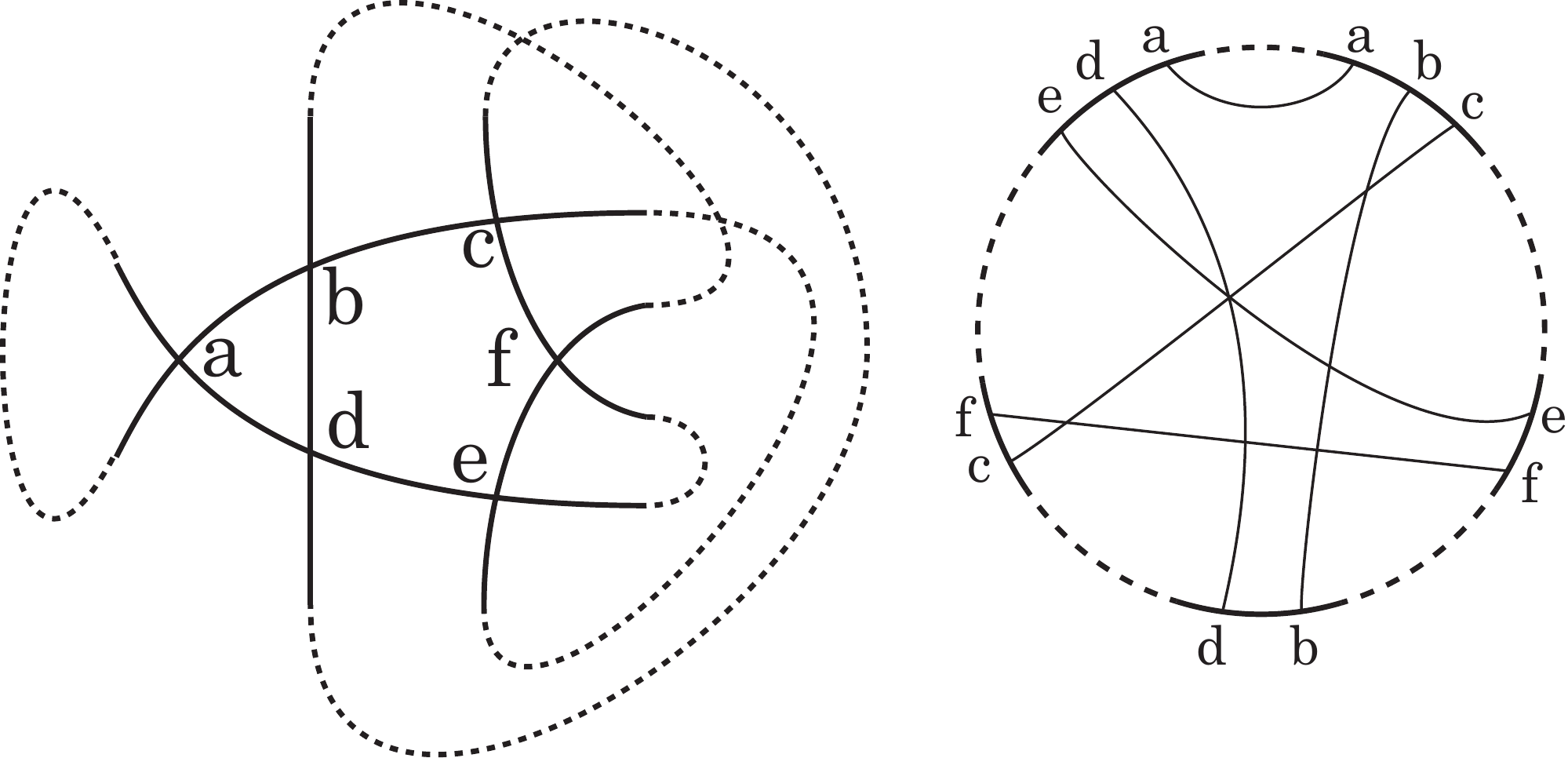} & \includegraphics[width=3.5cm]{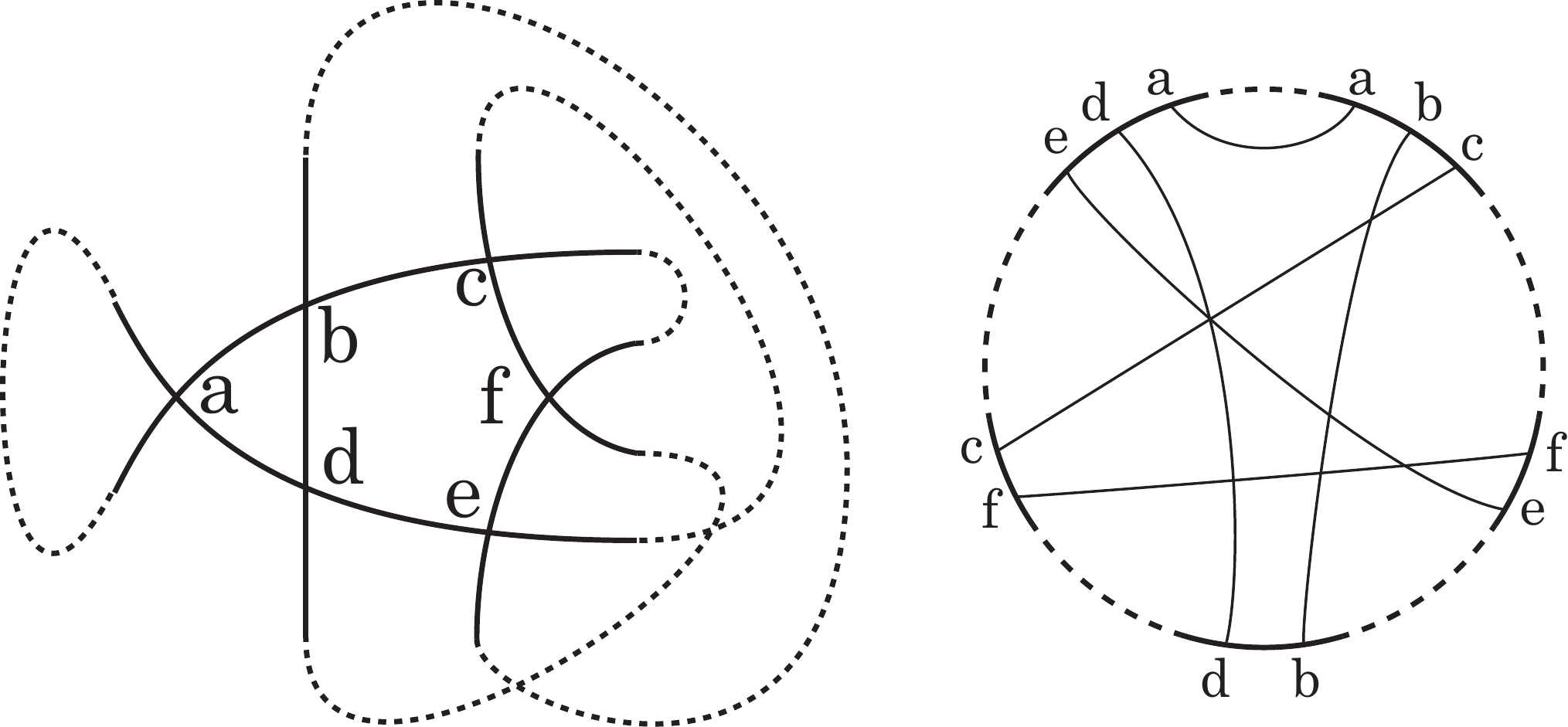} & \includegraphics[width=3.5cm]{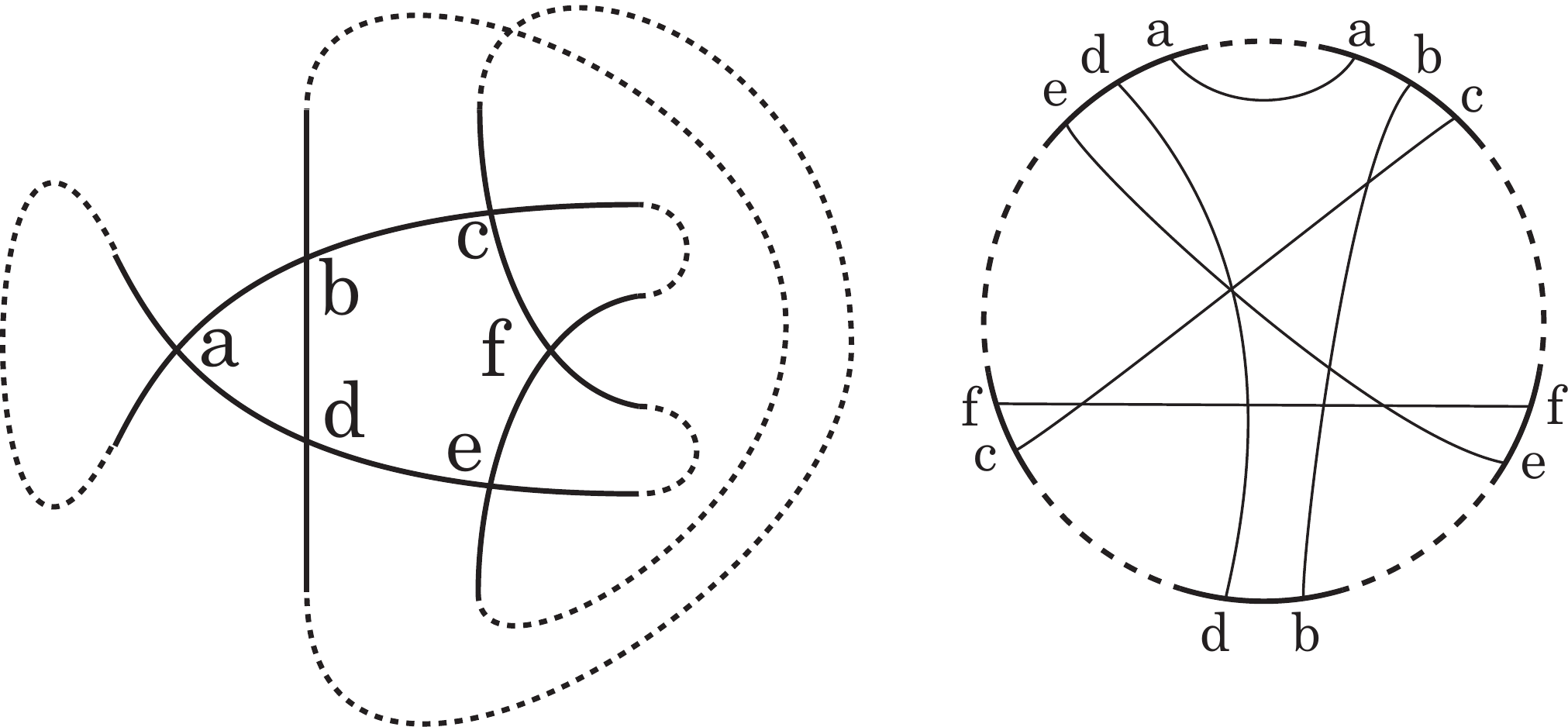} \\ \hline
\end{tabular}
\begin{tabular}{|c|c|} \hline
Case 25 & Case 25a \\ \hline 
\includegraphics[width=5.4cm]{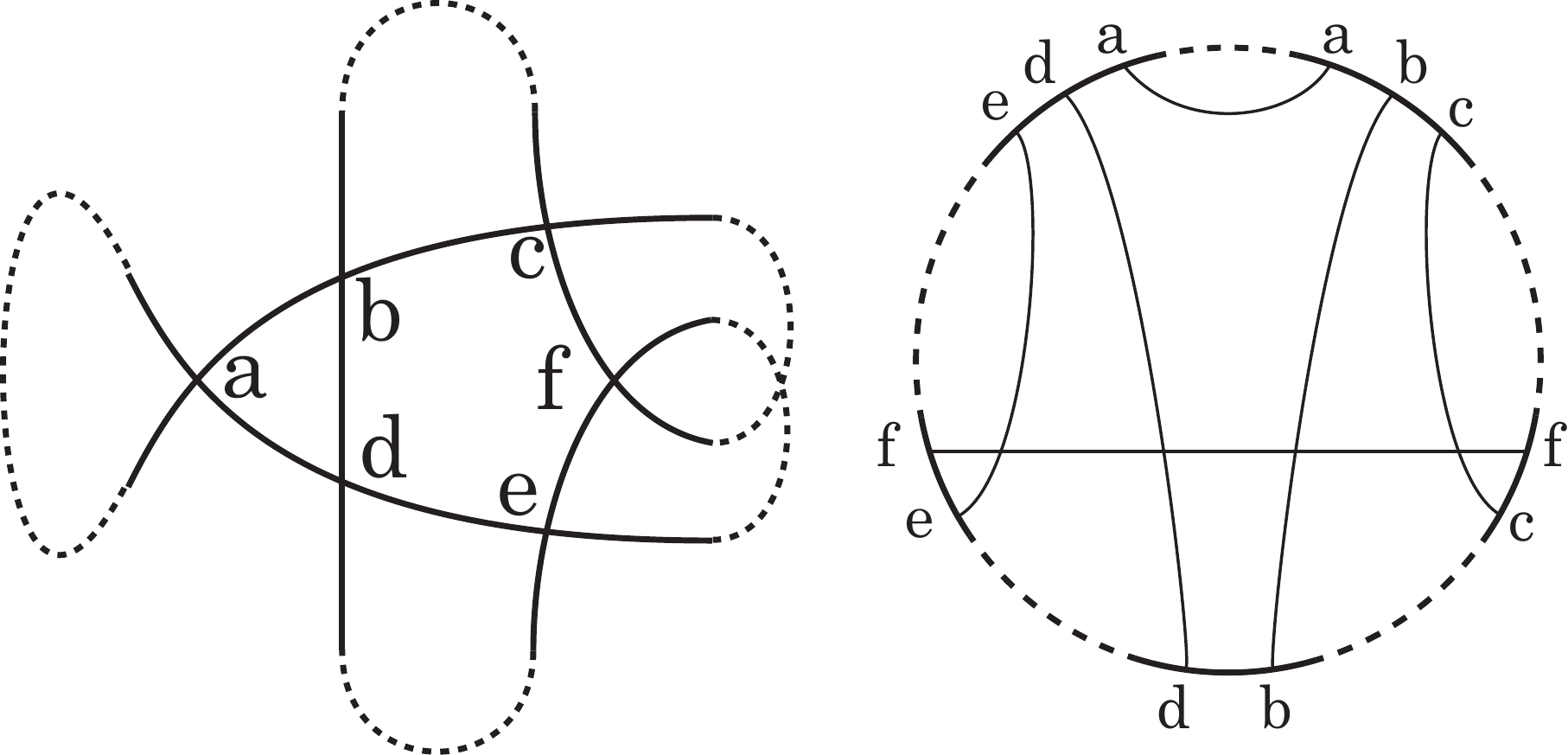} & \includegraphics[width=5.4cm]{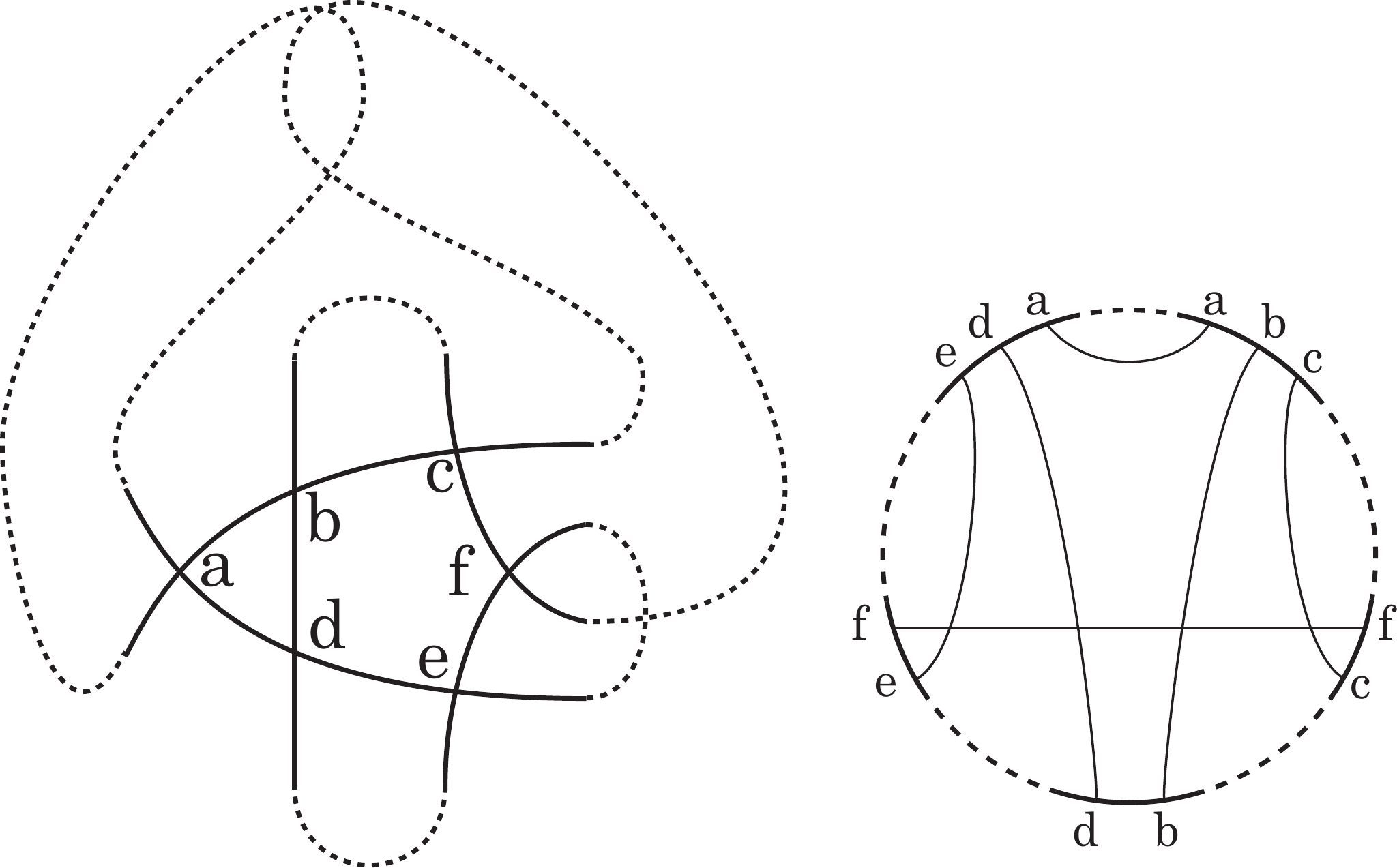} \\ \hline
Case 28 & Case 28a \\ \hline
\includegraphics[width=5.4cm]{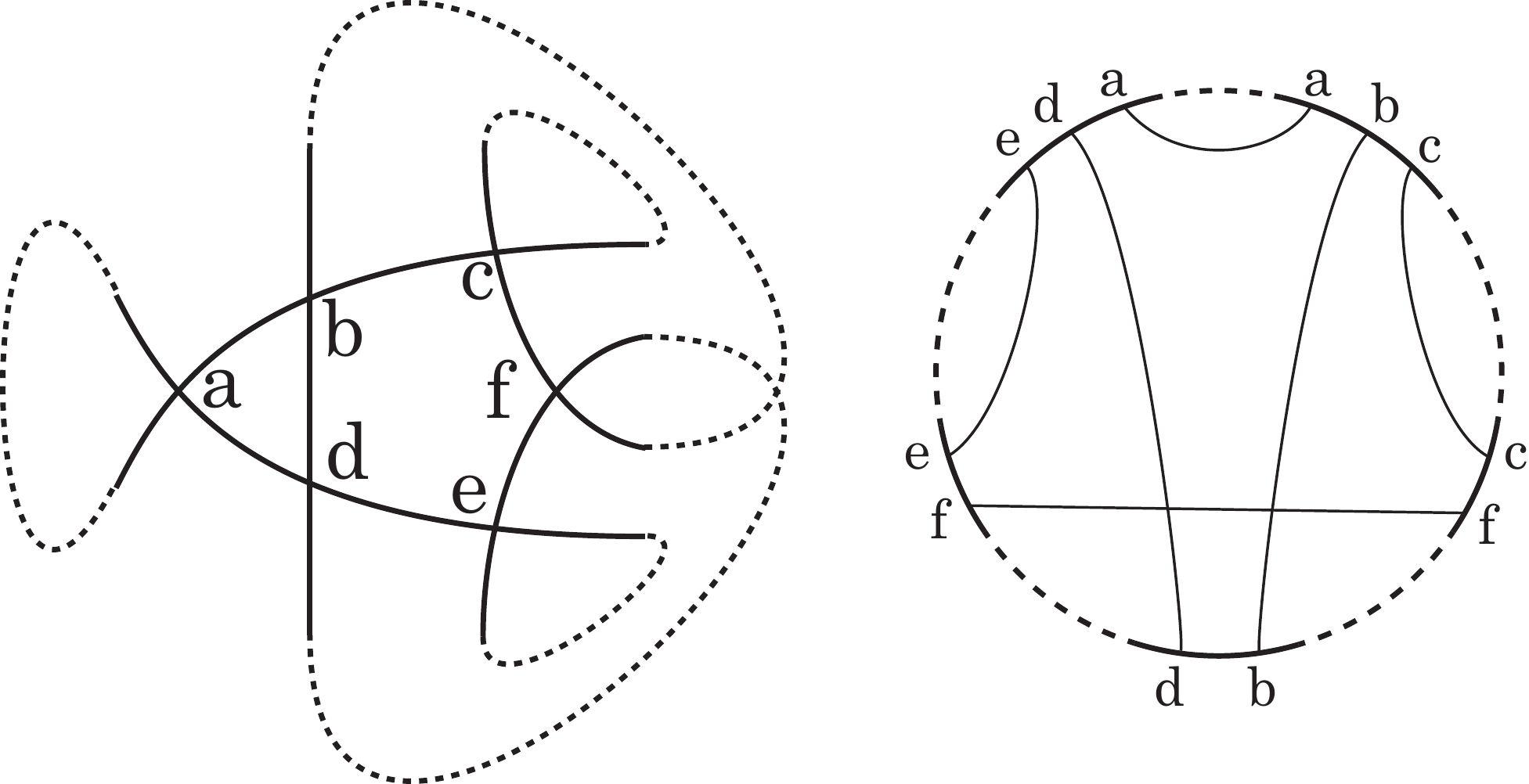} & \includegraphics[width=5.4cm]{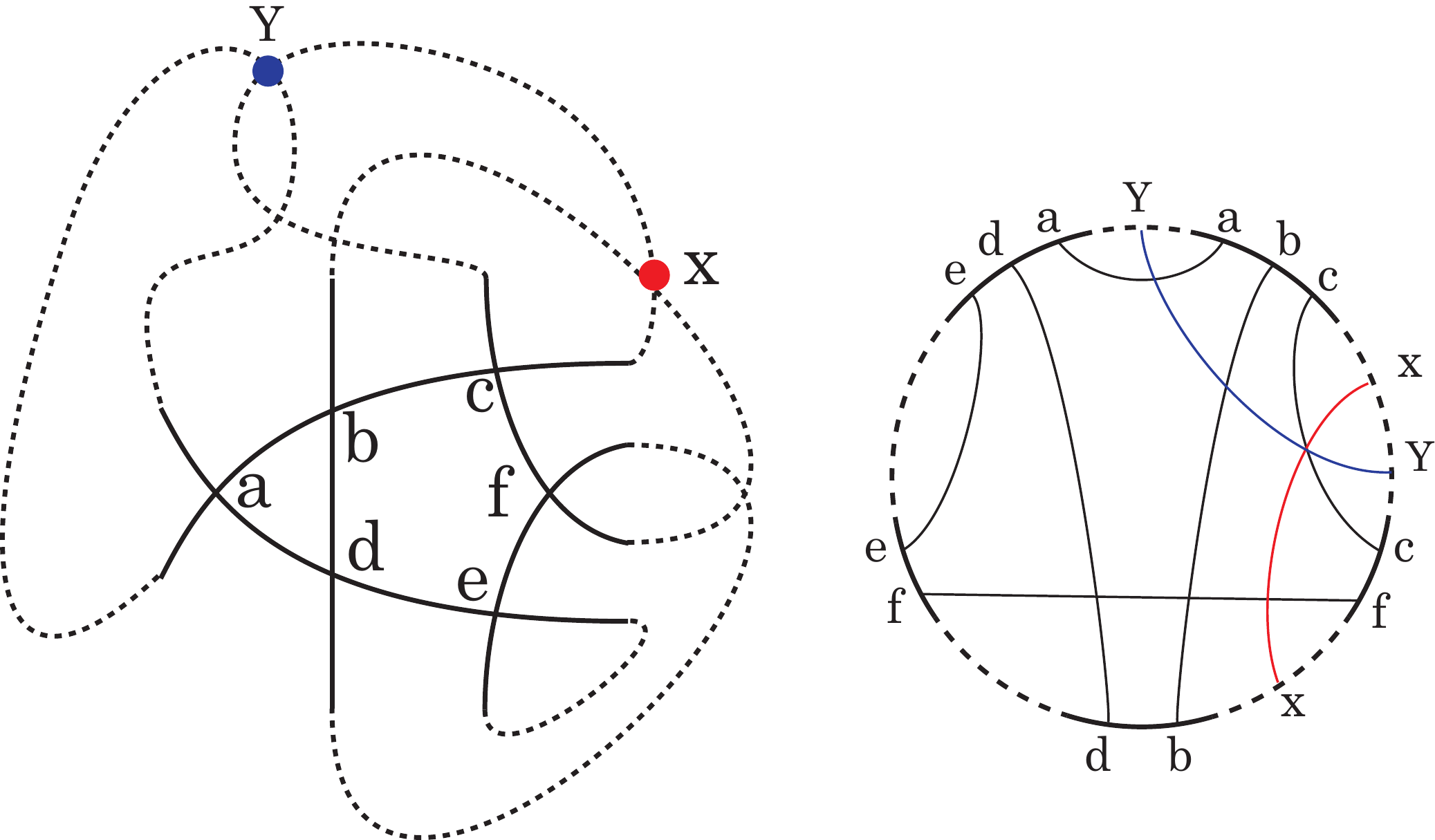} \\ \hline 
\end{tabular}
\end{table}

$\bullet$ {\bf{Case 28.}} See the bottom line of Fig.~\ref{cases25--32}.  
By the assumption, the considered knot projection $P$ is a prime knot projection with no $1$- or $2$-gons.  Thus, $P$ is reduced (Lemma \ref{lem_reduced}), and the dotted arc (a$\sim$a) intersects the other dotted arcs.  
\begin{itemize}
\item If (a$\sim$a) intersects (b$\sim$f), then $P$ has a triple chord in $CD_P$.
\item If (a$\sim$a) intersects (d$\sim$f), then $P$ has a triple chord in $CD_P$.  
\item If (a$\sim$a) intersects (c$\sim$c), but not (b$\sim$f) or (d$\sim$f), then $P$ has a triple chord in $CD_P$, as shown in the bottom line of Fig.~\ref{cases25--32}.  
\item If (a$\sim$a) intersects (e$\sim$e), but not (b$\sim$f) or (d$\sim$f), then $P$ has a triple chord in $CD_P$ by the same reasoning as that of (c$\sim$c) via their symmetry between (c$\sim$c) and (e$\sim$e).  
\end{itemize} 
These 32 cases complete the proof of Theorem \ref{main3}.  
\end{proof}

\section{Reductivity and Triple chords}\label{sec4.1}
This section mentions a relation between the reductivity of a knot projection and triple chords.  The reductivity is introduced by A. Shimizu \cite[Sec.~1]{shimizu2014} using local replacement $A^{-1}$ (also called $A^{-1}$ move in this paper) that appears in \cite{IS}.  
\begin{definition}
The local replacement $A^{-1}$ move at a double point is defined by Fig.~\ref{ych14}.  
\begin{figure}[h!]
\includegraphics[width=5cm]{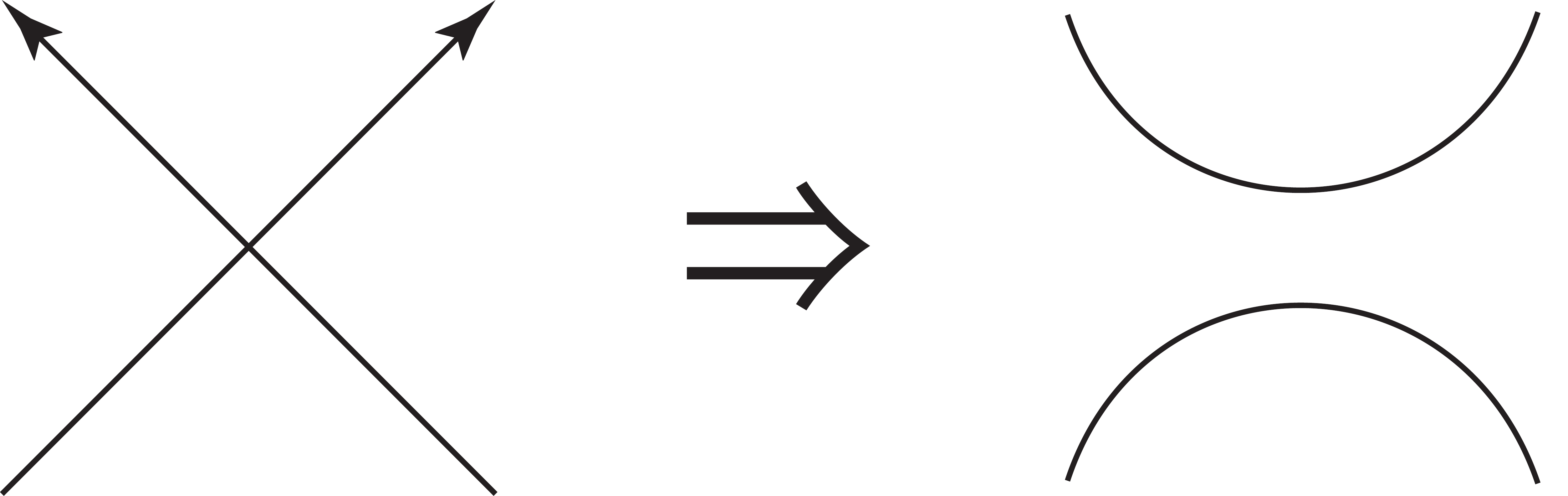}
\caption{$A^{-1}$ move.}\label{ych14}
\end{figure}
The reductivity $r(P)$ of a knot projection is the minimal number of $A^{-1}$ moves to produce a reducible knot projection.  
\end{definition}
\begin{remark}
It is worthwhile mentioning the following fact here.  Any reduced knot projections are related by a finite sequence consisting of $A^{-1}$ moves and inverse moves, where every knot projection appearing in each step in the sequence is reduced \cite[Corollary 1.2]{IS}.  Therefore, it is natural to consider the notion of reductivity \cite{shimizu2014}.  
\end{remark}
We characterize knot projections with $r(P)=1$.  
\begin{theorem}\label{kiyaku_thm}
For a reduced knot projection $P$, there exists a circle with two double points as shown in Fig.~\ref{ych15} if and only if $r(P)=1$.
\begin{figure}[h!]
\includegraphics[width=3cm]{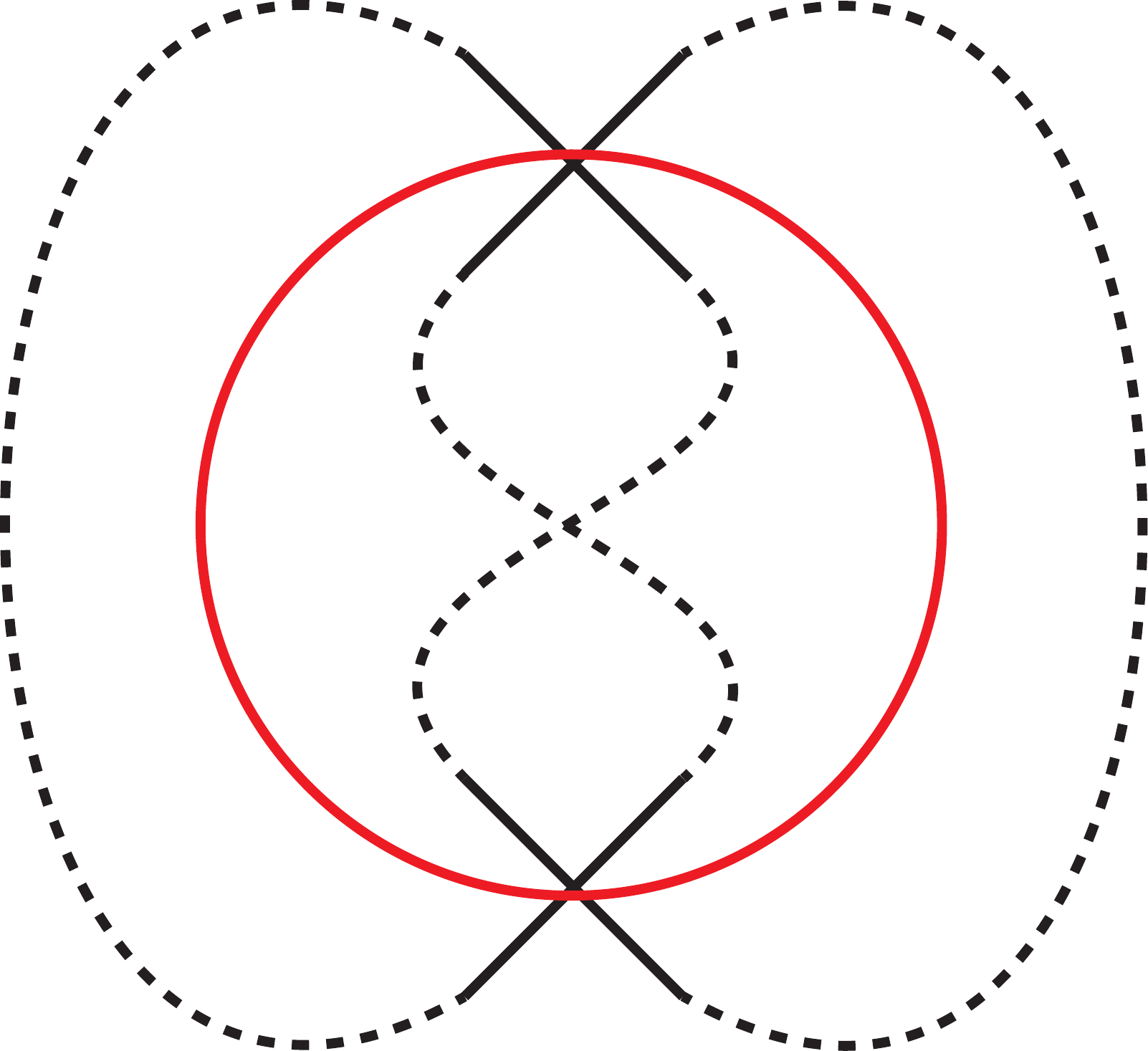}
\caption{Circle with two double points splitting $S^2$ into two disks.  Dotted arcs indicate the connections of branches.}\label{ych15}
\end{figure}
\end{theorem}
\begin{proof}
\begin{itemize}
\item ({\it{If part}}) Assume that $r(P)=1$.  Let $P'$ be a reduced knot projection obtained from $P$ by applying an $A^{-1}$ move at a double point, say $a$, of $P$, and $b$ a reducible crossing of $P'$.  Then it follows by definition that there exists a simple circle which intersects $P$ with $a$ and $b$ only.  There are four cases with respect to the connectivity among the paths at $a$ and $b$ as shown in Fig.~\ref{kiyaku2a}.  Since $P$ is an immersion of a single circle, we have the second and third cases.  
\begin{figure}[htbp]
\includegraphics[width=8cm]{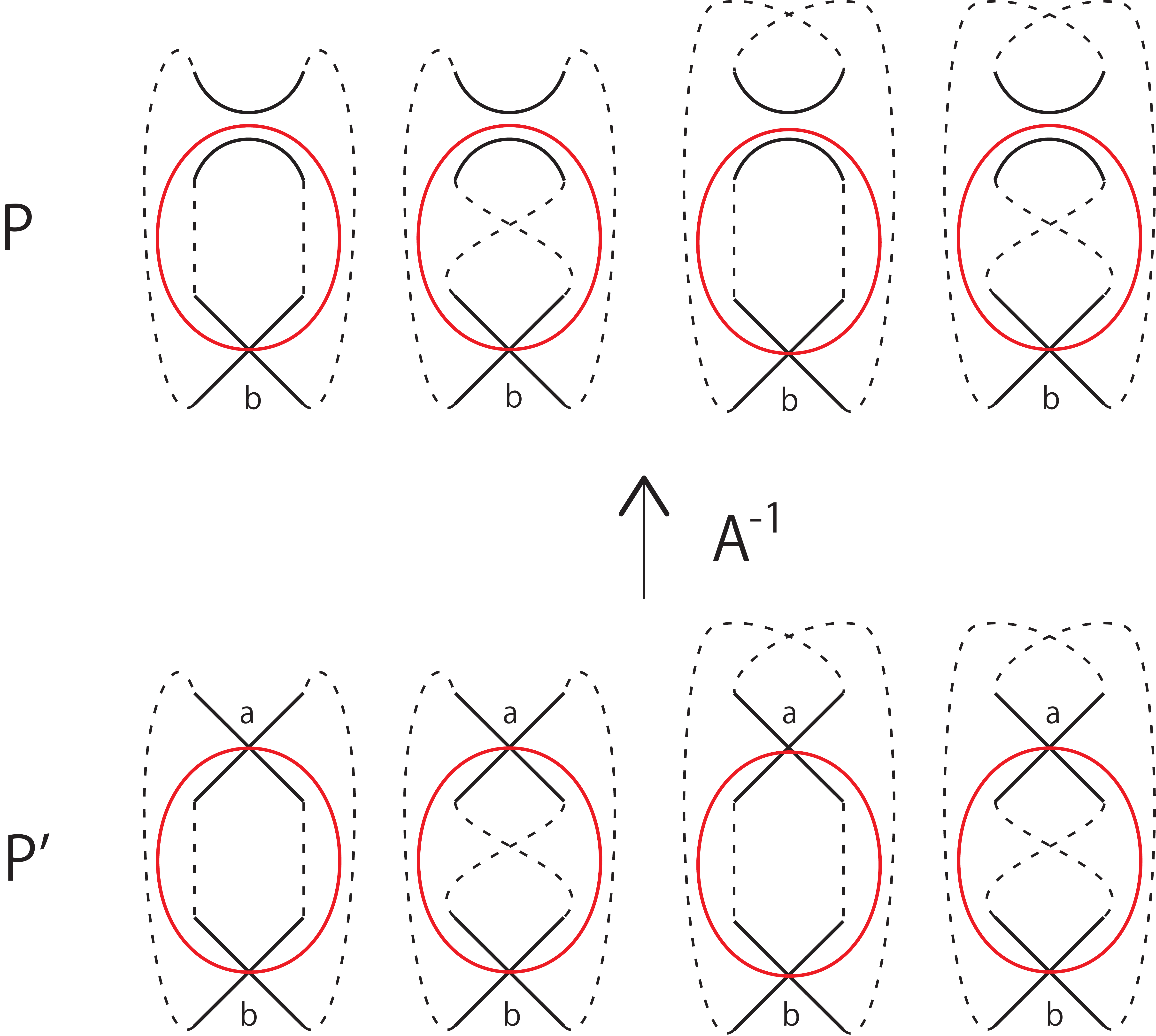}
\caption{All the possibilities of the distributions of arcs and the circle splitting $S^2$ into two disks after applying one $A^{-1}$ move (upper) and all the possibilities before applying $A^{-1}$ (lower).}\label{kiyaku2a}
\end{figure}
\item ({\it{Only if part}}) 
Assume that a reduced knot projection $P$ has two double points as shown in Fig.~\ref{ych15}.  By applying an $A^{-1}$ move at one of the double points, we have a reducible knot projection.  
\end{itemize}
\end{proof}
\begin{corollary}\label{cor1}
A knot projection $P$ with $r(P)=1$ has at least one triple chord in $CD_P$.  
\end{corollary}
\begin{proof}
If $P$ satisfies $r(P)=1$, then it has two double points, say $a$ and $b$, as shown in Fig.~\ref{ych15}.  Let $x$ be a double point of $P$ in the region surrounded by the red circle.  Then the double points $a$, $b$, and $x$ gives a triple chord in $CD_P$.  
\end{proof}

\noindent{\bf{Acknowledgments.}}  The authors would like to thank Professor Kouki Taniyama for his helpful comments.


\begin{thebibliography}{99}
\bibitem{EHK} S. Eliahou, F. Harary, L. H. Kauffman, Lune-free knot graphs, \emph{J. Knot Theory Ramifications} {\bf{17}} (2008), 55--74.  
\bibitem{IS} N. Ito and A. Shimizu, The Half-twisted splice operation on reduced knot projections, J. Knot Theory Ramifications {\bf{21}} (2012), 1250112, 10pp.  
\bibitem{IT} N. Ito and Y. Takimura, (1, 2) and weak (1, 3) homotopies on knot projections, J. Knot Theory Ramifications {\bf{22}} (2013), 1350085 (14 pages).  
\bibitem{ITT} N. Ito, Y. Takimura, and K. Taniyama, Strong and weak (1, 3) homotopies on knot projections, Osaka J. Math., to appear.  
\bibitem{khovanov} M. Khovanov, Doodle groups, \emph{Trans. Amer. Math. Soc.} {\bf{349}} (1997), 2297--2315.    
\bibitem{ST} M. Sakamoto and K. Taniyama, Plane curves in an immersed graph in $\Bbb R\sp 2$, \emph{J. Knot Theory Ramifications} {\bf{22}} (2013), 1350003, 10pp.  
\bibitem{shimizu2014} A. Shimizu, The reductivity of spherical curves, to appear in \emph{Topology Appl}.  
\end{thebibliography}
\end{document}